\documentclass[letterpaper,10pt]{article}%
\usepackage{etex}
\usepackage[utf8]{inputenc}%
\usepackage[T1]{fontenc}%
\usepackage{ae,aecompl}%

\usepackage{XCharter}%

\usepackage{textcomp}%
\usepackage{mathrsfs}%
\usepackage[scaled=.92]{helvet}%
\usepackage{eulervm}%
\usepackage{euscript}%
\usepackage[margin=20pt,labelfont=bf,labelsep=period,format=plain]{caption}%
\usepackage{cmll}%
\usepackage{xparse}%
\usepackage{marvosym}%
\usepackage{calculator}%
\usepackage[UKenglish]{isodate}%

\DeclareFontFamily{T1}{qzc}{}
\DeclareFontShape{T1}{qzc}{m}{it}{<-> s * [1.2] ec-qzcmi}{}
\DeclareMathAlphabet{\mcal}{\encodingdefault}{qzc}{m}{it}

\usepackage{mathtools}%
\usepackage{amsmath}%
\usepackage{amssymb}%
\usepackage{wasysym}%

\DeclareFontFamily{U}{mathx}{\hyphenchar\font45}
\DeclareFontShape{U}{mathx}{m}{n}{
      <5> <6> <7> <8> <9> <10>
      <10.95> <12> <14.4> <17.28> <20.74> <24.88>
      mathx10
      }{}
\DeclareSymbolFont{mathx}{U}{mathx}{m}{n}
\DeclareFontSubstitution{U}{mathx}{m}{n}
\DeclareMathAccent{\widecheck}{0}{mathx}{"71}
\DeclareMathAccent{\wideparen}{0}{mathx}{"75}

\usepackage{amsthm}

\makeatletter
\def\@endtheorem{\endtrivlist}
\makeatother

\newtheoremstyle{mythmstyle}
  {\topsep}%
  {\topsep}%
  {\slshape}%
  {0pt}%
  {\scshape}%
  {.}%
  { }%
  {\thmname{#1}\thmnumber{ #2}\thmnote{ (\normalfont #3)}}%
\newtheoremstyle{mythmlikestyle}
  {\topsep}%
  {\topsep}%
  {\normalfont}%
  {0pt}%
  {\scshape}%
  {.}%
  { }%
  {\thmname{#1}\thmnumber{ #2}\thmnote{ (\normalfont #3)}}%
\theoremstyle{mythmstyle}
\newtheorem{theorem}{Theorem}[section]%
\theoremstyle{mythmlikestyle}

\newtheorem{definition}[theorem]{Definition}
\newtheorem{lemma}[theorem]{Lemma}

\newenvironment{proofof}[1]{\begin{trivlist}\item{}\normalfont\textit{Proof of #1.}}{\hfill$\square$\end{trivlist}}

\usepackage[scale={0.69,0.85},centering,nohead]{geometry}
\usepackage{enumitem}%
\usepackage{booktabs}%

\usepackage[usenames,dvipsnames]{xcolor}%
\colorlet{dblue}{black!15!blue}
\colorlet{dred}{black!15!red}
\colorlet{dgreen}{black!45!green}
\colorlet{dorange}{black!60!orange}
\colorlet{dpurple}{blue!45!red}
\colorlet{dbrown}{black!35!brown}
\colorlet{bindingcolour}{dgreen}

\usepackage{xspace}%
\RequirePackage[pdfencoding=auto,unicode=true,linktocpage]{hyperref}

\usepackage{rotating}%
\usepackage{floatpag}%
\rotfloatpagestyle{empty}%
\floatpagestyle{empty}%

\usepackage[usestackEOL]{stackengine}%
\usepackage{scalerel}%

\usepackage[svgnames]{pstricks}

\usepackage{pst-node,pst-tree,
pstricks-add}%
\selectcolormodel{natural}%

\psset{nodesep=3pt,linewidth=.5pt,
	arrowsize=4pt 1,
	arrowlength=.9,
	arrowinset=.4
}

\usepackage[page]{appendix}%
\let\appendixpagenameorig\appendixpagename
\renewcommand{\appendixpagename}{\Large\appendixpagenameorig}

\newcommand\comsymbol{\MVComma}
\newcommand\com{\mbox{\kern-1.5pt\comsymbol\kern-1.1pt}}
\newcommand\comm{\mkern1mu\com\mkern5mu}

\usepackage{prftree}%
\newcommand{\heightdepth}[3]{\raisebox{0ex}[#1ex][#2ex]{$#3$}}
\newcommand{\prfseq}[1]{\heightdepth{2}{1.05}{#1}}

\newcommand\unaryruleright[3]{\prftree[r]{\,$#1$}{#2}{\prfseq{#3}}}

\newcommand\binaryruleright[4]{\prftree[r]{\,$#1$}{#2}{#3}{\prfseq{#4}}}

\newcommand\rulename[1]{\textsf{#1}}
\newcommand\rulelabel[1]{\raisebox{.27ex}{\small{\rulename{#1}}}}

\newcommand\axiomrulelabel{}%
\newcommand\axiomruleright[1]{\unaryruleright{\axiomrulelabel}{\prfseq{}}{#1}}

\newcommand\verboseaxiomruleright[1]{\unaryruleright{\axiomrulelabel}{\prfseq{}}{#1}}

\newcommand\forallrulelabel{\rulelabel{for all}}
\newcommand\forallruleright[2]{\unaryruleright{\forallrulelabel}{#1}{#2}}

\newcommand\existsrulelabel{\rulelabel{there exists}}
\newcommand\existsruleright[2]{\unaryruleright{\existsrulelabel}{#1}{#2}}

\newcommand\impliesrulelabel{\rulelabel{implies}}
\newcommand\impliesruleright[2]{\unaryruleright{\impliesrulelabel}{#1}{#2}}

\newcommand\contractionrulelabel{\rulelabel{contract}}
\newcommand\contractionruleright[2]{\unaryruleright{\contractionrulelabel}{#1}{#2}}

\newcommand\weakeningrulelabel{\rulelabel{weaken}}
\newcommand\weakeningruleright[2]{\unaryruleright{\weakeningrulelabel}{#1}{#2}}

\newcommand\xname{\rulename{X}}
\newcommand\cname{\rulename{C}}
\newcommand\wname{\rulename{W}}

\newcommand\foneonerule{\axiomruleright{\mkern2mu1\mkern2mu}}

\newcommand\Rxrulehyp{\sequent\com\formula\com\formulaa\com\sequenta}
\newcommand\Rxruleconc{\sequent\com\formulaa\com\formula\com\sequenta}
\newcommand\Rcrulehyp{\sequentcom\,\formula\com\formulap}
\newcommand\Rcruleconc{\sequentcom\formula}
\newcommand\Rwrulehyp{\sequent}
\newcommand\Rwruleconc{\sequentcom\,\formula}
\newcommand\Rorrulehyp{\prfseq{\sequentcom \, \formula\com\,\formulaa}}
\newcommand\Rorruleconc{\sequentcom \, \formula\tightvee \formulaa}
\newcommand\Rexistsrulehyp{\prfseq{\sequentcom \, \formula\assignment{\assign\variable\term}}}
\newcommand\Rexistsruleconc{\sequentcom \, \ex x \formula}
\newcommand\Rforallrulehyp{\prfseq{\sequentcom \, \formula}}
\newcommand\Rforallruleconc{\sequentcom \, \all x \formula}
\newcommand\Randrulehypone{\prfseq{\sequentcom \formula}}
\newcommand\Randrulehyptwo{\prfseq{\sequentacom\formulaa}}
\newcommand\Randruleconc{\sequentcom \, \sequentacom \, \formula \tightwedge \formulaa}
\newcommand\Randrulehypindex[1]{\prfseq{\sequent_{#1}\com \formula_{#1}}}
\newcommand\Randruleconcindexed{\sequent_1\com \, \sequent_2\com \, \formula_1 \tightwedge \formula_2}

\newcommand\Raxiomrule[1]{\axiomruleright{#1\com\dual{#1}}}
\newcommand\Ronerule{\foneonerule}
\newcommand\Rxrule[2]{\unaryruleright{\rulelabel{\xname}}{#1}{#2}}
\newcommand\Rcrule[2]{\unaryruleright{\rulelabel{\cname}}{#1}{#2}}
\newcommand\Rwrule[2]{\unaryruleright{\rulelabel{\wname}}{#1}{#2}}
\newcommand\Randrule[3]{\binaryruleright{\wedge}{#1}{#2}{#3}}
\newcommand\Rorrule[2]{\unaryruleright{\vee}{#1}{#2}}
\newcommand\Rforallrule[2]{\unaryruleright{\forall}{#1}{#2}}
\newcommand\Rexistsrule[2]{\unaryruleright{\exists}{#1}{#2}}

\newcommand\latinstyle[1]{\emph{#1}}%
\makeatletter
\DeclareRobustCommand\onedot{\futurelet\@let@token\@onedot}
\def\@onedot{\ifx\@let@token.\else.\null\fi\xspace}
\def\eg{\latinstyle{e.g}\onedot} 
\def\ie{\latinstyle{i.e}\onedot} 
\def\cf{\latinstyle{c.f}\onedot}

\makeatother

\newcommand\WLOG{Without loss of generality\xspace}
\newcommand\recombulas{rectified fographs\xspace}%

\newcommand\defn[1]{{\textit{\textbf{#1}}}}

\DeclareSymbolFont{oldsymbols}{OMS}{cmsy}{m}{n}
\DeclareMathSymbol\wedge\mathbin{oldsymbols}{"5E}%
\DeclareMathSymbol\vee\mathbin{oldsymbols}{"5F}
\DeclareMathSymbol\neg\mathord{oldsymbols}{"3A}

\def\myoverline#1{{\ThisStyle{{%
  \setbox0=\hbox{{$\SavedStyle#1$}}%
  \stackengine{{1\LMpt}}{{$\SavedStyle#1$}}{{\rule{\wd0}{.4\LMpt}}}{O}{c}{F}{F}{S}%
}}}}
\newcommand{\dual}[1]{{\mkern1mu\myoverline{\mkern-1mu#1\mkern-1.5mu}\mkern1.5mu}}%

\newcommand\todo[1]{}%
\newcommand\note[1]{}%

\renewcommand\v[1]{\vspace*{#1ex}}

\newcommand\hh[1]{\hspace*{-#1ex}}

\newcommand\hyp{\mathcal{H}}
\newcommand\conc{\mathcal{C}}

\newcommand\graph{G}
\newcommand\fograph{G}
\newcommand\fographp{G\mkern-1mu'}
\newcommand\fographa{H}
\newcommand\fographap{H\primed}
\newcommand\fographaa{K}
\newcommand\fographaaa{L}
\newcommand\net{N}
\newcommand\netp{\net\primed}
\newcommand\fographnet{\net}
\newcommand\graphp{{\graph\mkern-2.3mu'}}
\newcommand\grapha{H}
\newcommand\graphaa{K}
\newcommand\vertices{V}
\newcommand\verticesp{{\vertices\mkern-1mu'}}
\newcommand\vertex{v}
\newcommand\vertexp{{\vertex\mkern-1.3mu'}}
\newcommand\vertexa{w}
\newcommand\vertexaa{u}
\newcommand\vertexap{w\primed}
\newcommand\edges{E}
\newcommand\edgesp{\edges\mkern-2.5mu'}

\newcommand\guniongp{\graph\graphunion\graphp}
\newcommand\gjoingp{\graph\graphjoin\graphp}

\newcommand\cograph{G}

\newcommand\p{p}\newcommand\pp{\dual p}\newcommand\ppp{\dual\pp}
\newcommand\q{q}\newcommand\qq{\dual q}
\newcommand\px{\p x}
\newcommand\ppx{\pp x}
\newcommand\ppy{\pp y}
\newcommand\ppa{\pp\mkern-1.5mu a}
\newcommand\pa{\p\mkern-1.4mu a}

\newcommand\qab{\q\mkern-1.4mu a\mkern-1.4mu b}
\newcommand\qba{\q\mkern-1.4mu b\mkern-1.4mu a}
\newcommand\qqab{\qq\mkern-1.4mu a\mkern-1.4mu b}
\newcommand\qqba{\qq\mkern-1.4mu b\mkern-1.4mu a}

\newcommand\qx{\q x}

\newcommand\py{\p y}
\newcommand\pfy{\p\mkern-1mu f y}
\newcommand\pffy{\p\mkern-1mu  f\mkern-1mu \fy}
\newcommand\ppffy{\pp\mkern-1mu  f\mkern-1mu \fy}
\newcommand\qy{\q y}
\newcommand\qqy{\qq y}
\newcommand\pxy{\p x y}

\newcommand\qxy{\q\mkern-.5mu x y}

\newcommand\pz{\p z}
\newcommand\ppz{\pp z}

\newcommand\fy{fy}
\newcommand\fz{fz}
\newcommand\rx{rx}
\newcommand\rz{rz}

\newcommand\qfz{\q \mkern-1mu fz}

\newcommand\pfx{\p \mkern-1mu fx}
\newcommand\qqfz{\qq \mkern-1mu fz}

\newcommand\all[1]{\forall #1\mkern2mu}
\newcommand\ex[1]{\exists #1\mkern2mu}
\newcommand\allx{\forall x}
\newcommand\ally{\forall y}
\newcommand\existsx{\exists x}

\newcommand\axpx{\forall x\mkern2.2mu\px}

\newcommand\formula{\varphi}%
\newcommand\monadicformula{\formula}
\newcommand\monadicformulap{\formulap}

\newcommand\modalformula{\mu}
\newcommand\modalformulap{\mu\primed}

\newcommand\formulaa{\theta}%
\newcommand\formulaaa{\psi}%
\newcommand\formulap{\formula\mkern-.5mu\primed}

\newcommand\gformula{\graphof{\mkern-1mu\formula\mkern-1mu}}
\newcommand\gformulaa{\graphof{\mkern-1mu\formulaa\mkern-1mu}}

\newcommand\atomone{\alpha_1}%
\newcommand\atomtwo{\alpha_2}%
\newcommand\variable{x}

\newcommand\zone{z_{\mkern-1.5mu 1}}
\newcommand\ztwo{z_{\mkern-.5mu 2}}

\renewcommand\tag{\lambda}

\newcommand\binder{b}
\newcommand\bindera{c}

\newcommand\binderp{\binder\primed}
\newcommand\atom{\alpha}
\newcommand\dualatom{\dual\atom}

\newcommand\singletonx{\singleton x}
\newcommand\singletony{\singleton y}
\newcommand\singletonz{\singleton z}

\newcommand\singletonq{\singleton \q}
\newcommand\singletonqq{\singleton \qq}
\newcommand\singletonpx{\singleton \px}
\newcommand\singletonpy{\singleton \py}
\newcommand\singletonppy{\singleton \ppy}
\newcommand\singletonqx{\singleton \qx}

\newcommand\singletonppx{\singleton \ppx}

\newcommand\cfa{G}

\newcommand\homom{h}

\newcommand\homomof[1]{\homom(\mkern-1mu#1\mkern-1mu)}

\newcommand\map{f}
\newcommand\mapp{f\primed}

\newcommand\graphpairof[2]{(#1,#2)}
\newcommand\digraphpairof[2]{\graphpairof{#1}{#2}}
\newcommand\graphpair{\graphpairof\vertices\edges}
\newcommand\graphpairp{\graphpairof\verticesp\edgesp}
\newcommand\digraphpair{\graphpairof\vertices\edges}
\newcommand\digraphpairp{\graphpairof\verticesp\edgesp}
\newcommand\diedge[2]{\langle #1,#2\rangle}
\newcommand\graphpairofgraph[1]{\graphpairof{\verticesof{#1}}{\edgesof{#1}}}

\newcommand\drinkergraph{D}

\newcommand\xp{x\primed}

\newcommand\fib{f}
\newcommand\bifib{f}
\newcommand\bifibp{f\primed}
\newcommand\bifiba{g}

\newcommand\cover{K}
\newcommand\coverp{K\mkern-1.5mu\primed}
\newcommand\base{G}
\newcommand\cp{\skewfib:\cover\to\gformula}

\renewcommand{\implies}{\Rightarrow}

\newcommand\diredgestyle[1]{\psset{arrows=#1,nodesep=2pt,arrowsize=1.8pt 2,arrowinset=.2,arrowlength=1.2,linewidth=.6pt}}
\newcommand\diredgesymb[1]{\raisebox{0pt}{\!\kern-.5pt\begin{psmatrix}[colsep=2.7ex]\rnode{l}{\rule{0pt}{1ex}}&\rnode{r}{\rule{0pt}{1ex}}%
\diredgestyle{#1}\ncline{l}{r}\end{psmatrix}\!\kern-.5pt}}
\newcommand\widediredgesymb[2]{\raisebox{0pt}{\!\kern-.5pt\begin{psmatrix}[colsep=#2]\rnode{l}{\rule{0pt}{1ex}}&\rnode{r}{\rule{0pt}{1ex}}%
\diredgestyle{#1}\ncline{l}{r}\end{psmatrix}\!\kern-.5pt}}
\def\bindrel{\mathrel{\scalerel*{\diredgesymb{->}}{\rule{0ex}{1ex}}}}%

\newcommand\widebindrel[1]{\mathrel{\scalerel*{\widediredgesymb{->}{#1}}{\rule{0ex}{1ex}}}}

\newcommand\graphofsymbol{\mcal{G}}

\newcommand\formulaofsymbol{\Phi}
\newcommand\cotreeofsymbol{\textsf{T}}%

\newcommand\graphof[1]{\graphofsymbol(#1)}

\newcommand\formulaof[1]{\formulaofsymbol(#1)}

\newcommand\cotreeof[1]{\cotreeofsymbol(#1)}

\newcommand\verticesofsymbol{V}
\newcommand\verticesof[1]{\verticesofsymbol_{#1}}
\newcommand\edgesof[1]{E_{#1}}

\newcommand\graphunion{\tightplus}
\newcommand\graphjoin{\tighttimes}
\newcommand\tightgraphunion{\mkern-2mu\graphunion\mkern-2mu}
\newcommand\tightgraphjoin{\mkern-2mu\graphjoin\mkern-2mu}

\newcommand\inlinegreenvx[1]{\mkern5mu\greenvx{0,0.08}{#1}\mkern5mu}%
\newcommand\inlineredvx[1]{\mkern5mu\redvx{0,0.08}{#1}\mkern5mu}%
\newcommand\singletonblue[1]{\mkern5mu\bluevx{0,.08}{v}\mkern6mu#1}
\newcommand\singletonred[1]{\mkern5mu\redvx{0,.08}{v}\mkern6mu#1}
\newcommand\singletonredleft[1]{#1\mkern6mu\redvx{0,.08}{v}\mkern5mu}

\newcommand\singletonleft[1]{{#1\mkern.8mu\bullet}}
\newcommand\singletonright[1]{{\bullet\mkern.8mu#1}}
\newcommand\singleton[1]{\singletonright{#1}}
\newcommand\singletonbuffer{\hspace*{1pt}\hspace*{\vertexrad}}
\newcommand\namedvx[1]{\singletonbuffer\cnode*(0,.5ex){\vertexrad}{#1}\singletonbuffer}
\newcommand\namedsingletonleft[2]{{#2}\namedvx{#1}}
\newcommand\namedsingletonright[2]{\namedvx{#1}{#2}}

\newcommand\skewfib{f}

\newcommand\skewfibnolabels{f\primed}
\newcommand\netnolabels{\net\primed}
\newcommand\basenolabels{\base\primed}

\newcommand\graphall[2]{\singleton{#1}\graphunion\graphof{#2}}
\newcommand\graphex [2]{\singleton{#1}\graphjoin\graphof{#2}}

\newcommand\allypy{\forall y \mkern3mu \py}
\newcommand\predrinkerformula{\py\mkern-1mu\implies\mkern-1mu \allypy}
\newcommand\drinkerformula{\exists x\mkern1mu(\mkern1mu\px\mkern-1mu\implies\mkern-1mu \allypy)}
\newcommand\drinkerformulanamed[2]{\exists #1\mkern1mu(\mkern1mu\p #1\mkern-1mu\implies\mkern-1mu \forall #2\mkern3mu p #2)}
\newcommand\drinkerformulamodallike{\drinkerformulanamed x x}
\newcommand\veedrinkerformula{\exists x\mkern1mu(\mkern1mu\ppx\mkern-1mu\vee\mkern-1mu \allypy)}
\newcommand\variantveedrinkerformula{\exists x\mkern2mu\forall y\mkern2mu(\py\mkern-1mu\vee\mkern-1mu\ppx)}

\newcommand\fibrewidth{1pt}
\newcommand\fibredotsep{3.2pt}
\newcommand\fibrenodesep{2pt}
\newcommand\fibrestyle{\psset{linestyle=dotted,linewidth=\fibrewidth,dotsep=\fibredotsep,nodesep=\fibrenodesep}}
\newcommand\semicombinatorialfibrestyle{\psset{linestyle=solid,linewidth=.5pt,nodesep=3pt}}
\newcommand\dualitystyle{\psset{linecolor=black!30!orange,linestyle=dashed,linewidth=.7pt,dash=4pt 2pt,nodesep=1pt,,dash=3pt 1.5pt}}
\newcommand\bindingstyle{\psset{linecolor=bindingcolour,arrows=->,dash=3pt 1.5pt,nodesep=2pt}}

\newenvironment{pic}[2]{\begin{center}\begin{pspicture}(0,#1)(0,#2)\(}{\)\end{pspicture}\end{center}}

\newcommand\drinkerFormulaDisplayed{\begin{pic}{-1}{.4} \rput(0,0){\drinkerbase} \end{pic}}
\newcommand\drinkerDisplayed{\begin{pic}{-.9}{2.9} \rput(0,0){\drinkerfibcoloured} \end{pic}}
\newcommand\drinkerSemicombinatorialTarget{
  \exists
  \Rnode x x
  \mkern2mu
  (
  \mkern2mu
  \Rnode{pp}{\pp}
  x
  \mkern-2mu
  \vee
  \forall
  \Rnode y y
  \mkern2mu
  \Rnode p p
  y
  \mkern2mu
  )
}
\newcommand\drinkerSemicombinatorialSourceVerbose{
  (
  \mkern2mu
  \exists
  \Rnode{x1} x
  \mkern2mu
  \Rnode{pp1}{\pp}
  x
  \mkern2mu
  )
  \vee
  (
  \exists
  \Rnode{xx1}x
  \mkern2mu
  \forall
  \Rnode{y1}y
  \mkern2mu
  \Rnode{p1}{p}
  y
  \mkern2mu
  )
}
\newcommand\drinkerSemicombinatorialSource{
  \exists
  \Rnode{x1} x
  \mkern2mu
  \Rnode{pp1}{\pp}
  x
  \hspace{3ex}
  \exists
  \Rnode{xx1}x
  \mkern2mu
  \forall
  \Rnode{y1}y
  \mkern2mu
  \Rnode{p1}{p}
  y
  \mkern2mu
}
\newcommand\drinkerSemicombinatorialSourceVerboseLinked{
\drinkerSemicombinatorialSourceVerbose\ncbar[angleA=90,angleB=90,nodesep=2.5pt,linecolor=dblue,arm=7pt]{pp1}{p1}
}
\newcommand\drinkerSemicombinatorialSourceVerboseInline{
(\exists x\mkern2mu\ppx)\vee(\exists x\mkern2mu\forall y\mkern2mu\px)
}
\newcommand\drinkerSemicombinatorialSourceLinked{
\drinkerSemicombinatorialSource\ncbar[angleA=90,angleB=90,nodesep=2.5pt,linecolor=dblue,arm=7pt]{pp1}{p1}
}
\newcommand\drinkerSemicombinatorialSourceInline{
\exists x\mkern2mu\ppx\com\,\exists x\mkern2mu\forall y\mkern2mu\px
}
\newcommand\semicombinatorialFibeheightBoost{.25}
\newcommand\drinkerSemicombinatorialVerbose{
\rput(0,\fibheight){\rput(0,\semicombinatorialFibeheightBoost)\drinkerSemicombinatorialSourceVerboseLinked}
\rput(0,0){\drinkerSemicombinatorialTarget}
\fibrestyle
\psset{nodesep=3pt}
\ncline{pp1}{pp}
\ncline{p1}p
\ncline{x1}x
\ncline{xx1}x
\ncline{y1}y
}
\newcommand\drinkerSemicombinatorial{
\rput(0,\fibheight){\rput(0,\semicombinatorialFibeheightBoost){\drinkerSemicombinatorialSourceLinked}}
\rput(0,0){\drinkerSemicombinatorialTarget}
\semicombinatorialfibrestyle
\ncline{pp1}{pp}
\ncline{p1}p
}

\newcommand\figdrinkernolabels{\begin{figure*}
\begin{pic}{-.9}{3.3}
\rput(-3.2,0){\drinkerfibcoloured}
\rput(3.2,0){\drinkerfibnolabels}
\end{pic}
\caption{\label{fig:drinker-no-labels}A combinatorial proof $\skewfib:\net\to\base$ of the monadic formula $\drinkerformula$ (left), copied from the Introduction, and its homogeneous combinatorial proof
$\skewfibnolabels:\netnolabels\to\basenolabels$ (right). The
directed edges of $\netnolabels$ and $\basenolabels$ are those of the binding graphs $\protect\bindinggraphof\net$ and $\protect\bindinggraphof\base$,
the dashed edge of $\netnolabels$ captures the colour of $\net$,
and the dashed edge of $\basenolabels$ captures the duality between the two predicate symbols $p$ and $\protect\pp$ in $\base$.%
}\figrule\end{figure*}}

\newlength\itemlen
\newcommand\fibitem[2]{%
  \settowidth\itemlen{\ensuremath #1}
  \rput(0.5\itemlen,.4){#2}#1
}

\newcommand\peirceformula{\left(\rule{0ex}{1.7ex}\mkern-1mu(\pp\mkern-1mu\tightvee\mkern-3mu q)\mkern-2mu\tightwedge\mkern-2mu\pp\right)\mkern-1mu\tightvee\mkern-0mu p}
\newcommand\peirceimpliesformula{\left(\rule{0ex}{1.7ex}\mkern-1mu(p\mkern-1mu\tightimplies\mkern-2mu q)\mkern-1mu\tightimplies p\right)\mkern-2mu\tightimplies\mkern-.3mu p}

\newcommand\drinkerInlineDisplayedPic{\begin{pic}{-.1}{.85}\drinkerInlineDisplayed\end{pic}}
\newcommand\drinkerInlineDisplayed{%
  \def\yu{.19}
  \fibitem{\existsx}{\vx{0,0}{xl}\vx{0,\yu}{xu}}
  \mkern2mu
  (
  \mkern2mu
  \fibitem{\px}{\bluevx{0,0}{px}}
  \mkern-2mu
  \implies
  \mkern-2mu
  \fibitem{\ally}{\vx{0,0}{y}}
  \mkern2mu
  \fibitem{\py}{\bluevx{0,0}{py}}
  \mkern2mu
  )
  \e {xl} {px}
  \psset{nodesep=0}
  \nccurve[angleA=0,angleB=160]{xu}{y}
  \nccurve[angleA=8,angleB=155]{xu}{py}
}

\newcommand\modaldrinkerInlineDisplayed{%
  \def\yu{.19}
  \fibitem{\pos}{\vx{0,0}{xl}\vx{0,\yu}{xu}}
  \mkern2mu
  (
  \mkern2mu
  \fibitem{p}{\bluevx{0,0}{px}}
  \mkern-2mu
  \implies
  \mkern-2mu
  \fibitem{\nec}{\vx{0,0}{y}}
  \mkern2mu
  \fibitem{p}{\bluevx{0,0}{py}}
  \mkern2mu
  )
  \e {xl} {px}
  \psset{nodesep=0}
  \nccurve[angleA=0,angleB=160]{xu}{y}
  \nccurve[angleA=12,angleB=154]{xu}{py}
}

\newcommand\drinkerx{-1.1}
\newcommand\drinkerxx{-.4}%
\newcommand\drinkerxxx{.4}%
\newcommand\drinkerxxxx{1.15}

\newcommand\drinkerdown{.235}
\newcommand\drinkerdowndown{.37}
\newcommand\drinkercovergap{.445}
\newcommand\drinkercoverfirstdown{.3}
\newcommand\drinkercoverseconddown{.21}
\newcommand\coverppxlabelangle{0}
\newcommand\coverppxlabeldist{2pt}
\newcommand\coverylabelangle{-35}
\newcommand\drinkerbasevertices{
  \lvxd{\drinkerx,0}{x}{x}
  \lvxd{\drinkerxx,-\drinkerdowndown}{ppx}{\ppx}%
  \lvxd{\drinkerxxx,-\drinkerdown}{y}{y}%
  \lvxd{\drinkerxxxx,0}{py}{\py}
}
\newcommand\drinkerbaseverticesnolabels{
  \lvxd{\drinkerx,0}{x}{}
  \lvxd{\drinkerxx,-\drinkerdowndown}{ppx}{}
  \lvxd{\drinkerxxx,-\drinkerdown}{y}{}
  \lvxd{\drinkerxxxx,0}{py}{}
}
\newcommand\drinkerbaseedges{
  \e x y
  \e x {py}
  \e x {ppx}
}
\newcommand\drinkerbase{
  \drinkerbasevertices
  \drinkerbaseedges
}
\newcommand\drinkercoververtices{%
  \rput(0,\drinkercovergap){
    \vx{\drinkerx,0}{xa}
    \vx{\drinkerxxx,-\drinkercoverseconddown}{ya}
    \vx{\drinkerxxxx,0}{pya}
  }
  \vx{\drinkerx,0}{xb}
  \vx{\drinkerxx,-\drinkercoverfirstdown}{ppxb}
}
\newcommand\drinkercoververticescoloured{%
  \rput(0,\drinkercovergap){
    \vx{\drinkerx,0}{xa}
    \vx{\drinkerxxx,-\drinkercoverseconddown}{ya}
    \bluevx{\drinkerxxxx,0}{pya}
  }
  \vx{\drinkerx,0}{xb}
  \bluevx{\drinkerxx,-\drinkercoverfirstdown}{ppxb}
}
\newcommand\drinkercoveredges{
  \e{xb}{ppxb}
  \e{xa}{pya}
  \e{xa}{ya}
}

\newcommand\cpfibheight{1.8}

\newcommand\onionxone{.3}
\newcommand\onionxtwo{.85}
\newcommand\onionxthree{1.4}
\newcommand\onionxfour{2.1}
\newcommand\onionxfive{2.8}
\newcommand\onionxsix{3.5}
\newcommand\onioncoverradius{.3}
\newcommand\onionnudge{-.22}%
\newcommand\onioncp{
  \rput(0,\cpfibheight){
    \rput(0,-\onioncoverradius){
      \vx{\onionxone,0}{xa}
      \redvx{\onionxtwo,\onionnudge}{pfxa}
      \greenvx{\onionxthree,0}{ppxa}
    }
    \rput(0,\onioncoverradius){
      \vx{\onionxone,0}{xb}
      \greenvx{\onionxtwo,\onionnudge}{pfxb}
      \bluevx{\onionxthree,0}{ppxb}
    }
    \vx{\onionxfour,0}{yc}
    \redvx{\onionxfive,0}{ppffyc}
    \bluevx{\onionxsix,0}{pyc}
  }
  \lvxd{\onionxone,0} x {\likepbar x\hspace{1ex}}
  \lvxd{\onionxtwo,\onionnudge}{pfx}{\pfx}
  \lvxd{\onionxthree,0}{ppx}{\hspace{1.5ex}\ppx}
  \lvxd{\onionxfour,0} y {\likepbar y}
  \lvxd{\onionxfive,0}{ppffy}{\likepbar\ppffy\hspace{.2ex}}
  \lvxd{\onionxsix,0}{py}{\hspace{.5ex}\likepbar\py}
  \e {xa} {pfxa}
  \e {xa} {ppxa}
  \e {pfxa} {ppxa}
  \e {xb} {pfxb}
  \e {xb} {ppxb}
  \e {pfxb} {ppxb}
  \e {x} {pfx}
  \e {x} {ppx}
  \e {pfx} {ppx}
  \fibrestyle
  \ncline {xb}{xa}
  \ncline {xa}{x}
  \ncline {pfxb}{pfxa}
  \ncline {pfxa}{pfx}
  \ncline {ppxb}{ppxa}
  \ncline {ppxa}{ppx}
  \ncline {yc}{y}
  \ncline {ppffyc}{ppffy}
  \ncline {pyc}{py}
}
\newcommand\onioncpinline{
  \def\yu{.3}
  \left(\strut
  \fibitem{\allx}{\vx{0,0}{xl}\vx{0,\yu}{xu}}
  (
  \fibitem{\pfx}{\redvx{0,0}{pfxl}\greenvx{0,\yu}{pfxu}}
  \mkern-2mu
  \implies
  \mkern-2mu
  \fibitem{\px}{\greenvx{0,0}{pxl}\bluevx{0,\yu}{pxu}}
  )
  \mkern-1mu
  \right)
  \mkern3mu
  \implies
  \mkern3mu
  \fibitem{\ally}{\vx{0,0}{y}}
  \mkern2mu
  (
  \mkern1mu
  \fibitem{\pffy}{\redvx{0,0}{pffy}}
  \mkern-1mu
  \implies
  \mkern-1mu
  \fibitem{\py}{\bluevx{0,0}{py}}
  \mkern1mu
  )
  \e {xl} {pfxl}
  \e {pfxl} {pxl}
  \e {xu} {pfxu}
  \e {pfxu} {pxu}
  \psset{nodesep=0}
  \nccurve[angleA=20,angleB=165]{xl}{pxl}
  \nccurve[angleA=20,angleB=165]{xu}{pxu}
}

\newcommand\pabxone{-1.2}
\newcommand\pabxtwo{-.5}
\newcommand\pabxthree{.2}
\newcommand\pabxfour{.7}
\newcommand\pabxfive{1.2}
\newcommand\pabcoverradius{.3}
\newcommand\pabnudge{-.22}%
\newcommand\pabcp{
  \rput(0,\cpfibheight){
    \bluevx{\pabxone,0}{pabm}
    \redvx{\pabxtwo,0}{pcdm}
    \rput(0,\pabcoverradius){
      \vx{\pabxthree,0}{xu}
      \vx{\pabxfour,\pabnudge}{yu}
      \bluevx{\pabxfive,0}{pxyu}
    }
    \rput(0,-\pabcoverradius){
      \vx{\pabxthree,0}{xl}
      \vx{\pabxfour,\pabnudge}{yl}
      \redvx{\pabxfive,0}{pxyl}
    }
  }
  \lvxd{\pabxone,0}{pab}{\hspace{-1ex}\likepbar\qqab}
  \lvxd{\pabxtwo,0}{pcd}{\likepbar\qqba\hspace{-.5ex}}
  \lvxd{\pabxthree,0}{x}{\likepbar x}
  \lvxd{\pabxfour,\pabnudge}{y}{y}
  \lvxd{\pabxfive,0}{pxy}{\likepbar\qxy\hspace{-1ex}}
  \e {pabm} {pcdm}
  \e {xu} {yu}
  \e {xu} {pxyu}
  \e {yu} {pxyu}
  \e {xl} {yl}
  \e {xl} {pxyl}
  \e {yl} {pxyl}
  \e {pab} {pcd}
  \e {x} {y}
  \e {x} {pxy}
  \e {y} {pxy}
  \fibrestyle
  \ncline{pabm}{pab}
  \ncline{pcdm}{pcd}
  \ncline{xu}{xl}
  \ncline{xl}{x}
  \ncline{yu}{yl}
  \ncline{yl}{y}
  \ncline{pxyu}{pxyl}
  \ncline{pxyl}{pxy}
}
\newcommand\pabcpinline{
  \def\yu{.3}
  \fibitem{\qab}{\bluevx{0,0}{pab}}
  \vee
  \fibitem{\qba}{\redvx{0,0}{pcd}}
  \mkern5mu
  \implies
  \mkern4mu
  \fibitem{\exists x}{\vx{0,0}{xl}\vx{0,\yu}{xu}}
  \mkern4mu
  \fibitem{\exists y}{\vx{0,0}{yl}\vx{0,\yu}{yu}}
  \mkern3mu
  \fibitem{\qxy}{\redvx{0,0}{pxyl}\bluevx{0,\yu}{pxyu}}
  \e {pab} {pcd}
  \e {xl} {yl}
  \e {yl} {pxyl}
  \e {xu} {yu}
  \e {yu} {pxyu}
  \psset{nodesep=0}
  \nccurve[angleA=25,angleB=160]{xl}{pxyl}
  \nccurve[angleA=25,angleB=160]{xu}{pxyu}
}

\newcommand\pfyxone{-1.2}
\newcommand\pfyxtwo{-.6}
\newcommand\pfyxthree{0}
\newcommand\pfyxfour{.6}
\newcommand\pfyxfive{1.2}
\newcommand\pfycoverradius{.23}
\newcommand\pfycp{
  \rput(0,\cpfibheight){
    \rput(0,-\pfycoverradius){
      \vx{\pfyxone,0}{xa}
      \bluevx{\pfyxtwo,0}{ppxa}
    }
    \rput(0,\pfycoverradius){
      \vx{\pfyxone,0}{xb}
      \redvx{\pfyxtwo,0}{ppxb}
    }
    \vx{\pfyxthree,0}{yc}
    \bluevx{\pfyxfour,0}{pyc}
    \redvx{\pfyxfive,0}{pfyc}
  }
  \lvxd{\pfyxone,0} x {\likepbar x}
  \lvxd{\pfyxtwo,0}{ppx}{\likepbar\ppx}
  \lvxd{\pfyxthree,0}{y}{\likepbar y}
  \lvxd{\pfyxfour,0}{py}{\likepbar\py}
  \lvxd{\pfyxfive,0}{pfy}{\mkern14mu\likepbar\pfy}
  \e {xa} {ppxa}
  \e {xb} {ppxb}
  \e {pyc} {pfyc}
  \e {x} {ppx}
  \e {py} {pfy}
  \fibrestyle
  \ncline {xb}{xa}
  \ncline {xa}{x}
  \ncline {ppxb}{ppxa}
  \ncline {ppxa}{ppx}
  \ncline {yc}{y}
  \ncline {pfyc}{pfy}
  \ncline {pyc}{py}
}
\newcommand\pfycpinline{
  \def\yu{.2}
  (
  \fibitem{\allx}{\vx{0,0}{xl}\vx{0,\yu}{xu}}
  \mkern2mu
  \fibitem{\px}{\bluevx{0,0}{pxl}\redvx{0,\yu}{pxu}}
  \mkern2mu
  )
  \mkern4mu
  \implies
  \mkern4mu
  \fibitem{\ally}{\vx{0,0}{y}}
  \mkern2mu
  (
  \mkern1mu
  \fibitem{\py}{\bluevx{0,0}{py}}
  \mkern-2mu
  \wedge
  \mkern-2mu
  \fibitem{\pfy}{\redvx{0,0}{pfy}}
  \mkern1mu
  )
  \e {xu} {pxu}
  \e {xl} {pxl}
}

\newcommand\eabxone{-.8}
\newcommand\eabxtwo{-.3}
\newcommand\eabxthree{.3}
\newcommand\eabxfour{.825}
\newcommand\eabcoverradius{.44}%
\newcommand\eaby{-.2}
\newcommand\eabyy{-.35}
\newcommand\eabfibheightboost{.2}
\newcommand\eabcp{
  \rput(0,\cpfibheight){
    \rput(0,\eabfibheightboost){
    \rput(0,-\eabcoverradius){
      \vx{\eabxone,0}{xl}
      \bluevx{\eabxtwo,\eabyy}{ppal}
      \redvx{\eabxthree,\eaby}{ppbl}
    }
    \rput(0,0){
      \vx{\eabxone,0}{xm}
      \redvx{\eabxfour,0}{pxm}
    }
    \rput(0,\eabcoverradius){
      \vx{\eabxone,0}{xu}
      \bluevx{\eabxfour,0}{pxu}
    }
    }
  }
  \lvxd{\eabxone,0} x {x\hspace{1.3ex}}
  \lvxd{\eabxtwo,\eabyy}{ppa}{\likepbar\ppa\hspace{.2ex}}
  \lvxd{\eabxthree,\eaby}{ppb}{\hspace{.9ex}\likepbar\ppy}
  \lvxd{\eabxfour,0}{px}{\px\hspace{-1.5ex}}
  \e {xl} {ppal}
  \e {xl} {ppbl}
  \e {ppal} {ppbl}
  \e {xm} {pxm}
  \e {xu} {pxu}
  \e {x} {ppa}
  \e {x} {ppb}
  \e {ppa} {ppb}
  \e {x} {px}
  \fibrestyle
  \ncline {xu}{xm}
  \ncline {xm}{xl}
  \ncline {xl}{x}
  \ncline {ppal}{ppa}
  \ncline {ppbl}{ppb}
  \ncline {pxu}{pxm}
  \ncline {pxm}{px}
}
\newcommand\eabcpinline{
  \def\yu{.5}
  \def\ym{.25}
  \fibitem{\existsx}{\vx{0,0}{x}\vx{0,\ym}{xm}\vx{0,\yu}{xu}}
  \mkern2mu
  (
  \mkern2mu
  \fibitem{\pa}{\bluevx{0,0}{pa}}
  \tightvee
  \mkern1.2mu
  \fibitem{\py}{\redvx{0,0}{pb}}
  \mkern-1mu
  \implies
  \mkern-1mu
  \fibitem{\px}{\redvx{0,0}{px}\bluevx{0,\ym}{pxm}}
  \mkern2mu
  )
  \e {x}{pa}
  \e {pa}{pb}
  \psset{nodesep=0}
  \nccurve[angleA=20,angleB=160]{x}{pb}
  \nccurve[angleA=20,angleB=163]{xm}{px}
  \nccurve[angleA=20,angleB=163]{xu}{pxm}
}

\newcommand\addalldrinkercoverlabels[3]{%
  \labsep{xa}{#1}{170}{2pt}
  \labsep{xb}{#2}{-170}{2pt}
  \labsep{ya}{\likex #3}{\coverylabelangle}{2pt}
  \labsep{ppxb}{\likex{\mkern-1mu\pp\mkern-.6mu #2}}{\coverppxlabelangle}{\coverppxlabeldist}
  \labsep{pya}{\likex{p #3}}{0}{2pt}
}
\newcommand\adddrinkercoverlabels[2]{\addalldrinkercoverlabels{#1}{#2}{y}}

\newcommand\drinkerfibres{
 {\fibrestyle
  \ncline{xa}{xb}
  \ncline{xb}{x}
  \ncline{ppxb}{ppx}
  \ncline{pya}{py}
  \ncline{ya}{y}
 }
}

\newcommand\fibheight{2}

\newcommand\drinkerfibvertices{
  \rput(0,0){\drinkerbasevertices}
  \rput(0,\fibheight){\drinkercoververtices}
  \drinkerfibres
}

\newcommand\drinkerfibverticesnolabels{
  \rput(0,0){\drinkerbaseverticesnolabels}
  \rput(0,\fibheight){\drinkercoververtices}
  \drinkerfibres
}

\newcommand\drinkerfibcolouredvertices{
  \rput(0,0){\drinkerbasevertices}
  \rput(0,\fibheight){\drinkercoververticescoloured}
  \drinkerfibres
}

\newcommand\drinkerfib{
  \drinkerfibvertices
  \drinkerfibres
  \drinkercoveredges
  \drinkerbaseedges
}

\newcommand\drinkerfibcoloured{
  \drinkerfibcolouredvertices
  \drinkerfibres
  \drinkercoveredges
  \drinkerbaseedges
}

\newcommand\drinkerbasemographedges{
  \drinkerbaseedges
  {\psset{linecolor=bindingcolour}
    \nccurve[arrows=->,angleA=5,angleB=-155,dash=3pt 1.5pt,nodesep=2pt] {y} {py}
    \nccurve[arrows=->,angleA=-60,angleB=-175,dash=3pt 1.5pt,nodesep=2pt] {x} {ppx}
  }
  {\dualitystyle
    \nccurve[angleA=-20,angleB=-115,dash=3pt 1.5pt] {ppx} {py}
  }
}

\newcommand\drinkercovermographedges{
  \drinkercoveredges
  {\psset{linecolor=bindingcolour}
    \nccurve[arrows=->,angleA=5,angleB=-155,dash=3pt 1.5pt,nodesep=2pt] {ya} {pya}
    \nccurve[arrows=->,angleA=-60,angleB=-175,dash=3pt 1.5pt,nodesep=2pt] {xb} {ppxb}
  }
  {\dualitystyle
    \nccurve[angleA=-15,angleB=-110,dash=3pt 1.5pt] {ppxb} {pya}
  }
}

\newcommand\drinkerhomfibedges{
  \drinkercovermographedges
  \drinkerbasemographedges
}

\newcommand\drinkerfibnolabels{
  \drinkerfibverticesnolabels
  \drinkerfibres
  \drinkerhomfibedges
}

\newcommand\drinkerbasemograph{
  \drinkerbaseverticesnolabels
  \drinkerbasemographedges
}

\newcommand\collapsedegformula{\forall x\mkern2mu \exists y (py\tightvee\ppy)}

\newcommand\collapsedegx{2ex}
\newcommand\collapsedegxx{6ex}
\newcommand\collapsedegnudge{.23}
\newcommand\collapsedsourceduality{{\dualitystyle\nccurve[angleA=30,angleB=150] {pya} {ppya}}}
\newcommand\collapsedsourceedges{\e{pya}{ya}\e{ya}{ppya}}
\newcommand\collapsedsourcebindings{{\bindingstyle
  \psset{nodesepB=1pt}
  \nccurve[angleA=-150,angleB=-30] {ya} {pya}
  \nccurve[angleA=-30,angleB=-150] {ya} {ppya}
}}
\newcommand\collapsedtargetduality{{\dualitystyle\nccurve[angleA=-30,angleB=-150] {py} {ppy}}}
\newcommand\collapsedtargetedges{\e{py}{y}\e{y}{ppy}}
\newcommand\collapsedtargetbindings{{\bindingstyle
  \psset{nodesepB=1pt}
  \nccurve[angleA=150,angleB=30] {y} {py}
  \nccurve[angleA=30,angleB=150] {y} {ppy}
}}

\newcommand\collapsedsource{
  \vx{-\collapsedegxx,0}{xa}
  \vx{-\collapsedegx,0}{pya}
  \vx{\collapsedegx,0}{ya}
  \vx{\collapsedegxx,0}{ppya}
  \collapsedsourceedges
  \collapsedsourcebindings
  \collapsedsourceduality
}

\newcommand\uncollapsedsource{
  \vx{-\collapsedegxx,\collapsedegnudge}{x2}
  \vx{-\collapsedegxx,-\collapsedegnudge}{x1}
  \vx{-\collapsedegx,0}{pya}
  \vx{\collapsedegx,0}{ya}
  \vx{\collapsedegxx,0}{ppya}
  \collapsedsourceedges
  \collapsedsourcebindings
  \collapsedsourceduality
}

\newcommand\collapsedsourcecoloured{
  \vx{-\collapsedegxx,0}{xa}
  \bluevx{-\collapsedegx,0}{pya}
  \vx{\collapsedegx,0}{ya}
  \bluevx{\collapsedegxx,0}{ppya}
  \collapsedsourceedges
}

\newcommand\collapsedtargetvertices[4]{%
  \lvxd{-\collapsedegxx,0}{x}{#1}
  \lvxd{-\collapsedegx,0}{py}{#2}
  \lvxd{\collapsedegx,0}{y}{#3}
  \lvxd{\collapsedegxx,0}{ppy}{#4}
}

\newcommand\collapsedtarget{
  \collapsedtargetvertices{}{}{}{}
  \collapsedtargetedges
  \collapsedtargetbindings
  \collapsedtargetduality
}

\newcommand\collapsedtargetlabelled{
  \collapsedtargetvertices{x}{\py}{y}{\ppy}
  \collapsedtargetedges
}

\newcommand\uncollapsedfibres{
  \fibrestyle
  \ncline{x2}{x1}
  \ncline{x1}{x}
  \ncline{ya}{y}
  \ncline{pya}{py}
  \ncline{ppya}{ppy}
}

\newcommand\collapsedegfibheight{1.9}
\newcommand\uncollapsedeg{
  \rput(0,\collapsedegfibheight){\uncollapsedsource}
  \collapsedtarget
  \uncollapsedfibres
}

\newcommand\collapsedfibres{
  \fibrestyle
  \ncline{xa}{x}
  \ncline{ya}{y}
  \ncline{pya}{py}
  \ncline{ppya}{ppy}
}

\newcommand\collapsedeg{
  \rput(0,\collapsedegfibheight){\collapsedsource}
  \collapsedtarget
  \collapsedfibres
}

\newcommand\labelledcollapsedeg{
  \rput(0,\collapsedegfibheight){\collapsedsourcecoloured}
  \collapsedtargetlabelled
  \collapsedfibres
}

\newcommand\drinkerfiblabelledpair[2]{%
  \drinkerfib
  \adddrinkercoverlabels{#1}{#2}
}

\newcommand\drinkerbindingfiblabelled[2]{%
  \drinkerfibvertices
  \adddrinkercoverlabels{#1}{#2}
  \de{ya}{pya}
  \de{xb}{ppxb}
  \de{y}{py}
  \de{x}{ppx}
}

\newcommand\likex[1]{\raisebox{0ex}[1ex][0ex]{$#1$}}
\newlength\fheight
\newlength\pbarheight
\newcommand\likef[1]{\raisebox{0ex}[\the\fheight][0ex]{$#1$}}
\newcommand\likepbar[1]{\raisebox{0ex}[\the\pbarheight][0ex]{$#1$}}

\newcommand\vertexrad{2pt}

\newcommand\vx[2]{%
\Cnodeput*[fillcolor=black,linecolor=black,radius=\vertexrad,framesep=0pt](#1){#2}{}}

\newcommand\bluevx[2]{%
\Cnodeput[linewidth=.8pt,linecolor=black!35!blue,fillstyle=solid,fillcolor=white!92!blue,radius=\vertexrad,framesep=0pt](#1){#2}{}}

\newcommand\redvx[2]{%
\fnode[linewidth=.8pt,linecolor=black!55!red,fillstyle=solid,fillcolor=white!55!red,framesize=4.5pt,framesep=0pt](#1){#2}{}}

\newcommand\greenvx[2]{%
\rput(#1){\dianode[linewidth=.8pt,linecolor=black!65!green,fillstyle=solid,fillcolor=white!70!green,framesep=0pt]{#2}{\rule{2pt}{0pt}\rule{0pt}{2pt}}}}

\newcommand\labsep[4]{%
\nput*[labelsep=#4,framesep=0pt]{#3}{#1}{\ensuremath{#2}}}

\newcommand\lvxsep[5]{%
\vx{#1}{#2}\nput*[labelsep=#5,framesep=0pt]{#4}{#2}{\ensuremath{#3}}}

\newcommand\bluelvxsep[5]{%
\bluevx{#1}{#2}\nput*[labelsep=#5,framesep=0pt]{#4}{#2}{\ensuremath{#3}}}

\newcommand\redlvxsep[5]{%
\redvx{#1}{#2}\nput*[labelsep=#5,framesep=0pt]{#4}{#2}{\ensuremath{#3}}}

\newcommand\redlvx[4]{%
\redlvxsep{#1}{#2}{#3}{#4}{2pt}}
\newcommand\lvx[4]{%
\lvxsep{#1}{#2}{#3}{#4}{2pt}}
\newcommand\lvxl[3]{%
\lvx{#1}{#2}{\likex{#3}}{180}}
\newcommand\lvxr[3]{%
\lvx{#1}{#2}{\likex{#3}}{0}}
\newcommand\lvxu[3]{%
\lvxsep{#1}{#2}{#3}{90}{3pt}}
\newcommand\lvxd[3]{%
\lvxsep{#1}{#2}{#3}{-90}{3pt}}

\newcommand\e[2]{\ncline[nodesep=0pt]{#1}{#2}}%
\newcommand\de[2]{\ncline[nodesep=2pt,arrows=->,arrowinset=.3,arrowlength=.7]{#1}{#2}}
\newcommand\dee[2]{\de[2]}

\def\edgelen{.8}%
\SQUAREROOT{.75}{\equitriangleheightmultiplier}
\MULTIPLY{\edgelen}{\equitriangleheightmultiplier}{\height}%
\DIVIDE{\edgelen}{2}{\halfedgelen}
\DIVIDE{\halfedgelen}{2}{\quarteredgelen}
\DIVIDE{\quarteredgelen}{2}{\eighthedgelen}
\DIVIDE{\height}{2}{\halfheight}
\MULTIPLY{\edgelen}{2}{\twoedgelen}
\MULTIPLY{\edgelen}{3}{\threeedgelen}

\newcommand\lf{\TC*}
\newcommand\leaf[1]{\TC*~{#1}}
\newcommand\plusnode{\TCircle[radius=5.5pt,linestyle=none]{\raisebox{-1pt}[5.5pt]{\(\graphunion\)}}}
\newcommand\timesnode{\TCircle[radius=5.5pt,linestyle=none]{\raisebox{2pt}{\(\graphjoin\)}}}

\newcommand\cotreesepsize[4]{%
  \psTree[radius=2pt,nodesepB=0pt,nodesepA=0pt,levelsep=\height cm,
  tnheight=1ex,labelsep=3pt,treesep=#2,treenodesize=#3,tnheight=1ex]
  {#1}
  {#4}
  \endpsTree
}
\newcommand\cotreesep[3]{\cotreesepsize{#1}{#2}{0}{#3}}%

\newcommand\cotree[2]{\cotreesep{#1}{\edgelen}{#2}}%

\newcommand \plustree[1]{\cotree {\plusnode}{#1}}%
\newcommand\timestree[1]{\cotree{\timesnode}{#1}}%
\newcommand\plustreesep[2]{\cotreesep{\plusnode}{#1}{#2}}%
\newcommand\timestreesep[2]{\cotreesep{\timesnode}{#1}{#2}}%

\newcommand\leafy{\leaf{y\mkern-1mu}}%

\renewcommand\put[1]{\rput(0,0){#1}}
\newcommand\column[1]{\begin{array}{@{}c@{}}#1\end{array}}
\newcommand\rputcolumn[1]{\rput(0,0){\column{#1}}}

\newcounter{egno}
\newcommand\example[1]{\refstepcounter{egno}\label{#1}}
\newcommand\showeg[1]{(E\ref*{#1})\xspace}
\newcommand\puteg[1]{\mbox{\strut\example{#1}}\\[6ex]\put{\showeg{#1}}}
\newcommand\refeg[1]{E\ref{#1}\xspace}
\newcommand\refegs[2]{\refeg{#1}--\ref{#2}}

\newcommand\xypxytriangle{
\lvx{-\halfedgelen,-\halfheight}{x}{x\:}{-120}
\lvxu{0,\halfheight}{y}{y}
\lvx{\halfedgelen,-\halfheight}{pxy}{\;\;\pxy}{-60}}

\newcommand\xypxybindinggraph{
\xypxytriangle
\de{x}{pxy}
\de{y}{pxy}
}

\newcommand\xyonetriangle{
\lvx{-\halfedgelen,-\halfheight}{x}{x\:}{-120}
\lvxu{0,\halfheight}{y}{y}
\lvx{\halfedgelen,-\halfheight}{one}{\mkern1mu1}{-40}}

\newcommand\xyonebindinggraph{\xyonetriangle}

\newcommand\pfyformula{(\axpx) \mkern1mu \implies \mkern1mu \forall y\mkern3mu (\mkern1mu\py\tightwedge \pfy\mkern1mu)}
\newcommand\eabformula{\exists x\mkern2mu (\mkern1mu\pa\tightvee\mkern1.2mu\py\mkern-1mu\implies\mkern-1mu px\mkern1mu)}
\newcommand\onionformula{\left(\strut\forall x\mkern1mu  (\mkern1mu \pfx\tightimplies \px)\right)\mkern2mu \implies\mkern2mu  \forall y\mkern2mu  (\mkern1mu \pffy\tightimplies \py\mkern1mu )}
\newcommand\pabformula{\qab\vee \qba\mkern4mu \implies \mkern4mu \exists x\mkern4mu \exists y\mkern3mu \qxy}

\newcommand\cponeformula{\pfyformula}
\newcommand\cptwoformula{\pabformula}
\newcommand\cpthreeformula{\onionformula}
\newcommand\cpfourformula{\eabformula}

\newcommand\cpone{\pfycp}
\newcommand\cptwo{\pabcp}
\newcommand\cpthree{\onioncp}
\newcommand\cpfour{\eabcp}

\newcommand\cponeinline{\pfycpinline}
\newcommand\cptwoinline{\pabcpinline}
\newcommand\cpthreeinline{\onioncpinline}
\newcommand\cpfourinline{\eabcpinline}

\newcommand\cponformula[3]{%
  \rput(#2,1.7){#1}
  \rput[B](0,0){#3}%
}

\newcommand\dcponformula[3]{%
  \rput(#2,1){#1}
  \rput[B](0,0){#3}%
}

\newcommand\figcps{\begin{figure*}\begin{center}%
\newcommand\radius{3.8}
\begin{pic}{-1.5}{11}
  \rput(-\radius,6.8){\cponformula{\cpone}{-.13}{\cponeformula}}
  \rput(\radius,6.8){\cponformula{\cptwo}{0}{\cptwoformula}}
  \rput(-\radius,0){\cponformula{\cpthree}{-2}{\cpthreeformula}}
  \rput(\radius,0){\cponformula{\cpfour}{-.15}{\cpfourformula}}
\end{pic}
\end{center}%
\caption{\label{fig:cps}%
  Four combinatorial proofs, each shown above the formula proved.
  Here $x$ and $y$ are variables, $f$ is a unary function symbol, $a$ and $b$ are constants (nullary function symbols), $p$ is a unary predicate symbol, and $q$ is a binary predicate symbol.%
}\end{figure*}}

\newcommand\figcpscondensed{\begin{figure*}\begin{center}%
\newcommand\radius{3.6}%
\begin{pic}{-.5}{3}
  \rput(-\radius,2.1){\cponeinline}
  \rput(\radius,2.1){\cptwoinline}
  \rput(-\radius,0){\cpthreeinline}
  \rput(\radius,0){\cpfourinline}
\end{pic}%
\end{center}%
\caption{\label{fig:cps-condensed}Condensed forms of the four combinatorial proofs in Fig.\,\ref{fig:cps}.}\figrule\end{figure*}}

\newcommand\figGraphExamples{\begin{figure*}\begin{center}\begin{math}\mkern-50mu
\newcommand\xpxhoriz{\lvxl{-\halfedgelen,0}{x}{x}\lvxr{\halfedgelen,0}{px}{\px}}
\newcommand\xpxcotree[1]{\cotree{##1}{\leaf{x}\leaf{px}}}%
\begin{array}{@{\hspace{-6ex}}c@{\hspace{16ex}}c@{\hspace{19ex}}c@{\hspace{20ex}}c@{\hspace{20ex}}c@{\hspace{0ex}}}
\\[-12ex]
\puteg{axpx}
&
\put{\forall x\, \px }
&
\xpxhoriz
&
\rput(0,-.1ex){\xpxcotree{\plusnode}}
&
\xpxhoriz
\de x {px}
\\[9ex]
\puteg{expx}
&
\put{\exists x\, \px}
&
\xpxhoriz
\e {px} {x}
&
\rput(0,-.1ex){\xpxcotree{\timesnode}}
&
\xpxhoriz
\de x {px}
\\[11ex]
\puteg{exaypxy}
&
\put{\exists x\,\forall y\,\pxy}
&
\xypxytriangle
\e{y}{x}
\e{pxy}{x}
&
\rput(0,.8ex){\timestree{\leaf{x}\plustree{\leafy\leaf{pxy}}}}
&
\xypxybindinggraph
\\[13ex]
\puteg{axeyone}
&
\put{\forall x\,\exists y\,1}
&
\xyonetriangle
\e y {one}
&
\rput(0,.8ex){\plustree{\leaf{x}\timestree{\leafy\leaf{\raisebox{2pt}{\(1\)}}}}}
&
\xyonebindinggraph
\\[12ex]
\puteg{axaypxy}
&
\rputcolumn{\forall x\,\forall y\,\pxy \\[2ex] \forall y\,\forall x\,\pxy}
&
\xypxytriangle
&
\rput(0,.1ex){\plustree{\leaf{x}\leafy\leaf{pxy}}}
&
\xypxybindinggraph
\\[13ex]
\puteg{exeypxy}
&
\rputcolumn{\exists x\,\exists y\,\pxy \\[2ex] \exists y\,\exists x\,\pxy}
&
\xypxytriangle
\e{y}{x}
\e{pxy}{x}
\e{pxy}{y}
&
\rput(0,.1ex){\timestree{\leaf{x}\leafy\leaf{pxy}}}
&
\xypxybindinggraph
\\[12ex]
\end{array}\mkern-50mu\end{math}\end{center}\caption{\label{fig:graph-examples}Examples \protect\refegs{axpx}{exeypxy}.
Each syntactic formula $\formula$ is followed by its graph $\protect\cfa=\protect\gformula$,
cotree $\protect\cotreeof\cfa$, and binding graph $\protect\bindinggraphof\cfa$.
Examples
\protect\refeg{axaypxy} and \protect\refeg{exeypxy}
show two syntactic formulas with the same combinatorial formulas.}\end{figure*}}

\newcommand\xpxqxtriangle{
\lvx{-\halfedgelen,-\halfheight}{x}{x\:}{-120}
\lvxu{0,\halfheight}{px}{px}
\lvx{\halfedgelen,-\halfheight}{qx}{\;\mkern2mu\qx}{-60}}

\newcommand\xpxqtriangle{
\lvx{-\halfedgelen,-\halfheight}{x}{x\:}{-120}
\lvxu{0,\halfheight}{px}{px}
\lvx{\halfedgelen,-\halfheight}{q}{\mkern1mu q}{-37}
}

\newcommand\xpxyqysquare{
\lvx{-\halfedgelen,\halfedgelen}{x}{x\:}{120}
\lvx{-\halfedgelen,-\halfedgelen}{y}{y\:}{-120}
\lvx{\halfedgelen,\halfedgelen}{px}{\mkern0mu\px}{41}
\lvx{\halfedgelen,-\halfedgelen}{qy}{\;\mkern.5mu\qy}{-60}
}

\newcommand\axiomeg{\mkern5mu{}\singletonred{\mkern2mu p\mkern-1mu  x\mkern-1mu fy}\mkern10mu\singletonred{\mkern2mu \pp\mkern-1mu  x\mkern-1mu f y}\mkern3mu}

\newcommand\twolinkp{\{\mkern2mu\singletonred\ppx,\singletonred\pz\mkern2mu\}}
\newcommand\twolinkq{\{\mkern2mu\singletonblue\qqy,\singletonblue{\q fz}\mkern2mu\}}
\newcommand\twolinkfographvertices{
  \rput(-\halfedgelen,0){
    \lvx{-\edgelen,\halfedgelen}{x}{x}{180}
    \redlvx{0,\halfedgelen}{ppx}{\ppx}{90}
    \lvx{-\edgelen,-\halfedgelen}{y}{y}{180}
    \bluelvxsep{0,-\halfedgelen}{qqy}{\likex\qqy}{-90}{4pt}
  }
  \rput(\halfedgelen,0){
    \redlvx{0,\halfedgelen}{pz}{\pz}{90}
    \bluelvxsep{0,-\halfedgelen}{qfz}{\likex{\qfz}}{-90}{4pt}
  }
  \rput(\edgelen,0){\lvx{\halfedgelen,0}zz0}
}
\newcommand\twolinkfographedges{
  \e x {ppx}
  \e y {qqy}
  \e {pz} {qfz}
}
\newcommand\twolinkfograph{\twolinkfographvertices\twolinkfographedges}
\newcommand\monadicfographeg{
  \rput(-\edgelen,0){\lvx{-\halfedgelen,0}{x}{x}{-180}}
  \rput(-\halfedgelen,0){
    \lvx{0,\halfedgelen}{pxu}{\px}{90}
    \lvxsep{0,-\halfedgelen}{pxd}{\likex{\px}}{-90}{4pt}
  }
  \rput(\halfedgelen,0){
    \lvx{\edgelen,0}{y}{y}{0}
    \lvx{0,0}{ppy}{\ppy}{90}
  }
  \e{pxu}{pxd}
  \e{y}{ppy}
}
\newcommand\mographeg{
  \rput(-\edgelen,0){\vx{-\halfedgelen,0}{x}}
  \rput(-\halfedgelen,0){
    \vx{0,\halfedgelen}{pxu}
    \vx{0,-\halfedgelen}{pxd}
  }
  \rput(\halfedgelen,0){
    \vx{\edgelen,0}{y}
    \vx{0,0}{ppy}
  }
  \e{pxu}{pxd}
  \e{y}{ppy}
   {\bindingstyle
    \nccurve[angleA=50,angleB=180] {x} {pxu}
    \nccurve[angleA=-50,angleB=-180] {x} {pxd}
    \nccurve[angleA=150,angleB=30] {y} {ppy}
   }
  \dualitystyle
    \nccurve[angleA=0,angleB=130] {pxu} {ppy}
    \nccurve[angleA=0,angleB=-130] {pxd} {ppy}
}

\newcommand\twolinkassignment{\assignopen\assign x z,\assign y{fz}\assignclose}
\newcommand\twolinkleapgraphedges{{\psset{linestyle=dashed,linewidth=.8pt,nodesep=0pt}
  \nccurve[angleA=-30,angleB=-150]{ppx}{pz}
  \nccurve[angleA=30,angleB=150]{qqy}{qfz}
  \psset{ncurv=1.2}
  \nccurve[angleA=55,angleB=105] x z
  \nccurve[angleA=-55,angleB=-105] y z}}
\newcommand\twolinkleapgraph{
  \twolinkfographvertices
  \twolinkleapgraphedges
}

\newcommand\monetegvertices{
  \rput(-\edgelen,0){\vx{-\halfedgelen,0}{x}}
  \rput(-\halfedgelen,0){
    \vx{0,\halfedgelen}{pxu}
    \vx{0,-\halfedgelen}{pxd}
  }
  \rput(\halfedgelen,0){
    \vx{0,\halfedgelen}{ppyu}
    \vx{0,-\halfedgelen}{ppyd}
    \vx{\edgelen,0}{y}
  }
}

\newcommand\moneteg{
  \monetegvertices
  \e{pxu}{pxd}
  \e{y}{ppyu}
  \e{y}{ppyd}
   {\bindingstyle
    \nccurve[angleA=45,angleB=-170] {x} {pxu}
    \nccurve[angleA=-45,angleB=170] {x} {pxd}
    \nccurve[angleA=120,angleB=10] {y} {ppyu}
    \nccurve[angleA=-120,angleB=-10] {y} {ppyd}
   }
  \monetegdualities
}

\newcommand\monetegdualities{
  \dualitystyle
    \nccurve[angleA=20,angleB=160] {pxu} {ppyu}
    \nccurve[angleA=-20,angleB=-160] {pxd} {ppyd}
}

\newcommand\monetegleapgraphedges{
  \monetegdualities
  \dualitystyle
  \nccurve[angleA=10,angleB=170] {x} {y}
}

\newcommand\monetegleapgraph{
  \monetegvertices
  \monetegleapgraphedges
}

\newcommand\drinkersquare{
\lvx{-\halfedgelen,\halfedgelen}{x}{x}{135}
\lvx{\halfedgelen,\halfedgelen}{px}{\mkern.5mu\likex\ppx}{45}
\lvx{-\halfedgelen,-\halfedgelen}{y}{\mkern.5mu\likex y}{-135}
\lvx{\halfedgelen,-\halfedgelen}{py}{\py}{-45}
}
\newcommand\drinkersquareunlabelled{
\lvx{-\halfedgelen,\halfedgelen}{x}{}{135}
\lvx{\halfedgelen,\halfedgelen}{px}{}{45}
\lvx{-\halfedgelen,-\halfedgelen}{y}{}{-135}
\lvx{\halfedgelen,-\halfedgelen}{py}{}{-45}
}

\newcommand\figMoreGraphExamples{\begin{figure*}\begin{center}\begin{math}
\begin{array}{@{\hspace{-25ex}}c@{\hspace{18ex}}c@{\hspace{21ex}}c@{\hspace{22ex}}c@{\hspace{22ex}}c@{\hspace{-20ex}}}
\\[-10.95ex]
\puteg{axpxqx}
&
\rputcolumn{\forall x\: (\px\wedge\qx) \\[2ex] \forall x\: (\qx\wedge\px)}
&
\xpxqxtriangle
\e{px}{qx}
&
\rput(0,.8ex){\plustree{\leaf{x}\timestree{\leaf{\px}\leaf{\qx}}}}
&
\xpxqxtriangle
\de x {px}
\de x {qx}
\\[11ex]
\puteg{expxqx}
&
\rputcolumn{\exists x\: (\px\vee\qx) \\[2ex] \exists x\: (\qx\vee\px)}
&
\xpxqxtriangle
\e{x}{px}
\e{x}{qx}
&
\rput(0,.8ex){\timestree{\leaf{x}\plustree{\leaf{px}\leaf{qx}}}}
&
\xpxqxtriangle
\de x {px}
\de x {qx}
\\[12ex]
\puteg{expxq}
&
\rputcolumn{
  \exists x\: (\,\px\wedge q\,)
  \\[1.6ex]
  \exists x\: (\,q \wedge\px\,)
  \\[1.6ex]
  (\,\exists x\, \px\,) \wedge q
  \\[1.6ex]
  q \wedge (\,\exists x\, \px\,)
}
&
\xpxqtriangle
\e{x}{px}
\e{x}{q}
\e{px}{q}
&
\rput(0,.8ex)
  {\timestree{\leaf{\rule{0ex}{1.3ex}x}\leaf{\rule{0ex}{1.3ex}px}\leaf{\mkern2mu q}}}
&
\xpxqtriangle
\de x {px}
\\[12ex]
\puteg{expxeyqy}
&
\rputcolumn{
  (\exists x\: \px) \vee (\exists y\: \qy)
  \\[2ex]
  (\exists y\: \qy) \vee (\exists x\: \px)
}
&
\xpxyqysquare
\e x {px}
\e y {qy}
&
\rput(0,.5ex){\plustree{\timestree{\leaf x \leaf {\px}}\timestree{\leafy \leaf {\qy}}}}
&
\xpxyqysquare
\de x {px}
\de y {qy}
\\[12ex]
\puteg{axpxayqy}
&
\rputcolumn{
  (\forall x\: \px) \wedge (\forall y\: \qy)
  \\[2ex]
  (\forall y\: \qy) \wedge (\forall x\: \px)
}
&
\xpxyqysquare
\e x y
\e x {qy}
\e y {px}
\e {px} {qy}
&
\rput(0,.5ex){\timestree{\plustree{\leaf x \leaf {\px}}\plustree{\leafy \leaf {\qy}}}}
&
\xpxyqysquare
\de x {px}
\de y {qy}
\\[12ex]
\puteg{expxayqy}
&
\rputcolumn{
  \exists x \left(\rule{0ex}{1.7ex}\,\px \vee (\forall y\, \qy) \,\right)
  \\[2ex]
  \exists x \left(\rule{0ex}{1.7ex}\,(\forall y\, \qy) \vee \px \,\right)
  \\[2ex]
  \exists x \left(\rule{0ex}{1.7ex}\,\forall y\, (\px \vee \qy) \,\right)
  \\[2ex]
  \exists x \left(\rule{0ex}{1.7ex}\,\forall y\, (\qy \vee \px) \,\right)
}
&
\xpxyqysquare
\e x y
\e x {px}
\e x {qy}
&
\rput(0,.5ex){\timestree{\leaf x \tspace{3ex} \plustree{\leafy\leaf{\px}\leaf{\qy}}}}
&
\xpxyqysquare
\de x {px}
\de y {qy}
\\[12ex]
\end{array}\end{math}\end{center}\caption{\label{fig:more-graph-examples}Examples \refegs{axpxqx}{expxayqy}.
Each syntactic formula $\formula$ is followed by its graph $\cfa\tighteq\gformula$,
cotree $\protect\cotreeof\cfa$, and binding graph $\protect\bindinggraphof\cfa$.
Each example shows multiple syntactic formulas with the same combinatorial formula.}\end{figure*}}

\newcommand\cpropy{-1.8}%
\newcommand\dpropmid{-.2}
\newcommand\dpropgap{-.4}%
\newcommand\dpropy{-3.6}%
\newcommand\onethreecurve{\nccurve[angleA=20,angleB=165] 1 3}
\newcommand\twofourcurve{\nccurve[angleA=-15,angleB=-160] 2 4}
\newcommand\onetwocurve{\nccurve[angleA=-65,angleB=165] 1 2}
\newcommand\threefourcurve{\nccurve[angleA=-25,angleB=125] 3 4}
\newcommand\peirceovercombprop{
  \left(\rule{0ex}{1.7ex}
  \mkern-1mu
  (
  \fibitem{\pp}{\rput(0,\cpropy){\lvxu{0,0}{c1}{\pp}}\rput(0,\dpropy){\vx{0,0}{1}}}
  \mkern1mu
  \tightvee
  \mkern-1mu
  \fibitem{q}{\rput(0,\cpropy){\lvxd{0,\dpropgap}{c2}{q}}\rput(0,\dpropy){\vx{0,\dpropgap}{2}}}
  )
  \mkern-.3mu
  \tightwedge
  \mkern1mu
  \fibitem{\pp}{\rput(0,\cpropy){\lvxu{0,\dpropmid}{c3}{\pp}}\rput(0,\dpropy){\vx{0,\dpropmid}{3}}}
  \right)
  \mkern-1.5mu
  \tightvee
  \fibitem{p}{\rput(0,\cpropy){\lvxd{0,\dpropmid}{c4}{p}}\rput(0,\dpropy){\vx{0,\dpropmid}{4}}}
  \e 1 3
  \e 2 3
  \e{c1}{c3}
  \e{c2}{c3}
  \dualitystyle
  \nccurve[angleA=15,angleB=150] 1 4
  \nccurve[angleA=-30,angleB=-150] 3 4
}

\newcommand\combproptwo{
  (
  \fibitem{p}{\rput(0,\cpropy){\lvxu{0,0}{c1}{p}}\rput(0,\dpropy){\vx{0,0}{1}}}
  \mkern-1mu
  \vee
  \mkern-1mu
  \fibitem{\pp}{\rput(0,\cpropy){\lvxd{0,\dpropgap}{c2}{\pp}}\rput(0,\dpropy){\vx{0,\dpropgap}{2}}}
  )
  \mkern-1mu
  \wedge
  \mkern-1mu
  (
  \fibitem{q}{\rput(0,\cpropy){\lvxu{0,0}{c3}{q}}\rput(0,\dpropy){\vx{0,0}{3}}}
  \mkern-1mu
  \vee
  \mkern-1mu
  \fibitem{\qq}{\rput(0,\cpropy){\lvxd{0,\dpropgap}{c4}{\qq}}\rput(0,\dpropy){\vx{0,\dpropgap}{4}}}
  )
  \e 1 3
  \e 2 3
  \e 1 4
  \e 2 4
  \e{c1}{c3}
  \e{c2}{c3}
  \e{c1}{c4}
  \e{c2}{c4}
  \dualitystyle
  \onetwocurve
  \threefourcurve
}

\newcommand\combpropone{
  (
  \fibitem{p}{\rput(0,\cpropy){\lvxu{0,0}{c1}{p}}\rput(0,\dpropy){\vx{0,0}{1}}}
  \mkern-1mu
  \wedge
  \mkern-1mu
  \fibitem{p}{\rput(0,\cpropy){\lvxd{0,\dpropgap}{c2}{p}}\rput(0,\dpropy){\vx{0,\dpropgap}{2}}}
  )
  \mkern-1mu
  \vee
  \mkern-1mu
  (
  \fibitem{\pp}{\rput(0,\cpropy){\lvxu{0,0}{c3}{\pp}}\rput(0,\dpropy){\vx{0,0}{3}}}
  \mkern-1mu
  \wedge
  \mkern-1mu
  \fibitem{\pp}{\rput(0,\cpropy){\lvxd{0,\dpropgap}{c4}{\pp}}\rput(0,\dpropy){\vx{0,\dpropgap}{4}}}
  )
  \e 1 2
  \e 3 4
  \e{c1}{c2}
  \e{c3}{c4}
  \dualitystyle
  \onethreecurve
  \nccurve[angleA=10,angleB=-170] 2 3
  \nccurve[angleA=0,angleB=180] 1 4
  \twofourcurve
}

\newcommand\combpropthree{
  (
  \fibitem{p}{\rput(0,\cpropy){\lvxu{0,0}{c1}{p}}\rput(0,\dpropy){\vx{0,0}{1}}}
  \mkern-1mu
  \vee
  \mkern-1mu
  \fibitem{\pp}{\rput(0,\cpropy){\lvxd{0,\dpropgap}{c2}{\pp}}\rput(0,\dpropy){\vx{0,\dpropgap}{2}}}
  )
  \mkern-1mu
  \wedge
  \mkern-1mu
  (
  \fibitem{\pp}{\rput(0,\cpropy){\lvxu{0,0}{c3}{\pp}}\rput(0,\dpropy){\vx{0,0}{3}}}
  \mkern-1mu
  \vee
  \mkern-1mu
  \fibitem{p}{\rput(0,\cpropy){\lvxd{0,\dpropgap}{c4}{p}}\rput(0,\dpropy){\vx{0,\dpropgap}{4}}}
  )
  \e 1 3
  \e 2 3
  \e 1 4
  \e 2 4
  \e{c1}{c3}
  \e{c2}{c3}
  \e{c1}{c4}
  \e{c2}{c4}
  \dualitystyle
  \onethreecurve
  \onetwocurve
  \threefourcurve
  \twofourcurve
}

\newcommand\bindinggraphof[1]{\overset{\scalebox{.7}{\raisebox{-3.2pt}[0ex][0pt]{\(\bindrel\mkern-3.5mu\)}}}{#1}}
\newcommand\widebindinggraphof[2]{\overset{\scalebox{.7}{\raisebox{-3.2pt}[0ex][0pt]{\(\widebindrel{#2}\mkern-3.5mu\)}}}{#1}}

\newcommand\bindingrelof[1]{\mathbin{\mkern-2mu\bindinggraphof{#1}\mkern-2mu}}

\newcommand\leapedges{L}
\newcommand\leapsof[1]{\leapedges_{#1}}
\newcommand\leapgraphof[1]{\mathcal{\leapedges}_{#1}}

\newcommand\tightplus{\mathbin+}
\newcommand\tighttimes{\mathbin\times}
\newcommand\tightwedgeveeshrink{\mkern-2mu}
\newcommand\tightvee{\mathbin{\tightwedgeveeshrink\vee\tightwedgeveeshrink}}
\newcommand\tightwedge{\mathbin{\tightwedgeveeshrink\wedge\tightwedgeveeshrink}}

\newcommand\skewlifting[1]{\widetilde{#1}}

\newcommand\colourequiv{\sim}

\newcommand\setof[1]{\{\,{#1}\,\}}
\newcommand\setst[2]{\{\,{#1}:{#2}\,\}}
\newcommand\setsuchthat[2]{\setst{#1}{#2}}
\newcommand\tightin{\mathbin{\mkern-5mu{}\in{}\mkern-5mu}}
\newcommand\tightnotin{\mathbin{\mkern-5mu{}\notin{}\mkern-3mu}}
\newcommand\induced[2]{#1[#2]}

\newcommand\tightsubseteq{\mathbin{\mkern-1mu\subseteq\mkern-1mu}}
\newcommand\tightsupseteq{\mathbin{\mkern-1mu\supseteq\mkern-1mu}}
\newcommand\tightsetminus{\mathbin{\mkern-2mu\setminus\mkern-2mu}}

\newcommand\tightle{\mathbin{\mkern-2mu\le\mkern-2mu}}
\newcommand\tightge{\mathbin{\mkern-2mu\ge\mkern-2mu}}
\newcommand\tightlt{\mathbin{\mkern-2mu<\mkern-2mu}}

\newcommand\module{M}
\newcommand\modulep{\module\mkern-2mu'}

\newcommand\cleaningegone{
\namedsingletonleft{x1}{x}\hspace*{1.2ex}\namedsingletonright{px}{px}\e{x1}{px}
\hspace*{1.2ex}
\namedsingletonleft{x2}{x}\hspace*{1.2ex}\namedsingletonright{qx}{qx}\e{x2}{qx}
}

\newcommand\cleaningegtwo{
\namedsingletonleft{x1}{x}\hspace*{1.2ex}\namedsingletonright{px}{px}\e{x1}{px}
\hspace*{1ex}
\namedsingletonright{qx}{qx}
}

\newcommand\bigcleaningeg[2]{
\namedsingletonright{px}{px}
\hspace*{1.2ex}
\namedsingletonleft{x1}{#1}\hspace*{1.2ex}\namedsingletonright{qx}{q#1}\e{x1}{qx}
\hspace*{1.2ex}
\namedsingletonleft{x2}{#2}\hspace*{1.2ex}\namedsingletonright{rxx}{r#2#2}\e{x2}{rxx}
}
\newcommand\primed{\mkern-1.2mu'}

\newcommand\diredge[2]{\langle #1,#2\rangle}

\newcommand\shortimplies{\mathop{\mkern-0mu\implies\mkern-1mu}}
\newcommand\tightimplies{\shortimplies}
\newcommand\tighteq{\mathop{\mkern-1mu=\mkern-1mu}}
\newcommand\tightneq{\mathop{\mkern-1mu\neq\mkern-1mu}}
\newcommand\calculusfont[1]{\textnormal{\textbf{\textsf{#1}}}}
\newcommand\LK{\calculusfont{LK}\xspace}

\newcommand\figurelkdrinkerproof{\begin{figure*}\begin{center}\vspace{-2ex}\psscalebox{.92}{\begin{math}
\contractionruleright{
  \existsruleright{
    \impliesruleright{
      \weakeningruleright{%
        \forallruleright{
          \existsruleright{
            \impliesruleright{
              \weakeningruleright{
                \verboseaxiomruleright{
                  \py \vdash \py
                }
              }
              {
                \py \vdash \py \comm \allypy
              }
            }
            {
              \vdash \py \comm \predrinkerformula
            }
          }
          {
            \vdash \py \comm \drinkerformula
          }
        }
        {
          \vdash \allypy \comm \drinkerformula
        }
      }
      {
        \py \vdash \allypy \comm \drinkerformula
      }
    }
    {
      \vdash \predrinkerformula \comm \drinkerformula
    }
  }
  {
    \vdash \drinkerformula \comm \drinkerformula
  }
}
{
  \vdash \drinkerformula
}
\end{math}}\end{center}\vspace{-2ex}\caption{A syntactic proof of $\drinkerformula$, in Gentzen's sequent calculus \LK \cite{Gen35}.\label{fig:lk-drinker-proof}}\end{figure*}}

\newcommand\fateq{\mkern2mu=\mkern2mu}

\newcommand{\shortmapsto}{\psscalebox{.8 1}{\mathbin{\ensuremath\mapsto}}}
\newcommand\assign[2]{#1\mkern.5mu\shortmapsto\mkern.5mu #2}
\newcommand\assignopen{\{\mkern.5mu}
\newcommand\assignclose{\mkern.5mu\}}

\newcommand\assignment[1]{\assignopen #1 \assignclose}

\newcommand\literal{l}

\newcommand\literala{k}

\newcommand\unirel{\approx}
\newcommand\unirelof[1]{\unirel_{#1}}

\newcommand\termoratom{e}

\newcommand\parag[1]{\paragraph*{#1.}\hspace{-1ex}}
\newcommand\nec{\Box}
\newcommand\pos{\raisebox{-1pt}{\(\Diamond\)}}

\newcommand\openmodaldrinkerformula{p \tightimplies \nec p}
\newcommand\modaldrinkerformula{\pos(\openmodaldrinkerformula)}
\newcommand\restrsym{\mathop{\mkern-3mu\restriction\mkern-1mu}}
\newcommand\restr[2]{#1\restrsym{_{\mkern-2mu#2}}}

\newcommand\dualitiessym{\bot}
\newcommand\dualitiesof[1]{\dualitiessym_{#1}}
\newcommand\dualityedgesof[1]{\dualitiesof{#1}}
\newcommand\dualizinggraphpairof[1]{\graphpairof{\verticesof{#1}}{\dualitiesof{#1}}}%
\newcommand\dualities{\dualitiessym}

\newcommand\figrule{\vspace{1.3ex}\hrule}%

\newcommand\portion{P}
\newcommand\portionp{\portion\primed}
\newcommand\dualportion{\portion^*}
\newcommand\dualportionp{\portionp^*}
\newcommand\fusion{F}

\newcommand\tightf{\mkern-1mu f\mkern-1mu}

\newcommand\combulavar{G}

\newcommand\sequent{\Gamma}
\newcommand\sequenta{\Delta}

\newcommand\gsone{\calculusfont{GS1}\xspace}

\newcommand\sequentcom{\sequent\mkern-3mu\com}
\newcommand\sequentacom{\sequenta\mkern-1mu\com}

\newcommand\neighbourhoodof[1]{N(#1)}

\newcommand\rectifiedformulaeg[2]  {(\px\tightvee\exists y\mkern2mu\qy) #1 \exists z\mkern2mu\rz#2}
\newcommand\unrectifiedformulaeg[2]{(\px\tightvee\exists x\mkern2mu\qx) #1 \exists x\mkern2mu\rx#2}

\newcommand\term{t}

\newcommand\dualizer{\delta}
\newcommand\dualizerp{\delta\primed}
\newcommand\dualizera{\gamma}
\newcommand\occs{\omega}

\newcommand\subst{\sigma}

\newcommand\stem{\zone}%
\newcommand\stema{\ztwo}%

\newcommand\stemeg{
\namedsingletonleft{x}{x}\hspace*{1.2ex}\singletonred{\px}\e{x}v
\hspace*{1.2ex}
\namedsingletonleft{y}{y}\hspace*{1.2ex}\singletonred{\ppy}\e{y}v
\hspace*{1.2ex}
\singletonz
}

\newcommand\stemegmin{\assignment{\assign x \stem,\assign y \stem}}
\newcommand\stemeguniversal{\assignment{\assign x z,\assign y z}}

\newcommand\genegvia{\assignment{\assign\stem{fza}}}

\newcommand\dep{\{\singletonx,\singletony\}}

\newcommand\isvalid{\models}

\newcommand\occsubst[3]{\assignopen #1\mapsto_{#2}#3\assignclose}

\newcommand\megagraph{\Omega}

\newcommand\xgraphofsymbol{\mathbb{G}}
\newcommand\xgraphof[1]{\xgraphofsymbol(#1)}

\newcommand\xgraphall[2]{\singleton{#1}\graphunion\xgraphof{#2}}
\newcommand\xgraphex [2]{\singleton{#1}\graphjoin\xgraphof{#2}}

\newcommand\weakeningsymbol{\mathsf{W}}
\newcommand\slackeningsymbol{\mathsf{S}}
\newcommand\contractionsymbol{\mathsf{C}}
\newcommand\contractionof[1]{\contractionsymbol_{#1}}
\newcommand\weakeningof[2]{\weakeningsymbol_{#1}^{#2}}
\newcommand\slackeningof[2]{\slackeningsymbol_{#1}^{#2}}

\newcommand\im[1]{\textsf{Im}\mkern2mu#1}

\newcommand\childsym{\prec}%
\newcommand\parentsym{\succ}%
\newcommand\child{\mkern-2mu\childsym\mkern-3mu}
\newcommand\childp{\mkern-2mu\childsym\mkern-5.3mu\primed}
\newcommand\parent{\mkern-2mu\parentsym\mkern-3mu}
\newcommand\childin[1]{\mkern-2mu\childsym_{\mkern-1mu #1}\mkern-3mu}
\newcommand\parentin[1]{\mkern-2mu\parentsym_{\mkern-1mu #1}\mkern-2mu}
\newcommand\belowin[1]{\mkern-2mu<_{\mkern-1mu #1}\mkern-2mu}
\newcommand\atorbelowin[1]{\mkern-2mu\le_{\mkern-1mu #1}\mkern-2mu}
\newcommand\atorabovein[1]{\mkern-2mu\ge_{\mkern-1mu #1}\mkern-2mu}
\newcommand\abovein[1]{\mkern-2mu>_{\mkern-1mu #1}\mkern-2mu}

\newcommand\meet{\mathbin{\mkern-1mu\odot\mkern-1mu}}
\newcommand\jmeet{\mathbin{\mkern-1mu\otimes\mkern-1mu}}
\newcommand\umeet{\mathbin{\mkern-1mu\oplus\mkern-1mu}}

\newcommand\reflem[1]{Lemma\,\ref{#1}\xspace}

\newcommand\scopeofin[2]{\textsf{S}_{#2}(#1)}

\newcommand\node{n}
\newcommand\nodea{m}
\newcommand\nodeaa{o}

\newcommand\plustimes{$\graphunion\mkern-2mu\graphjoin\mkern-1mu$}
\newcommand\plustimestreestem{\plustimes{} tree}
\newcommand\plustimestree{\plustimestreestem\xspace}
\newcommand\plustimestrees{\plustimestreestem s\xspace}

\newcommand\plustimesnodestem{\plustimes{} node}
\newcommand\plustimesnode{\plustimesnodestem\xspace}
\newcommand\plustimesnodes{\plustimesnodestem s\xspace}

\newcommand\nodeset{N}
\newcommand\nodesetp{N\primed}

\newcommand\defaultplustimestree{\digraphpairof{\nodeset}{\child\,}}
\newcommand\defaultplustimestreep{\digraphpairof{\nodesetp}{\childp\,}}

\newcommand\cographofsymbol{\textsf{G}}
\newcommand\cographof[1]{\cographofsymbol(#1)}
\newcommand\tree{T}
\newcommand\treep{T\primed}

\newcommand\nodesof[1]{\nodeset_{#1}}
\newcommand\nodesetof[1]{\nodesof{#1}}
\newcommand\absorption[2]{#1\mkern1mu\raisebox{.3ex}{\ensuremath\uparrow}\mkern1mu#2}
\newcommand\absorptionsub[2]{#1\mkern1mu\raisebox{.45ex}{\ensuremath{{}_\uparrow}}\mkern1mu#2}

\newcommand\absorbed[1]{|#1|}

\newcommand\parentof[1]{\widehat{#1}}%
\newcommand\parentabovein[1]{\mathbin{\mkern1mu\rhd_{#1}}}
\newcommand\wsmap{\mbox{$\weakeningsymbol\slackeningsymbol$}-map\xspace}

\newcommand\defaultcotreeeg{\plustreesep{.9}{
  \timestreesep{.75}{\lf\lf}
    \lf
    \timestreesep{.6}{
      \lf
      \lf
      \tspace{-.235}
      \plustreesep{.6}{\lf\lf}}
  }}

\newcommand\markingof[1]{{#1}^*}

\newcommand\inlineaxiom[1]{\overline{#1\com\dual{#1}}}%

\newcommand\sequentsequence{\formula_1\com\ldots\com\formula_n}
\newcommand\formulaofsequent{\formula_1\tightvee(\formula_2\tightvee\ldots(\formula_{n-1}\tightvee\formula_n)\ldots)}

\newcommand\dcpcptwo{
  \newcommand\xxwidth{.9}
  \newcommand\xwidth{.3}
  \rput(0,\cpfibheight){
    \redvx{-\xxwidth,0}{A}
    \bluevx{\xwidth,0}{B}
    \bluevx{\xxwidth,-\pfycoverradius}{C1}
    \redvx{\xxwidth,\pfycoverradius}{C2}
  }
  \lvxd{-\xxwidth,0}{a}{\pp}
  \lvxd{-\xwidth,-.14}{q}{q}
  \lvxd{\xwidth,0}{b}{\pp}
  \lvxd{\xxwidth,0}{c}{p}
  \e A B
  \e a b
  \e q b
  \fibrestyle
  \ncline {A}{a}
  \ncline {B}{b}
  \ncline {C2}{C1}
  \ncline {C1}{c}
}
\newcommand\dcptwo{
  \newcommand\xxwidth{.9}
  \newcommand\xwidth{.3}
  \rput(0,\cpfibheight){
    \vx{-\xxwidth,0}{A}
    \vx{\xwidth,0}{B}
    \vx{\xxwidth,-\pfycoverradius}{C1}
    \vx{\xxwidth,\pfycoverradius}{C2}
  }
  \vx{-\xxwidth,0}{a}
  \vx{-\xwidth,-.14}{q}
  \vx{\xwidth,0}{b}
  \vx{\xxwidth,0}{c}
  \e A B
  \e a b
  \e q b
  \fibrestyle
  \ncline {A}{a}
  \ncline {B}{b}
  \ncline {C2}{C1}
  \ncline {C1}{c}
  \dualitystyle
  \nccurve[angleA=10,angleB=140,dash=3pt 1.5pt] B {C1}
  \nccurve[angleA=25,angleB=160] A {C2}
  \nccurve[angleA=25,angleB=155,dash=3pt 1.5pt] b c
  \nccurve[angleA=-50,angleB=-150] a c
}

\newcommand\tightminus{\mkern-3mu-\mkern-3mu}

\newcommand\vrel[1]{\simeq_{#1}}

\newcommand\monadiclink{\{\mkern2mu\singletonblue\px,\singletonblue\qy\mkern2mu\}}

\newcommand\dualizinggraph{D}
\newcommand\dualizinggrapha{C}
\newcommand\dualizinggraphofsymbol{\mathbb{D}}
\newcommand\dualizinggraphof[1]{\dualizinggraphofsymbol(#1)}
\newcommand\dualizinggraphofpropsymbol{\mathcal{D}}
\newcommand\dualizinggraphofprop[1]{\dualizinggraphofpropsymbol(#1)}

\newcommand\dualizinggraphofgraph[1]{\graphpairof{\verticesof{#1}}{\dualitiesof{#1}}}

\newcommand\dualizingcover{\dualizinggrapha}
\newcommand\dualizingbase{\dualizinggraph}

\newcommand\bindingsymbol{B}
\newcommand\bindingsof[1]{\bindingsymbol_{#1}}
\newcommand\bindingedgesof[1]{\bindingsof{#1}}

\newcommand\bindinggraphofgraph[1]{\graphpairof{\verticesof{#1}}{\bindingsof{#1}}}

\newcommand\mograph{M}
\newcommand\mographa{N}
\newcommand\mographp{\mograph\primed}
\newcommand\mographofsymbol{\mathbb{M}}
\newcommand\mographof[1]{\mographofsymbol(#1)}
\newcommand\mographofformulasymbol{\mathcal{M}}
\newcommand\mographofformula[1]{\mographofformulasymbol(#1)}
\newcommand\modalmographofsymbol{\mathscr{M}}
\newcommand\modalmographof[1]{\modalmographofsymbol(#1)}
\newcommand\linkedmographofsymbol{\Lambda}
\newcommand\linkedmographof[1]{\linkedmographofsymbol(#1)}

\newcommand\monadicformulaeg{\forall x
 \left(\rule{0ex}{1.7ex}
  (\px\tightwedge\px)
  \tightvee
   \exists y\mkern2mu\ppy
 \right)}

\newcommand\prop{\formula}%

\newcommand\cbg{K}

\newcommand\brel[1]{\vrel{#1}}%

\newcommand\wedgeorvee{\ast}
\newcommand\wedgeorveeinset{\wedgeorvee\tightin\setof{\wedge\mkern1mu,\mkern-3mu\vee}}

\newcommand\pfour{P_{\mkern-3mu 4}}
\newcommand\pfourvertexlist{\vertex_1,\mkern-2mu\vertex_2,\mkern-2mu\vertex_3,\mkern-2mu\vertex_4}
\newcommand\pfourvertices{\setof{\pfourvertexlist}}
\newcommand\pfouredges{\setof{\vertex_1\vertex_2,\mkern-1mu\vertex_2\vertex_3,\mkern-1mu\vertex_3\vertex_4}}
\newcommand\pfourgraph{{\newcommand\gap{\mkern4mu}\namedvx a \gap \namedvx b\e a b\gap \namedvx c\e b c\gap \namedvx d\e c d}}

\newcommand\cthree{C_3}
\newcommand\cthreevertexlist{\vertex_1,\mkern-2mu\vertex_2,\mkern-2mu\vertex_3}
\newcommand\cthreevertices{\setof{\cthreevertexlist}}
\newcommand\cthreeedges{\setof{\vertex_1\vertex_2,\mkern-1mu\vertex_2\vertex_3,\mkern-1mu\vertex_3\vertex_1}}
\newcommand\cthreegraph{\raisebox{-2pt}{\(\newcommand \gap{\mkern2mu} \namedvx a \gap \rput(0,.2){\namedvx b} \gap \namedvx c\e a c \e a b \e b c\)}}

\newcommand\vertexrenaming{\hat}%

\newcommand\forallorexists{\nabla}
\newcommand\forallorexistsinset{\forallorexists\in\setof{\forall\mkern-1mu,\mkern-1mu\exists}}

\newcommand\monet{\net}

\newcommand\equivclassof[2]{[#1]_{#2}}
\newcommand\quotientsymbol{\mathop{\mkern-3mu/\mkern-4mu}}
\newcommand\quotientof[2]{{#1\quotientsymbol#2}}
\newcommand\quotientgraphof[2]{\quotientof{#1}{#2}}

\newcommand\indist{\asymp}
\newcommand\collapseof[1]{#1_{\indist}}%

\newcommand\literaltype{\ast}
\newcommand\universaltype{\forall}
\newcommand\existentialtype{\exists}
\newcommand\vertextypes{\setof{\literaltype,\universaltype,\existentialtype}}
\newcommand\vertextypein[1]{\mathsf{typ}_{#1}}
\newcommand\typeof[2]{\vertextypein{#2}(#1)}

\newcommand\R{\calculusfont{LKR}\xspace}

\newcommand\skel[1]{{#1}_{\textsf{s}}}

\newcommand\tightcirc{\mathbin{\mkern-3mu\circ\mkern-3mu}}

\title{\vspace*{-4.5ex}\LARGE First-order proofs without syntax}
\author{\\[-2.5ex]
\large Dominic J.\,D.\,Hughes\\[1ex]
\small Stanford University \& U.C.\ Berkeley\thanks{%
  I conducted this research as a Visiting Scholar at Stanford then
  Berkeley. Many thanks to my hosts, Vaughan Pratt (Stanford Computer
  Science), Sol Feferman (Mathematics) and Wes Holliday
  (Berkeley Logic Group).
  I am extremely grateful for very helpful feedback from
  Willem Heijltjes, Lutz Stra\ss burger, Grisha Mints, Sam Buss, Martin Hyland, Marc Bagnol and Nil Demir\c{c}ubuk.
  In memoriam Sol and Grisha.}%
\date{\small\cleanlookdateon\today}}

\begin{document}

\maketitle\vspace{-4ex}

\begin{center}\begin{minipage}{4.75in}
\hspace{3ex}Proofs are traditionally syntactic, inductively generated objects.
This paper reformulates first-order logic (predicate calculus) with
proofs which are
graph-theoretic
rather than syntactic.
It defines a \textsl{combinatorial proof} of a formula $\formula$ as a lax fibration over a graph associated with $\formula$.
The main theorem is soundness and completeness: a formula is a valid if and only if it has a combinatorial proof.
\end{minipage}\end{center}

\section{Introduction}\label{sec:intro}

Proofs are traditionally syntactic, inductively generated objects.
For example,
Fig.\,\ref{fig:lk-drinker-proof} shows a syntactic proof of\note{the formula} $\drinkerformula$.
This paper reformulates first-order logic (predicate calculus) \cite{Fre} with proofs which are graph-theoretic rather than syntactic.
It defines a \textsl{combinatorial proof} of a formula $\formula$ as a lax graph fibration $\cp$ over a graph $\gformula$ associated with $\formula$, where $\cover$ is a partially coloured graph.
For example, if $\formula = \drinkerformula$
then $\gformula$ is
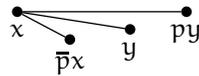
\drinkerFormulaDisplayed
and a combinatorial proof $\cp$ of $\formula$ is
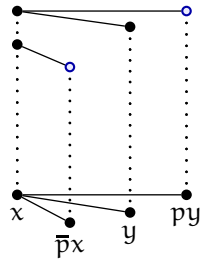
\begin{center}\drinkerDisplayed\end{center}
The upper
graph is $\cover$
(two coloured vertices
$\singletonblue{}\singletonblue{}$
and three uncoloured vertices),
the lower graph is $\gformula$,
and the dotted lines define $\skewfib$.
Additional combinatorial proofs are depicted in Fig.\,\ref{fig:cps}.
The
combinatorial proof $\cp$ above can be condensed
by leaving $\gformula$ implicit and drawing
$\cover$
over the formula $\formula$:
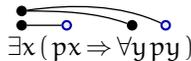
\begin{center}\drinkerInlineDisplayedPic\end{center}
The reader may contrast this with the
syntactic proof of the same formula in Fig.\,\ref{fig:lk-drinker-proof}.
The four combinatorial proofs of Fig.\,\ref{fig:cps} are rendered in condensed form in Fig.\,\ref{fig:cps-condensed}.

The main theorem of this paper is soundness and completeness: a formula is valid if and only if it has a combinatorial proof (Theorem\,\ref{thm:soundness-completeness}).\figurelkdrinkerproof{}
The propositional fragment was presented in \cite{Hug06}.

\figcps{}\figcpscondensed{}

\section{Notation and terminology}\label{sec:notation}

\parag{First-order logic}
We mostly follow the notation and terminology of \cite{Joh87} for first-order logic without equality \cite{Fre}.\todo{comment 0 1}
Terms and atoms (atomic formulas) are generated
 inductively from variables $x$, $y$, $z,\ldots$ by: if $\gamma$ is
an $n$-ary
function (resp.\ predicate) symbol
and $t_1,\ldots,t_n$ are terms then $\gamma t_1\ldots t_n$ is a term (resp.\ atom).
We extend the set of atoms with the logical constants $1$ (true) and $0$ (false).
For technical convenience we assume every predicate symbol $p$ is assigned a \defn{dual} predicate symbol $\pp$ with $\pp\tightneq p$ and $\ppp\tighteq p$, and extend duality to atoms with $\dual{p\rule{0ex}{1.1ex}\likex{t_1\ldots t_n}}=\pp t_1\ldots t_n$,
$\dual0=1$ and $\dual1=0$.
\defn{Formulas} are generated from atoms by binary
$\wedge$
and $\vee$
and
quantifiers $\forall x$
and $\exists x$
per variable $x$.
Define $\neg$
and $\implies$
as abbreviations:
$\neg(\alpha)=\dual\alpha$ on atoms $\alpha$,
$\neg(\formula\tightwedge\formulaa)=(\neg\formula)\vee(\neg\formula)$,
$\neg(\formula\tightvee\formulaa)=(\neg\formula)\wedge(\neg\formulaa)$,
$\neg\mkern2mu\forall x\mkern2mu \formula=\exists x\mkern2mu\neg\mkern1mu \formula$,
$\neg\mkern2mu\exists x\mkern2mu \formula=\forall x\mkern2mu\neg\mkern1mu \formula$,
and
$\formula\implies\formulaa=(\neg \formula) \vee \formulaa$.
A formula is \defn{rectified} if all bound variables are distinct from one another and from all free variables, \eg
\mbox{$\rectifiedformulaeg\tightwedge{}$}
but not
\mbox{$\unrectifiedformulaeg\tightwedge{}$}.
Throughout this paper we assume all formulas are rectified
(losing no generality since every unrectified formula has a logically equivalent rectified form).

\vspace{-2ex}\parag{Graphs}
An \defn{edge} on a set $\vertices$ is a two-element subset of $\vertices$.
A \defn{graph} $\graphpair$ is a finite
set $\vertices$ of \defn{vertices} and a set $\edges$ of edges on $\vertices$.
Write $\verticesof\graph$ and $\edgesof\graph$ for the vertex and edge sets of a graph $\graph$, and
$\vertex\vertexa$ for $\{\mkern1mu\vertex,\mkern-2mu\vertexa \mkern1mu\}$.
The \defn{complement}  of $\graphpair$ is the graph $\graphpairof{\vertices}{\edges^{\mkern-1mu\mathsf c}}$ with $vw\tightin\edges^{\mkern-1mu\mathsf c}$ if and only if $vw\tightnotin\mkern-2mu\edges$.
A graph $\graph$ is (partially) \defn{coloured} if it carries a partial equivalence relation $\colourequiv$ on $\verticesof\graph$ such that $\vertex\colourequiv\vertexa$ only if $\vertex\vertexa\tightnotin\edgesof\graph$; each equivalence class is a \defn{colour}.
A graph is \defn{labelled} in a set $L$ if each vertex has an element of $L$ associated with it, its \defn{label}.
A \defn{vertex renaming} of $\graphpair$ along a bijection
 $(\vertexrenaming{\hspace{1ex}}):\vertices\to\verticesp$ is the graph $\graphpairof\verticesp{\setst{\vertexrenaming\vertex\mkern2mu\vertexrenaming\vertexa}{\vertex\mkern1mu\vertexa\tightin\edges}}$, with colouring and/or labelling inherited (\ie, $\vertexrenaming\vertex\colourequiv\vertexrenaming\vertexa$ if $\vertex\colourequiv\vertexa$, and the label of $\vertexrenaming\vertex$ that of $\vertex$).
Following standard graph theory, we identify graphs modulo vertex renaming.
Let $\graph\tighteq\mkern3mu\graphpair$ and $\graphp\tighteq\mkern3mu\graphpairp$ be graphs.
A \defn{homomorphism} $\homom:\graph\to\graphp$ is a function \mbox{$\homom:\vertices\mkern-2mu\to\mkern-2mu\verticesp$} such that if
$\vertex\vertexa\mkern-2mu\in\mkern-2mu\edges$ then $\homomof\vertex\mkern1mu\homomof\vertexa\mkern-2mu\in\mkern-2mu\edgesp$.
Without loss of generality, assume
\mbox{$\vertices\cap\verticesp=\emptyset$} (by renaming vertices if needed).
The \defn{union} $\guniongp$ is \mbox{$\graphpairof{\mkern2mu\vertices\cup\verticesp\mkern-2mu}{\edges\cup\edgesp\mkern2mu}$}
and
\defn{join} \mbox{$\gjoingp$} is \mbox{$\graphpairof{\vertices\cup\verticesp\mkern-2mu}{\edges\cup\edgesp\cup\setst{\vertex\vertexp}{\vertex\tightin\vertices,\vertexp\tightin\verticesp}}$};
any colourings or labellings are inherited.
$\graph$ is \defn{disconnected} if $\graph=\graph_1\graphunion\graph_2$ for graphs $\graph_i$, else \defn{connected},
and \defn{coconnected} if
its complement
is connected.
The subgraph of
$\graphpair$ \defn{induced} by $W\tightsubseteq\vertices$ is $\graphpairof W {\restr \edges W}$ for $\restr \edges W$ the restriction of $\edges$ to edges on $W$.
A graph is \defn{$\graph$-free} if $\graph$ is not an induced subgraph.
A \defn{cograph} is a $\pfour$-free graph, where
$\pfour=\pfourgraph=\graphpairof{\pfourvertices}{\pfouredges}$.
In
$\graphpair$ the \defn{neighbourhood} $\neighbourhoodof\vertex$ of $\vertex\tightin\vertices$ is $\setsuchthat{\vertexa}{\vertex\vertexa\tightin\edges}$,
a \defn{module}
is a set $\module\tightsubseteq\vertices$
such that $\neighbourhoodof\vertex\tightsetminus\module = \neighbourhoodof\vertexa\tightsetminus\module$ for all $\vertex,\mkern-2mu\vertexa\tightin\module$,
and $\module$ is \defn{strong} if
every module $\modulep$ satisfies $\modulep\mkern-3mu\cap\mkern-1mu\module\mkern-2mu=\mkern-2mu\emptyset$, $\modulep\mkern-1mu\tightsubseteq\module$ or $\modulep\tightsupseteq\module$.
A \defn{directed graph} $\digraphpairof\vertices\edges$ is a set
$\vertices$ of vertices
and a set
$\edges\tightsubseteq\vertices\mkern-4mu\tighttimes\mkern-4mu\vertices$ of \defn{directed edges}.
A directed graph \defn{homomorphism} $\homom:\digraphpairof\vertices\edges\mkern-2mu\to\mkern-2mu\digraphpairof\verticesp\edgesp$ is a function $\homom:\vertices\mkern-3mu\to\mkern-3mu\verticesp$ such that $\diedge\vertex{\mkern-2mu\vertexa}\tightin\edges$ implies
$\diedge{\homom(\vertex)}{\homom(\vertexa)}\tightin\edgesp$.

\section{Fographs (first-order cographs)}\label{sec:fographs}

A cograph is \defn{logical} if every vertex is labelled by a variable or atom, and it has at least one atom-labelled vertex.
Write $\singleton\tag$ for a $\tag$-labelled vertex.
\begin{definition}\label{def:graph}\label{def:graph-of-formula}
The \defn{graph} $\gformula$ of a formula
$\formula$ is the logical cograph defined
inductively by:\footnote{$\graphofsymbol$ is a first-order extension of the propositional translation \textsl{G} of \cite[\S3]{Hug06}.
The latter is well-known in graph theory, as the function from a (prime-free) modular decomposition tree \cite{Gal67} or cotree \cite{Ler71,CLS81} to a cograph, and is employed in logic and category theory, \eg\ \cite{Gir87,Hu99,Ret03}. See \S\ref{sec:related} for details.}

\begin{center}\vspace{0ex}\begin{math}
\begin{array}{c}
\graphof{\atom}
\;=\;
\singleton\atom \hspace{1ex} \text{ for every atom\/ }\atom
\\[2ex]
\def\eqgap{\hspace{3ex}}
\begin{array}{r@{\eqgap=\eqgap}l}
\graphof{\,\formula\tightvee\formulaa\,} & \gformula\graphunion\gformulaa
\\[1.5ex]
\graphof{\,\formula\tightwedge\formulaa\,} & \gformula\graphjoin\gformulaa
\\[1.5ex]
\end{array}
\hspace{16ex}
\begin{array}{r@{\eqgap=\eqgap}l}
\graphof{\,\forall x\, \formula\,} & \graphall {x\,} {\formula}
\\[1.5ex]
\graphof{\,\exists x\, \formula\,} & \graphex  {x\,} {\formula}
\\[1.5ex]
\end{array}\\[3ex]\end{array}\end{math}
\end{center}
\end{definition}
For example, $\veedrinkerformula$ and $\variantveedrinkerformula$
have the same graph $\drinkergraph$:
{\newcommand\gap{\hspace{3.5ex}}\begin{align*}
\\[-1ex]
\drinkergraph \gap&=\gap \graphofsymbol\mkern4mu\left(\rule{0ex}{1.7ex}\mkern6mu\veedrinkerformula\mkern6mu\right)        \\[2ex]
              &=\gap \graphofsymbol\mkern4mu\left(\rule{0ex}{1.7ex}\mkern6mu\variantveedrinkerformula\mkern6mu\right) \gap =
\rput(1.3,.1){
\drinkersquare
\e x y
\e x {px}
\e x {py}
}
\hspace{16ex}\\[2ex]
\end{align*}}%
Vertices of $\graphof\formula$ correspond to occurrences of atoms and quantifiers in $\formula$:
each occurence of an atom $\atom$ in $\formula$ becomes an $\alpha$-labelled vertex, and each occurrence of a quantifier $\forall x$ or $\exists x$ becomes an $x$-labelled vertex.
A \defn{literal} is an atom-labelled vertex and a \defn{binder} is a variable-labelled vertex.
Thus $\drinkergraph$ has two literals,
$\singletonppx$ and $\singletonpy$,
and two binders,
$\singletonx$ and $\singletony$ (obtained from $\exists x$ and $\forall y$).

A module is \defn{proper}\label{sec:proper} if it has two or more vertices.
The \defn{scope} of a binder $\binder$
is the smallest proper strong module containing $\binder$.\footnote{Since, by definition, every logical cograph has a literal, the requisite strong module in the scope definition exists.}\textsuperscript{,}\footnote{To discern scope it is helpful to draw the modular decomposition tree \cite{Gal67}, \ie, cotree \cite{CLS81}. See Lemma\,\ref{lem:parent-scope}.}
For example, in $\drinkergraph$,
the scope of $\singletony$ is $\setof{\singletony,\singletonppx,\singletonpy}$,
and
the scope of $\singletonx$ is $\setof{\singletonx,\singletony,\singletonppx,\singletonpy}$, illustrated below by shading.
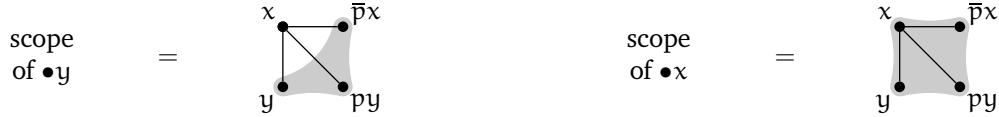
\begin{pic}{-1}{1}
\newcommand\greyradius{.75ex}
\newcommand\greydiameter{1.5ex}
\newcommand\greycolour{black!20!white}
\newcommand\greycircle[2]{\Cnodeput*[fillcolor=\greycolour,linecolor=\greycolour,radius=\greyradius,framesep=0pt](#1){#2}{}}
\rput(-2.5,0){%
\rput(-3,0){\begin{array}{c}\text{scope}\\\text{of }\singletony\end{array}\hspace{5ex}=}
\drinkersquare
{
  \psset{linewidth=\greydiameter}
  \psset{linecolor=\greycolour,linecap=2}
  \greycircle{\halfedgelen,\halfedgelen}{px}
  \greycircle{-\halfedgelen,-\halfedgelen}{y}
  \greycircle{\halfedgelen,-\halfedgelen}{py}
  \Cnodeput*[fillcolor=\greycolour,linecolor=\greycolour,radius=\greydiameter,framesep=0pt](\quarteredgelen,-\quarteredgelen){centre}{}
  \nccurve[angleA=-100,angleB=100,nodesep=0,linecap=2] {px}{py}
  \nccurve[angleA=10,angleB=170,nodesep=0,linecap=2] y {py}
  \nccurve[angleA=20,angleB=-110,nodesep=0,linecap=2] y {px}
}
\drinkersquareunlabelled
\e x y
\e x {px}
\e x {py}
}
\rput(5.7,0){%
\rput(-3,0){\begin{array}{c}\text{scope}\\\text{of }\singletonx\end{array}\hspace{5ex}=}
\drinkersquare
{
  \psset{linewidth=\greydiameter}
  \psset{linecolor=\greycolour,linecap=2}
  \greycircle{-\halfedgelen,\halfedgelen}{x}
  \greycircle{\halfedgelen,\halfedgelen}{px}
  \greycircle{-\halfedgelen,-\halfedgelen}{y}
  \greycircle{\halfedgelen,-\halfedgelen}{py}
  \Cnodeput*[fillcolor=\greycolour,linecolor=\greycolour,radius=\greydiameter,framesep=0pt](\quarteredgelen,-\quarteredgelen){centre}{}
  \e x {py}
  \e y {px}
  \nccurve[angleA=-100,angleB=100,nodesep=0,linecap=2] {px}{py}
  \nccurve[angleA=10,angleB=170,nodesep=0,linecap=2] y {py}
  \nccurve[angleA=20,angleB=-110,nodesep=0,linecap=2] y {px}
  \nccurve[angleA=70,angleB=-160,nodesep=0,linecap=2] y {px}
  \nccurve[angleA=-80,angleB=80,nodesep=0,linecap=2] x y
  \nccurve[angleA=-10,angleB=-170,nodesep=0,linecap=2] x {px}
}
\drinkersquareunlabelled
\e x y
\e x {px}
\e x {py}
}
\end{pic}
A binder $\binder$ is \defn{existential} (resp.\ \defn{universal}) in a logical cograph $\fograph$ if, for every other vertex $\vertex$ in the scope of $\binder$, we have $\binder\vertex\tightin\edgesof\fograph$ (resp.\ $\binder\vertex\tightnotin\edgesof\fograph$).\footnote{Since the scope of a binder is a proper strong module, every binder is either universal or existential (and not both).}
In $\drinkergraph$, for example,
$\singletonx$ is existential
and
$\singletony$ is universal
(corresponding to
$\exists x$
and
$\forall y$
in the formula(s) generating $\drinkergraph$).
An \defn{$\variable$-binder} is a binder with variable $\variable$, which is \defn{legal} if its scope contains at least one literal and no other $\variable$-binder.
\begin{definition}\label{def:fograph}
  A \defn{fograph}\/ or \defn{first-order cograph}\/ is a logical cograph whose binders are legal.
\end{definition}
For example, $\drinkergraph$ above is a fograph, but
\(\;
\namedsingletonleft x x
\hspace*{.8ex}
\namedsingletonright y y
\hspace*{.8ex}
\singletonright p
\e x y
\;\)
is not (since neither binder scope contains a literal),
nor is
\(\;
\namedsingletonright a x
\hspace*{.8ex}
\namedsingletonright b x
\hspace*{.8ex}
\singletonright {px}
\)\;
(since each $x$-binder is in the other's scope).
\begin{lemma}\label{lem:translation}
  The graph $\graphof\formula$ of every formula $\formula$ is a fograph.
\end{lemma}
\begin{proof}
  By structural induction on $\formula$. The base case with $\formula$ an atom is immediate.\todo{ensure def ref correct}
  For the induction step, note that all four operations defined in Def.\,\ref{def:graph-of-formula} preserve the property of being a fograph, since all formulas are rectified.\footnote{Naively applying $\graphofsymbol$ to an unrectified formula such as
$(\forall\mkern-1mu x p\mkern-2mu x\mkern-1mu)\tightvee (\forall\mkern-1mu x (\mkern-2mu q\mkern-2mu x\mkern-2.5mu\tightvee\mkern-2mu r\mkern-1mu x\mkern-1mu)\mkern-2mu)$
yields
$\:\newcommand\gap{\hspace{.8ex}}
\singleton{\mkern-1mu x}
\gap
\singleton{\mkern-1mu p\mkern-2mu x}
\gap
\singleton{\mkern-1mu x}
\gap
\singleton{\mkern-1mu q\mkern-2mu x}
\gap
\singleton{\mkern-1mu r\mkern-1.5mu x}
\:$
with all three literals bound ambiguously by both binders. Whence our assumption that every formula be in rectified form.}\todo{``place a formula'' in footnote}
\end{proof}
An \defn{$\variable$-literal} is one whose atom contains the variable $\variable$.
An $\variable$-binder \defn{binds} every $\variable$-literal in its scope.
In $\drinkergraph$ above, for example, $\singletonx$ binds $\singletonppx$ and $\singletony$ binds $\singletonpy$.
An $\variable$-binder is \defn{rectified} if it is the only $\variable$-binder and its scope contains every $\variable$-literal.
A fograph is \defn{rectified} if its binders are rectified.\footnote{In \S\ref{sec:soundness} we will observe that $\graphofsymbol$ (Def.\,\ref{def:graph-of-formula}) is a surjection onto rectified fographs (\reflem{lem:graph-surj}).}
For example,
$\drinkergraph$ above is rectified
but $\:\cleaningegone\:$ is not (since it has two $x$-binders), nor is $\:\cleaningegtwo\:$
(since
$\singletonx$ does not bind
$\singletonqx$).
To \defn{rectify} an unrectified \mbox{$\variable$-binder} $\binder$ in a fograph $\fograph$ is to
change its label to a variable $\xp$ which is fresh (\ie, not in any label of $\fograph$)
and
substitute $\xp$ for $x$ in the
label of every literal
bound by $\binder$.
A \defn{rectified form} is any result of rectifying binders until reaching a rectified fograph.
For example, \;$\bigcleaningeg{x}{x}$\: has the rectified form \;$\bigcleaningeg{y}{z}$\:.
This is analogous to the unrectified formula
\mbox{$\unrectifiedformulaeg\tightvee x$} having the rectified form
\mbox{$\rectifiedformulaeg\tightvee z$}.

The \defn{binding graph} $\bindinggraphof\graph$ of a fograph $\graph$ is the
directed graph
\mbox{$\digraphpairof{\verticesof\graph}{\,\setof{\diedge{\binder}{\mkern-2mu\literal}:\binder\text{ binds }\literal}\,}$}.
For example, the binding graph of $\drinkergraph$ above is
\begin{pic}{-1}{.7}
\rput(-1.5,0){\bindingrelof\drinkergraph\hspace{6ex}=}
\rput(1.3,0){\drinkersquare\de x {px}\de y {py}}
\end{pic}

\section{Skew bifibrations}

A directed graph homomorphism
$\fib:\digraphpair\to\digraphpairp$ is a \defn{fibration} \cite{Gro60,Gra66} if
for all
$\vertex\tightin\vertices$ and $\diedge{\vertexa}{\mkern-1mu\fib(\vertex)}\tightin\edgesp$ there exists a unique $\skewlifting\vertexa\tightin\vertices$ with $\diedge{\skewlifting\vertexa}{\mkern-2mu\vertex}\tightin\edges$ and $\fib(\skewlifting\vertexa) = \vertexa$.
This definition is illustrated below-left.
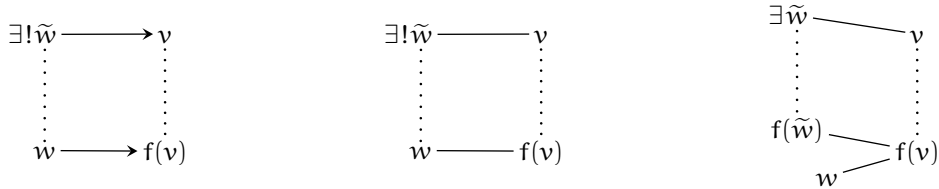
\begin{center}
 \begin{pspicture}[nodesep=2pt,labelsep=2pt](0,-.35)(0,2)
  \begin{math}
    \newcommand\rad{.8}
    \newcommand\shortrad{.4}
    \rput(-5,0){
      \rput(-\rad,1.55){\Rnode{liftw}{\hspace{-2ex}\likex{\exists\mkern1mu!\mkern1mu\skewlifting\vertexa}}}
      \rput(-\rad,0){\Rnode w \vertexa}
      \rput(\rad,1.55){\Rnode v v}
      \rput(\rad,0){\Rnode{fv}{\fib(v)}}
      \ncline[arrows=->]{w}{fv}%
      \ncline[arrows=->]{liftw}{v}%
      \fibrestyle
      \ncline{liftw}{w}
      \ncline{v}{fv}
    }
    \rput(0,0){
      \rput(-\rad,1.55){\Rnode{liftw}{\hspace{-2ex}\likex{\exists\mkern1mu!\mkern1mu\skewlifting\vertexa}}}
      \rput(-\rad,0){\Rnode w \vertexa}
      \rput(\rad,1.55){\Rnode v v}
      \rput(\rad,0){\Rnode{fv}{\fib(v)}}
      \ncline{w}{fv}%
      \ncline{liftw}{v}%
      \fibrestyle
      \ncline{liftw}{w}
      \ncline{v}{fv}
    }
    \rput(5,0){
      \rput(\rad,1.55){\Rnode{v}v}%
      \rput(\rad,0){\Rnode{fv}{f(v)}}
      \rput(-\rad,.3){\Rnode{fwhat}{f(\skewlifting w)}}
      \rput(-\rad,1.8){\Rnode{what}{\hspace{-1.4ex}\likex{\exists\mkern2mu\skewlifting\vertexa}}}
      \rput(-\shortrad,-.35){\Rnode{w}{w}}
      \ncline{v}{what}%
      \ncline{fv}{fwhat}
      \ncline{fv}{w}%
      \fibrestyle
      \ncline{v}{fv}
      \ncline{what}{fwhat}
    }
  \end{math}
 \end{pspicture}
\end{center}
Similarly, an undirected graph homomorphism $\fib:\graphpair\to\graphpairp$ is a \defn{fibration}
if for all
$\vertex\tightin\vertices$ and
${\vertexa}\mkern3mu{\fib(\vertex)}\in\edgesp$ there exists a unique $\skewlifting\vertexa\tightin\vertices$ with ${\skewlifting\vertexa}{\mkern2mu\vertex}\tightin\edges$ and $\fib(\skewlifting\vertexa) = \vertexa$.
This definition is illustrated above-centre.\footnote{An undirected graph fibration is a special case of a topological fibration \cite{Whi78}, by viewing every edge as a copy of the unit interval.}
An undirected graph homomorphism $\fib:\graphpair\to\graphpairp$ is a \defn{skew fibration} \cite{Hug06}
if for all
$\vertex\tightin\vertices$ and
${\vertexa}\mkern3mu{\fib\mkern-.3mu(\vertex)}\tightin\edgesp$ there exists
$\skewlifting\vertexa\tightin\vertices$ with
${\skewlifting\vertexa}\mkern2.5mu{\vertex}\in\edges$ and
$f(\skewlifting w)\mkern2mu w\notin\edgesp$.
This definition is illustrated above-right.
Since $\fib(\skewlifting\vertexa) \tighteq \vertexa$ implies
$f(\skewlifting w)\mkern2mu w\notin\edgesp$,
skew fibrations generalize fibrations.

A graph homomorphism $\fib:\cover\to\base$ between fographs \defn{preserves labels} if
for every vertex $\vertex\tightin\verticesof\cover$ the
label of $\vertex$ in $\cover$ equals the label of $\fib(\vertex)$ in $\base$,
and \defn{preserves existentials} if
for every existential binder $\binder$ in $\cover$ the vertex $\fib(\binder)$ is an existential binder in $\base$.
\begin{definition}
A \defn{skew bifibration} $\bifib\mkern-2mu:\mkern-2mu\cover\mkern-2mu\to\mkern-2mu\base$ between fographs is a label- and existential-preserving graph \mbox{homomorphism} such that
\begin{itemize}
  \item $\bifib:\cover\to\base$ is a skew fibration
  \item $\bifib:\bindinggraphof\cover\to\bindinggraphof\base$ is a fibration.
\end{itemize}
\end{definition}\begin{figure*}
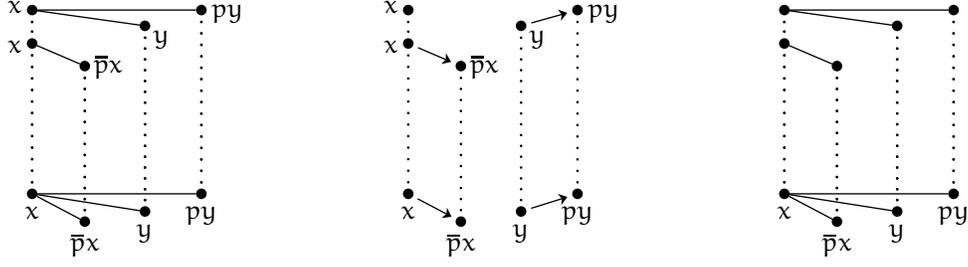
%
\begin{pic}{-1.1}{2.8}
  \rput(-5,0){\drinkerfiblabelledpair{x}{x}}
  \rput(0,0){\drinkerbindingfiblabelled{x}{x}}
  \rput(5,0){\drinkerfib}
\end{pic}%
\caption{\label{fig:drinkerbifib}A skew bifibration (left), its binding fibration (centre), and its skeleton (right).}\figrule\end{figure*}
We refer to $\bifib:\bindinggraphof\cover\to\bindinggraphof\base$ as the \defn{binding fibration}.
For example, a skew bifibration is shown in Fig.\,\ref{fig:drinkerbifib}, with its binding fibration.
The \defn{skeleton} of a skew bifibration is the result of dropping labels from its source.
Fig.\,\ref{fig:drinkerbifib} shows an example.
We identify a skew bifibration with its skeleton.
No information is lost since the source labels can be lifted from the target (because skew bifibrations preserve labels, by definition).\footnote{We need the explicit preservation of existentials in the definition of skew bifibration since that property does not follow from the other conditions.
For example, the unique label-preserving function from
$\graphof{\exists x\mkern2mu p}
=
\namedsingletonleft x x\hspace{1ex}\namedsingletonright p p\e x p$
to
$\graphof{(\forall x\mkern2mu q)\tightwedge p}
=
\namedsingletonleft x x\hspace{.8ex}\namedsingletonleft q q\hspace{1ex}\namedsingletonright p p\e q p\nccurve[nodesep=0pt,angleA=35,angleB=150]x p$
satisfies all the conditions of being a skew bifibration except existential preservation (since it maps an existential binder to a universal binder).}

\section{Fonets (first-order nets)}\label{sec:fonets}

Two atoms are \defn{pre-dual} if they have dual predicate symbols (\eg\ $\pxy$ and $\pp y fa$) and two literals are pre-dual if their atoms are pre-dual.
\begin{definition}\label{def:linked}
A \defn{linked fograph} is a coloured fograph such that
\begin{itemize}
\item every colour, called a \defn{link}, comprises two pre-dual literals, and
\item every literal is either 1-labelled or in a link.
\end{itemize}
\end{definition}
Fig.\,\ref{fig:leap}\todo{$p\shortmapsto q$ so validates cp?}
shows a linked fograph $\net$
with two links,
$\twolinkp$
and
$\twolinkq$.\begin{figure*}\begin{center}\vspace{9ex}
  \twolinkfograph
  \hspace{35ex}
  \twolinkleapgraph
  \vspace{7ex}\end{center}\caption{\label{fig:leap}A fonet $\net$ (left) with
    unique dualizer $\protect\twolinkassignment$
    and its
    leap graph $\protect\leapgraphof\net$ (right).}\figrule\end{figure*}
\begin{definition}\label{def:dualizer}
Let $\cover$ be a linked fograph.
\WLOG,
assume
$\cover$ is rectified (by rectifying binders as needed).
A \defn{dualizer} for $\cover$ is
a function $\dualizer$
assigning to each existential binder variable $x$ a term
$\dualizer(x)$
such that
for every link $\{\mkern2mu\singleton\atomone,\singleton\atomtwo\mkern2mu\}$, the atoms
$\atomone\dualizer$ and $\atomtwo\dualizer$ are dual, where $\atom\dualizer$ denotes the result of substituting $\dualizer(x)$ for $x$ throughout $\atom$
(simultaneously for each $x$).
\end{definition}
For example,
$\twolinkassignment$
is a dualizer\footnote{In the context of a function we write $\assign a b$ for the ordered pair $\langle a,b\rangle$.} for
$\net$ (Fig.\,\ref{fig:leap}) since
$\ppx\twolinkassignment=\ppz$ is dual to $\pz$, and $\qqy\twolinkassignment=\qqfz$ is dual to $\qfz$; this is the unique dualizer for $\net$.

A \defn{dependency}\todo{bond cord tether linkage link tie}
 $\dep$
of $\cover$
is an existential binder $\singletonx$ and a universal binder $\singletony$ such that every dualizer
for $\cover$
assigns to $x$ a term containing $y$.\footnote{In \S\ref{sec:ptime} we show that all dependencies can be constructed in polynomial time, despite quantification over \emph{every dualizer}.}\todo{link to p-time}
For example, $\{\mkern2mu\singletony,\singletonz\mkern2mu\}$ is a dependency of $\net$ (Fig.\,\ref{fig:leap})
since the unique dualizer $\twolinkassignment$ assigns $\fz$ to $y$.
A \defn{leap} is a dependency or link.
The \defn{leap graph} $\leapgraphof\cover$ is the graph $(\verticesof\cover,\leapsof\cover)$ where $\leapsof\cover$ comprises all leaps of $\cover$.
See Fig.\,\ref{fig:leap} for an example.

A graph $\graphpair$ is a \defn{matching} if $\vertices$ is non-empty and
for all $\vertex\mkern-3mu\in\mkern-2mu\vertices$ there is a unique $\vertexp\mkern-4mu\in\mkern-2mu\vertices$ with $\vertex\vertexp\mkern-3mu\in\mkern-2mu\edges$.
A set $W$ \defn{induces a bimatching} in a linked fograph $\cover$ if $W$ induces
a matching in $\cover$ and induces a matching in $\leapgraphof\cover$.
\begin{definition}
  A \defn{fonet} or \defn{first-order net} is a
 linked fograph which has a dualizer but no induced bimatching.
\end{definition}
See Fig.\,\ref{fig:leap} for an example of a fonet.
The minimal fonet is $\singleton 1$ (an uncoloured 1-labelled vertex).\footnote{A fonet can be viewed as a graph-theoretic abstraction and generalization of a unification net \cite{Hug18}. Upon forgetting vertex labels,
propositional fonets correspond to nicely coloured cographs \cite{Hug06},
which are in bijection with certain R\&B cographs \cite{Ret03}.
See \S\ref{sec:related} for details.}

\section{Combinatorial proofs}\label{sec:cps}

\begin{definition}\label{def:cp}
  A \defn{combinatorial proof} of a fograph\/ $\fograph$ is a skew bifibration\/ \mbox{$\bifib:\fographnet\to\fograph$}
  from a fonet\/ $\fographnet$.
  A combinatorial proof of a formula $\formula$ is a combinatorial proof of its graph $\gformula$.
\end{definition}
For examples, see \S\ref{sec:intro}.
\begin{theorem}[Soundness]
  A formula is valid if it has a combinatorial proof.
\label{thm:soundness}\end{theorem}
\begin{proof}
  Section\,\ref{sec:soundness}.
\end{proof}
\begin{theorem}[Completeness]
  Every valid formula has a combinatorial proof.
\label{thm:completeness}\end{theorem}
\begin{proof}
  Section\,\ref{sec:completeness}.
\end{proof}
Combining the two theorems above, we obtain the main theorem of this paper:
\begin{theorem}[Soundness \& Completeness]\label{thm:soundness-completeness}
  A formula
  of first-order logic
  is valid if and only if it has a combinatorial proof.
\end{theorem}

\section{Propositional combinatorial proofs without labels}\label{sec:prop-cps}
A \defn{proposition} is a formula with no quantifiers or terms, \eg\ $\peirceformula$, and a proposition is \defn{simple} if it has no logical constant ($1$ or $0$).
This section provides an alternative representation of fographs and combinatorial proofs in the simple propositional case, without labels (variables and atoms).
An illustrative example is shown in%
\begin{figure*}\begin{center}%
\newcommand\radius{2.7}%
\begin{pic}{.5}{3.3}
  \rput(-\radius,0){\dcponformula{\dcpcptwo}{-.08}{}}%
  \rput(\radius,0){\dcponformula{\dcptwo}{.1}{}}%
\end{pic}%
\end{center}\caption{\label{fig:homogeneous-peirce}A standard combinatorial proof (left)
and a homogeneous combinatorial proof (right) of Peirce's law
$\protect\peirceimpliesformula\,=\,\protect\peirceformula$.}\vspace{0ex}\figrule\end{figure*}
Fig.\,\ref{fig:homogeneous-peirce}.
The left side
shows a standard combinatorial proof (Def.\,\ref{def:cp}) of Peirce's law $\peirceimpliesformula=\peirceformula$.
The right side shows the label-free form, called a \emph{homogeneous combinatorial proof}, defined below.
The source colouring and target labels ($\pp$, $p$ and $q$) have disappeared,
replaced by \emph{duality} edges, shown dashed and curved.
The adjective \emph{homogeneous} reflects the common type of the source and target (both cographs with additional duality edges), in contrast to a standard combinatorial proof skeleton which is \emph{heterogeneous} (the source is coloured, while the target is labelled).

\subsection{Dualizing graphs}\label{sec:dualizing-graphs}

A graph is \defn{triangle-free} if it is $\cthree$-free, where $\cthree=\cthreegraph=\graphpairof\cthreevertices\cthreeedges$.
\begin{definition}\label{def:dualizing-graph}
A \defn{dualizing graph} is a non-empty cograph $\dualizinggraph$ equipped with a second set $\dualitiesof\dualizinggraph$ of undirected edges on $\verticesof\dualizinggraph$, called \defn{dualities}, such that $\graphpairof{\verticesof\dualizinggraph}{\dualitiesof\dualizinggraph}$ is a triangle-free cograph.
\end{definition}
Four examples of dualizing graphs are shown in the bottom row of%
\begin{figure*}\begin{pic}{-4}{.5}
\rput(-6,0){\peirceovercombprop}
\rput(-2,0){\combpropone}
\rput(2,0){\combproptwo}
\rput(6,0){\combpropthree}
\end{pic}\caption{\label{fig:dualizing-graphs}Four simple propositions $\prop$ (top row), their fographs $\graphof\prop$ (middle row), and their dualizing graphs $\dualizinggraphofprop\prop$.
Each vertex in $\graphof\prop$ and $\dualizinggraphofprop\prop$ is aligned vertically with the corresponding atom occurence in $\prop$. Dualities are shown dashed and curved.}%
\figrule\end{figure*}
Fig.\,\ref{fig:dualizing-graphs}.
Dualizing graphs generalize R\&B-cographs \cite{Ret03}.\footnote{An R\&B-cograph is a dualizing graph such that every vertex is in a unique duality.}
\begin{samepage}\begin{definition}\label{def:dualizing-graph-of-prop}
The dualizing graph $\dualizinggraphofprop\prop$ of a simple proposition $\prop$ is the dualizing graph $\dualizinggraph$ with
\begin{itemize}
\item $\verticesof\dualizinggraph=\setof{\text{occurrences of predicate symbols in }\prop}$,
\item $\vertex\vertexa\tightin\edgesof\dualizinggraph$ if and only if
the smallest subformula of $\prop$
containing both $\vertex$ and $\vertexa$ is
a conjunction (\ie, of the form $\formulaa\tightwedge\formulaaa$)
\item $\vertex\vertexa\tightin\dualitiesof\dualizinggraph$ if and only if $\vertex$ and $\vertexa$ have dual predicate symbols (\eg, $p$ and $\pp$).
\end{itemize}
\end{definition}\end{samepage}
For example, for each simple proposition $\prop$ in the top row of Fig.\,\ref{fig:dualizing-graphs},
the bottom row shows the corresponding dualizing graph $\dualizinggraphofprop\prop$.
For comparison, the fograph $\graphof\prop$ is in the middle row.
\begin{lemma}\label{lem:dualizing-graph-of-prop-well-defined}
$\dualizinggraphofprop\prop$ is a well-defined dualizing graph for every simple proposition $\prop$.\footnote{We will observe in \S\ref{sec:surjections} that $\dualizinggraphofpropsymbol$ is a surjection from simple propositions onto dualizing graphs (\reflem{lem:prop-surj}).}
\end{lemma}
\begin{proof}
Let $\dualizinggraph=\dualizinggraphofprop\prop$.
We must show
$\graphpairof{\verticesof{\dualizinggraph}}{\edgesof{\dualizinggraph}}$ and
$\graphpairof{\verticesof{\dualizinggraph}}{\dualitiesof{\dualizinggraph}}$ are $P_4$-free and
$\graphpairof{\verticesof{\dualizinggraph}}{\dualitiesof{\dualizinggraph}}$ is $\cthree$-free.

Suppose
$\graphpairof
{\setof{\vertex_1,\mkern-2mu\vertex_2,\mkern-2mu\vertex_3,\mkern-2mu\vertex_4}}
{\setof{\vertex_1\vertex_2,\mkern-1mu\vertex_2\vertex_3,\mkern-1mu\vertex_3\vertex_4}}$
is an induced subgraph of
$\graphpairof{\verticesof\fograph}{\edgesof\dualizinggraph}$.
Since $\vertex_1\vertex_2\tightin\edgesof\dualizinggraph$ there exist subformulas
$\prop_1$ and $\prop_2$ of $\prop$ containing $\vertex_1$ and $\vertex_2$, respectively, with
$\prop_1\tightwedge\prop_2$ a subformula of $\prop$.
Necessarily
$\vertex_3$ is in $\prop_1$, otherwise (since $\prop$ is a syntactic tree) $\vertex_1\vertex_3\tightin\edgesof\dualizinggraph$ (a contradiction), and similarly
$\vertex_4$ is in $\prop_2$, otherwise $\vertex_2\vertex_4\tightin\edgesof\dualizinggraph$ (a contradiction).
But then $\vertex_1\vertex_4\tightin\edgesof\dualizinggraph$, a contradiction.

Suppose
$\graphpairof
{\setof{\vertex_1,\mkern-2mu\vertex_2,\mkern-2mu\vertex_3,\mkern-2mu\vertex_4}}
{\setof{\vertex_1\vertex_2,\mkern-1mu\vertex_2\vertex_3,\mkern-1mu\vertex_3\vertex_4}}$
is an induced subgraph of
$\graphpairof{\verticesof\fograph}{\dualitiesof\dualizinggraph}$,
where $v_i$ is an occurrence of the nullary predicate symbol $p_i$.
By definition of $\dualitiesof\dualizinggraph$, we have $\pp_1=p_2$, $\pp_2=p_3$ and $\pp_3=p_4$.
Thus $p_3=p_1$, hence $p_4=\pp_1$, so $\vertex_1\vertex_4\in\dualitiesof\dualizinggraph$, a contradiction.

Suppose
$\graphpairof
{\setof{\vertex_1,\mkern-2mu\vertex_2,\mkern-2mu\vertex_3}}
{\setof{\vertex_1\vertex_2,\mkern-1mu\vertex_2\vertex_3,\mkern-1mu\vertex_3\vertex_1}}$
is an induced subgraph of $\graphpairof{\verticesof\fograph}{\dualitiesof\dualizinggraph}$,
where $v_i$ is an occurrence of the nullary predicate symbol $p_i$.
By definition of $\dualitiesof\dualizinggraph$, we have $\pp_1=p_2$, $\pp_2=p_3$ and $\pp_3=p_1$.
Thus $p_3=\pp_2=\dual\pp_1=p_1$, contradicting $\pp_3=p_1$.
\end{proof}

\subsection{Dualizing nets}\label{sec:dualizing-nets}

A set $W\subseteq\verticesof\dualizinggraph$ \defn{induces a bimatching} in
a dualizing graph
$\dualizinggraph$ if $W$ induces a matching
in $\graphpairof{\verticesof\dualizinggraph}{\edgesof\dualizinggraph}$
and induces a matching in
$\graphpairof{\verticesof\dualizinggraph}{\dualitiesof\dualizinggraph}$.
\begin{definition}
A \defn{dualizing net} $\net$ is a dualizing graph with no induced bimatching,
such that $\graphpairof{\verticesof\net}{\dualitiesof\net}$ is a matching.
\end{definition}
For example,
\,$\newcommand\gap{\mkern7mu}\namedvx a\gap\namedvx b \gap \namedvx c \gap \namedvx d \e a b \dualitystyle \nccurve[angleA=30,angleB=150] a c \nccurve[angleA=-30,angleB=-150] b d$\,
is a dualizing net, while
\,$\newcommand\gap{\mkern7mu}\namedvx a\gap\namedvx b \gap \namedvx c \gap \namedvx d \e a b \e c d  \dualitystyle \nccurve[angleA=30,angleB=150] a c \nccurve[angleA=-30,angleB=-150] b d $\,
and
\,$\newcommand\gap{\mkern7mu}\namedvx a\gap\namedvx b \gap \namedvx c \gap \namedvx d \e a b \dualitystyle \nccurve[angleA=25,angleB=155] a d \nccurve[angleA=-30,angleB=-150] b d $\,
are not.
The third dualizing graph in the bottom row of Fig\,\ref{fig:dualizing-graphs} is a dualizing net, while the other three
are not.
Dualizing nets
are in bijection with
even-length alternating elementary acyclic R\&B cographs \cite{Ret03}.

\subsection{Propositional homogeneous combinatorial proofs}

A \defn{skew fibration}
$\skewfib\mkern-0mu:\mkern-1mu\dualizingcover\mkern-1mu\to\mkern-1mu\dualizingbase$
of dualizing graphs is a skew fibration
$\skewfib:\graphpairof{\verticesof\dualizingcover}{\edgesof\dualizingcover}\to\graphpairof{\verticesof\dualizingbase}{\edgesof\dualizingbase}$
such that
$\skewfib:\dualizinggraphpairof\dualizingcover\to\dualizinggraphpairof\dualizingbase$\/ is a homomorphism.
\begin{definition}
  A \defn{homogeneous combinatorial proof} of a dualizing graph\/ $\dualizingbase$ is a skew fibration\/ \mbox{$\skewfib:\net\to\dualizingbase$}
  from a dualizing net\/ $\net$.
  A \defn{homogeneous combinatorial proof} of a simple proposition $\prop$ is a homogeneous combinatorial proof of its dualizing graph $\dualizinggraphofprop\prop$.
\end{definition}
For example,
a homogeneous combinatorial proof of Peirce's law $\peirceimpliesformula=\peirceformula$ is shown on the right of Fig\,\ref{fig:homogeneous-peirce}.

\subsection{Propositional homogeneous soundness and completeness}

\begin{theorem}[Propositional homogeneous soundness and completeness]\label{thm:prop-soundness-completeness}
A simple proposition is valid if and only if it has a homogeneous combinatorial proof.
\end{theorem}
\begin{proof}
A corollary of Theorem\,\ref{thm:soundness-completeness}, detailed in \S\,\ref{sec:proof-of-propositional-homogeneous-soundness-completeness}.
\end{proof}

\section{Monadic combinatorial proofs without labels}\label{sec:monadic-homogeneous}

A formula is \defn{monadic} if its predicate symbols are unary and it has no function symbols or logical constants, \eg, $\drinkerformula$.
This section extends homogeneous combinatorial proofs to the monadic case.
Fig.\,\ref{fig:drinker-no-labels}\figdrinkernolabels{} shows an illustrative example:
on the left is the combinatorial proof of \mbox{$\drinkerformula$} presented in the Introduction,
and on the right is the corresponding homogeneous combinatorial proof, to be defined below.

For technical convenience throughout this section we assume every monadic formula is closed, \ie, has no free variables. This loses no generality because a formula $\formula$ with free variables $x_1,\ldots,x_n$ is valid if and only if its closure $\forall x_1\ldots\forall x_n\mkern2mu \formula$ is valid.

Given a directed edge $e=\diredge\vertex\vertexa$, $\vertex$ is the \defn{source} of $e$, $\vertexa$ is the \defn{target} of $e$, and $\vertex$ and $\vertexa$ are \defn{in} $e$.
\begin{definition}\label{def:pre-mograph}
A \defn{pre-monadic graph} or \defn{pre-mograph} is a dualizing graph
$\mograph$ equipped with a non-empty set
$\bindingedgesof\mograph$ of directed edges on
$\verticesof\mograph$,
called \defn{bindings}, such that
if a vertex $\vertex$ is the target of a binding then
$\vertex$ is in no other binding.\footnote{In other words, if $\diredge\vertexa\vertex\tightin\bindingedgesof\mograph$, then (1)
$\diredge\vertex\vertexaa\tightnotin\bindingedgesof\mograph$ for all vertices $\vertexaa$, and (2)
$\diredge\vertexap\vertex\tightin\bindingedgesof\mograph$
implies $\vertexap=\vertexa$.}
\end{definition}
\begin{figure*}\begin{center}\vspace{5ex}\hspace{-8ex}\begin{math}
\monadicformulaeg
\end{math}
\hspace{20ex}
\monadicfographeg
\hspace{28ex}
\mographeg
\vspace{6ex}\end{center}\caption{\label{fig:mograph}%
A monadic formula $\monadicformula$,
its fograph $\graphof\monadicformula$,
and its mograph $\mographofformula\monadicformula$, respectively.%
}\figrule\end{figure*}%
An example of a pre-mograph is shown on the right of Fig.\,\ref{fig:mograph}, with two dualities (dashed and curved) and three bindings (directed and curved).
A vertex in a pre-mograph $\mograph$ is a \defn{literal} if it is the target of a binding, otherwise a \defn{binder}.
If $\diredge\binder\literal\tightin\bindingedgesof\mograph$ we say that
$\binder$ \defn{binds} $\literal$.\footnote{Note that, by the condition in the definition of pre-mograph,
$\binder$ must be a binder.}
The \defn{scope} of a binder $\binder$ in $\mograph$
is the smallest proper strong module of $\graphpairof{\verticesof\mograph}{\edgesof\mograph}$ containing $\binder$.
\begin{definition}\label{def:mograph}
A \defn{mograph} $\mograph$ is a pre-mograph such that
no binder is in a duality,
every binder has non-empty scope,
and
$\diredge\binder\literal\tightin\bindingedgesof\mograph$ only if
$\literal$ is in the scope of $\binder$.
\end{definition}
For example, the pre-mograph on the right of Fig.\,\ref{fig:mograph} is a mograph.
\begin{definition}\label{def:mograph-of-formula}
The \defn{mograph} $\mographofformula\monadicformula$ of a closed monadic formula $\monadicformula$ is the mograph defined by:
\begin{itemize}
\item $\verticesof\mograph=\setof{\text{occurrences of atoms and quantifiers in }\monadicformula}$,
\item $\vertex\vertexa\tightin\edgesof\mograph$ if and only if either
  \begin{itemize}
  \item the smallest subformula containing both $\vertex$ and $\vertexa$ is
a conjunction (\ie, of the form $\formula\tightwedge\formulaa$)
  \item $\vertex$ is an existential quantifier, $\vertexa$ is in its scope, and $\vertexa\tightneq\vertex$.
  \end{itemize}
\item $\vertex\vertexa\tightin\dualitiesof\mograph$ if and only if $\vertex$ and $\vertexa$ are atoms with dual predicate symbols (\eg, $\px$ and $\ppy$), and
\item $\diredge\vertex\vertexa\tightin\bindingsof\mograph$ if and only if $\vertex$ is a quantifier, $\vertexa$ is an atom, and $\vertex$ binds $\vertexa$.
\end{itemize}
\end{definition}
For example, in Figure\,\ref{fig:mograph}, the closed monadic formula $\monadicformula=\monadicformulaeg$ on the left has the mograph $\mographofformula\monadicformula$ on the right.
\begin{lemma}\label{lem:mograph-of-formula-well-defined}
$\mographofformula\monadicformula$ is a well-defined mograph for every closed monadic formula $\monadicformula$.\footnote{We will observe in \S\ref{sec:surjections} that $\mographofformulasymbol$ is a surjection from closed monadic formulas onto mographs (\reflem{lem:surj-closed-monadic-formulas-to-mographs}).}
\end{lemma}
\begin{proof}
Let $\mograph=\mographofformula\monadicformula$.
Since every atom-occurrence in $\monadicformula$ has a single variable, each literal is the target of at most one binding in $\mograph$, and since no atom-occurrence binds another atom-occurrence, $\mograph$ satisfies the condition on bindings in the definition of pre-mograph (Def.\,\ref{def:pre-mograph}).
By reasoning as in the proof of \reflem{lem:dualizing-graph-of-prop-well-defined}, $\dualizinggraphpairof{\mograph}$ is $\pfour$-free and $\cthree$-free.
By definition of $\mographofformulasymbol$, no binder is in a duality.
It remains to show that $\graphpairofgraph{\mograph}$ is a cograph, every binder has non-empty scope, and
$\diredge\binder\literal\tightin\bindingedgesof\mograph$ only if $\literal$ is in the scope of $\binder$.
We proceed by induction on the structure of $\monadicformula$.

Base case: $\monadicformula=\px$ for some $p$ and $x$, so $\mograph$ is a single vertex, hence a mograph.

Induction case: $\monadicformula=\monadicformula_1\wedgeorvee\monadicformula_2$ for $\wedgeorveeinset$.
By induction hypothesis $\mograph_i=\mographofformula{\monadicformula_i}$ is a mograph ($i=1,2$).
By definition of $\edgesof\mograph$, we have
$\graphpairofgraph{\mograph}=\graphpairofgraph{\mograph_1}\graphjoin\graphpairofgraph{\mograph_2}$
or
$\graphpairofgraph{\mograph}=\graphpairofgraph{\mograph_1}\graphunion\graphpairofgraph{\mograph_2}$,
thus $\graphpairofgraph\mograph$ is a cograph since each $\graphpairofgraph{\mograph_i}$ is a cograph.
The scope of a binder $\binder$ in $\mograph$ is at least the scope of $\binder$ in the $\mograph_i$ containing $\binder$,
thus the scope of $\binder$ in $\mograph$ is non-empty and contains every literal bound by $\binder$, since $\mograph_i$ is a mograph.

Induction case: $\monadicformula=\forallorexists\mkern-1mu x\mkern2mu\monadicformulap$ for $\forallorexistsinset$.
By induction hypothesis, $\mographp=\mographofformula\monadicformulap$ is a mograph.
By definition of $\edgesof\mograph$ we have
$\graphpairofgraph\mograph=\binder\graphunion\mographp$ or
$\graphpairofgraph\mograph=\binder\graphjoin\mographp$ for a vertex
$\binder$
(the initial occurrence of $\forallorexists\mkern-1mu x$ in $\formula$),
thus $\graphpairofgraph\mograph$ is a cograph since $\graphpairofgraph{\mographp}$ is a cograph.
The scope of $\binder$ in $\mograph$ comprises every literal, and is therefore non-empty and contains every literal bound by $\binder$.
The scope of any other binder $\binderp$ in $\mograph$ is equal the scope of $\binderp$ in $\mographp$,
so is non-empty and contains every literal bound by $\binderp$, since $\mographp$ is a mograph.
\end{proof}

\subsection{Monets}\label{sec:monets}

A mograph is \defn{linked} if every literal is in a unique duality.
An example of a linked mograph is shown in Fig.\,\ref{fig:monet} (left).
\begin{definition}
Let $\mograph$ be a linked mograph.
Its \defn{binder equivalence}\/
$\brel\mograph$ is the equivalence relation on binders generated by
$\binder_1\brel\mograph\binder_2$ if there exist literals $\literal_1$ and $\literal_2$ with
$\diredge{\binder_1}{\literal_1},\diredge{\binder_2}{\literal_2}\tightin\bindingsof\mograph$
and $\literal_1\literal_2\tightin\dualitiesof\mograph$.
\end{definition}
Thus
$\binder_1\brel\mograph\binder_2$
if and only if there exists a binding/duality pattern of the form
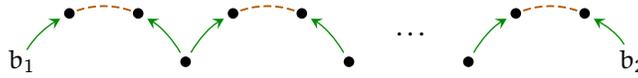
\begin{center}\begin{math}
\newcommand\bindright[2]{\nccurve[angleA=60,angleB=-150]{#1}{#2}}
\newcommand\bindleft[2]{\nccurve[angleA=120,angleB=-30]{#1}{#2}}
\newcommand\dualright[2]{\nccurve[angleA=20,angleB=160]{#1}{#2}}
\newcommand\pivot[1]{\hspace{.4ex}\namedvx{#1}\hspace{.4ex}}
\begin{array}{cccccccccccc}
  & \namedvx 1\quad & \namedvx 2 &            & \namedvx 3\quad & \namedvx 4 &            & \;\raisebox{-1.2ex}{\ldots} &            & \namedvx 5\quad & \namedvx 6 \\[1.3ex]
\Rnode{b}{\binder_1} &            &            & \pivot u &            &            & \pivot v &                         & \pivot w &            &            & \Rnode{c}{\binder_2}
\end{array}
{\bindingstyle
\bindleft{u}{2}
\bindright{u}{3}
\bindleft{v}{4}
\bindright{w}{5}
\bindleft{c}{6}
\psset{nodesepA=0pt}
\bindright{b}{1}
}
{\dualitystyle
\dualright 1 2
\dualright 3 4
\dualright 5 6
}
\end{math}\end{center}
Let $\mograph$ be a linked mograph.
A binder $\binder$ in $\mograph$ is \defn{existential} (resp.\ \defn{universal}) if, for every other vertex $\vertex$ in the scope of $\binder$, we have
$\binder\vertex\tightin\edgesof\mograph$ (resp.\ $\binder\vertex\tightnotin\edgesof\mograph$).\footnote{Since the scope of a binder is a proper strong module, every binder is either universal or existential (and not both).}
A \defn{conflict} in $\mograph$ is a pair $\setof{b,c}$ of distinct universal binders $b$ and $c$ such that $b\brel\cover c$.
\begin{definition}\label{def:mograph-consistent}
A mograph is \defn{consistent} if it has no conflict.
\end{definition}
A \defn{dependency} of $\mograph$ is a pair $\setof{b,c}$ of binders with $b\brel\cover c$, $b$ existential, and $c$ universal.
A \defn{leap} is a dependency or duality.
\begin{definition}
The \defn{leap graph} $\leapgraphof\mograph$ of a linked mograph $\mograph$ is
$\graphpairof{\verticesof\mograph}{\leapsof\mograph}$ for $\leapsof\mograph$ the set of leaps of $\mograph$.
\end{definition}
An example of a leap graph is shown in Fig.\,\ref{fig:monet} (right).
A set of vertices $W\subseteq\verticesof\mograph$ \defn{induces a bimatching} in a linked mograph
$\mograph$ if $W$ induces a matching in
$\graphpairof{\verticesof\mograph}{\edgesof\mograph}$
and induces a matching in
$\leapgraphof\mograph$.
\begin{definition}
A \defn{monet} (\emph{monadic net}) is a consistent linked mograph with no induced bimatching.
\end{definition}
An example of a monet is shown in\begin{figure*}\begin{center}\vspace{5ex}\begin{math}
\moneteg
\hspace{35ex}
\monetegleapgraph
\end{math}\vspace{6ex}\end{center}\caption{\label{fig:monet}A monet $\monet$ (left) and its leap graph $\leapgraphof\monet$ (right).}\figrule\end{figure*}
Fig.\,\ref{fig:monet}.

\subsection{Monadic homogeneous combinatorial proofs}
Let $\mographa$ and $\mograph$ be mographs.
A function $\fib:\verticesof\mographa\to\verticesof\mograph$ \defn{preserves existentials}
if
for every existential binder $\binder$ in $\mographa$
the vertex $\fib(\binder)$ is an existential binder in $\mograph$.
\begin{definition}\label{def:monadic-skew-bifib}
A \defn{skew bifibration} $\bifib:\mographa\to\mograph$ between mographs is
an existential-preserving skew fibration
$\bifib:\graphpairof{\verticesof\mographa}{\edgesof\mographa}\to\graphpairof{\verticesof\mograph}{\edgesof\mograph}$
such that
\begin{itemize}
\item
$\bifib:\graphpairof{\verticesof\mographa}{\dualityedgesof\mographa}\to\graphpairof{\verticesof\mograph}{\dualityedgesof\mograph}$
is a homomorphism and
\item
$\bifib:\graphpairof{\verticesof\mographa}{\bindingedgesof\mographa}\to\graphpairof{\verticesof\mograph}{\bindingedgesof\mograph}$
is a fibration.\vspace{3pt}
\end{itemize}
\end{definition}
An example of a skew bifibration between mographs is shown on the right of Fig.\,\ref{fig:drinker-no-labels}.
\begin{definition}
A \defn{homogeneous combinatorial proof} of a mograph $\mograph$ is a skew bifibration \mbox{$\bifib:\net\to\mograph$} from a monet $\net$.
A \defn{homogeneous combinatorial proof} of a closed monadic formula $\monadicformula$ is a homogeneous combinatorial proof of its mograph $\mographofformula\monadicformula$.\footnote{Although, for technical convenience, throughout this paper we have assumed (without loss of generality) that every formula is rectified, this definition of \emph{homogeneous combinatorial proof} also works directly for non-rectified closed monadic formulas, without the need to first transform to rectified form.
This is because $\mographofformulasymbol$ is agnostic to the choice of bound variables, since mographs do not contain any variables.
For example,
$\mographofformula{
(\forall x\mkern1mu px)
\vee
(\forall x\mkern1mu qx)
}
=
\mographofformula{
(\forall x\mkern1mu px)
\vee
(\forall y\mkern1mu qy)
}
=
\:
\namedsingletonleft x {}
\hspace{2ex}
\namedsingletonright{px}{}
\hspace{2ex}
\namedsingletonleft{xx}{}
\hspace{2ex}
\namedsingletonright{qx}{}
{\bindingstyle
\psset{nodesepA=1pt,nodesepB=.5pt}
\nccurve[angleA=30,angleB=155]{x}{px}
\nccurve[angleA=30,angleB=155]{xx}{qx}
}
\:
$
and
$\mographofformula{\forall x\mkern2mu\exists x\mkern2mu px}
=
\mographofformula{\forall x\mkern2mu\exists y\mkern2mu py}
=
\,
\namedsingletonleft x {}
\hspace{1.6ex}
\namedsingletonleft{xx} {}
\hspace{2ex}
\namedsingletonright {px} {}
\e{xx}{px}
{\bindingstyle
\psset{nodesepA=1.5pt,nodesepB=1pt}
\nccurve[angleA=35,angleB=150]{xx}{px}
}
\,
$.}
\end{definition}
A homogenous combinatorial proof of $\drinkerformula$ is shown in Fig\,\ref{fig:drinker-no-labels} (right).

\subsection{Monadic homogeneous soundness and completeness}\label{sec:monadic-homogeneous-soundness-completeness}

\begin{theorem}[Monadic homogeneous soundness and completeness]\label{thm:monadic-soundness-completeness}
A closed monadic formula is valid if and only if it has a homogeneous combinatorial proof.
\end{theorem}
\begin{proof}
A corollary of Theorem\,\ref{thm:soundness-completeness}, detailed in \S\,\ref{sec:proof-of-monadic-soundness-completeness}.\todo{ref}
\end{proof}

\section{Modal combinatorial proofs}\label{sec:modal-cps}

A \defn{modal} formula is
generated from
the \defn{modal operators}
$\nec$ (necessity) and $\pos$ (possibility) instead of quantifiers and has all predicate symbols nullary,
\eg $\modaldrinkerformula$.
Every modal formula abbreviates a standard first-order one \cite[\S3.3]{Min92}:
replace every $\nec$ by $\forall x$,
$\pos$ by $\exists x$, and
predicate symbol
$p$ by $px$. For example, $\modaldrinkerformula$ abbreviates $\drinkerformulamodallike$,
or
$\drinkerformula$ in rectified form.

\begin{definition}
A \defn{modal combinatorial proof} of a modal formula
$\modalformula$ is a standard combinatorial proof (Definition\,\ref{def:cp}) of the first-order formula abbreviated by $\modalformula$.
\end{definition}
For example, a modal combinatorial proof of $\modaldrinkerformula$ is shown below-left, in condensed form.
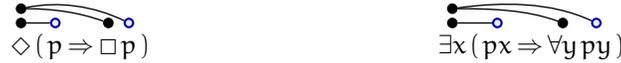
\begin{center}\begin{pic}{-.3}{.85}\rput(-3,0)\modaldrinkerInlineDisplayed\rput(3,0)\drinkerInlineDisplayed\end{pic}\end{center}
It abbreviates the first-order combinatorial proof
above-right (copied from the Introduction).
\begin{theorem}[S5 Modal Soundness \& Completeness]\label{thm:modal-soundness-completeness}
  A modal formula is valid in S5 modal logic if and only if it has a modal combinatorial proof.
\end{theorem}
\begin{proof}
  By Theorem\,3.2 of \cite[p.\,42]{Min92}, a modal formula is valid in S5 if and only if the first-order formula it abbreviates is valid in first-order logic.
Thus the result follows from Theorem\,\ref{thm:soundness-completeness}.
\end{proof}

\subsection{Modal combinatorial proofs without labels}

A modal formula is \defn{closed} if
every predicate symbol occurrence is bound by a modal operator (\eg $\modaldrinkerformula$ but not $\openmodaldrinkerformula$)
and \defn{simple} if it has no logical constant ($1$ or $0$).
\begin{definition}
The mograph $\modalmographof\modalformula$ of a simple closed modal formula $\modalformula$ is the mograph $\mograph$ defined by
\begin{itemize}
\item $\verticesof\mograph=\setof{\text{occurrences of predicate symbols and modal operators in }\modalformula}$,
\item $\vertex\vertexa\tightin\edgesof\mograph$ if and only if either
  \begin{itemize}
  \item the smallest subformula containing both $\vertex$ and $\vertexa$ is
a conjunction (\ie, of the form $\formula\tightwedge\formulaa$)
  \item $\vertex$ is a $\pos$ with $\vertexa$ is in its scope and $v\tightneq w$.
  \end{itemize}
\item $\vertex\vertexa\tightin\dualitiesof\mograph$ if and only if $\vertex$ and $\vertexa$ are dual predicate symbols, and
\item $\diredge\vertex\vertexa\tightin\bindingsof\mograph$ if and only if $\vertex$ is a modal operator, $\vertexa$ is a predicate symbol, and $\vertex$ binds $\vertexa$.
\end{itemize}
\end{definition}
For example,
\begin{center}\vspace{-1.5ex}\(
\modalmographof{\modaldrinkerformula}
\hspace{6ex}
=
\hspace{12ex}
\rput(0,1.3ex){\drinkerbasemograph}
\)\vspace{3ex}\end{center}
\begin{definition}
A \defn{homogeneous combinatorial proof} of a closed modal formula $\modalformula$
is a homogeneous combinatorial proof of its mograph $\modalmographof\modalformula$.
\end{definition}
For example, a homogeneous combinatorial proof of $\modaldrinkerformula$ is shown on the right of Fig.\,\ref{fig:drinker-no-labels}.
\begin{theorem}[Modal homogeneous soundness and completeness]\label{thm:modal-homogeneous-soundness-completeness}
A closed modal formula is valid in S5 modal logic if and only if it has a homogeneous combinatorial proof.
\end{theorem}
\begin{proof}
Since $\modalmographof\modalformula=\mographofformula\modalformulap$ for $\modalformulap$ the first-order formula encoded by $\modalformula$, the result is a corollary of Theorem\,\ref{thm:modal-soundness-completeness}.
\end{proof}

\section{Proof of the Soundness Theorem}\label{sec:soundness}\label{sec:proof-of-soundness}

In this section we prove the Soundness Theorem, Theorem~\ref{thm:soundness}.
\begin{lemma}\label{lem:graph-surj}
The function $\graphofsymbol$ (Def.\,\ref{def:graph}) is a surjection from formulas onto \recombulas.\footnote{Dropping the assumption that every formula is rectified leads to a surjection onto all fographs: see \reflem{lem:xgraph-surj}.}
Two formulas have the same graph if and only if they are equal modulo\footnote{Recall that, without loss of generality, we assume all formulas are rectified. Thus these equations do not include cases such as
$
px\wedge\exists x\mkern2mu qx
=
\exists x(px\wedge\mkern2mu qx)
$, an equality between formulas which are not logically equivalent.}
  \begin{align*}
    \formula\tightwedge\formulaa &\fateq \formulaa\tightwedge\formula\;\;\;\;
    &
    \formula\wedge(\formulaa\tightwedge\formulaaa) &\fateq (\formula\tightwedge\formulaa)\wedge\formulaaa\;\;\;\;
    &
    \exists x\mkern2mu\exists y\mkern2mu \formula &\fateq \exists y\mkern2mu\exists x\mkern2mu\formula\;\;\;\;
    &
    \formula\wedge\exists x\mkern2mu\formulaa &\fateq \exists x\mkern2mu(\formula\tightwedge\formulaa)
    \\[1ex]
    \formula\tightvee\formulaa &\fateq \formulaa\tightvee\formula
    &
    \formula\vee(\formulaa\tightvee\formulaaa) &\fateq (\formula\tightvee\formulaa)\vee\formulaaa
    &
    \forall x\mkern2mu\forall y\mkern2mu \formula &\fateq \forall y\mkern2mu\forall x\mkern2mu\formula
    &
    \formula\vee\forall x\mkern2mu\formulaa &\fateq \forall x\mkern2mu(\formula\tightvee\formulaa)
\\[-2.5ex]
\end{align*}
\end{lemma}
\begin{proof}
  A routine induction.\todo{elaborate in appendix}
\end{proof}
Let $\fograph$ be a rectified fograph. Using the above Lemma,
choose a formula $\formula$ such that $\graphof\formula\tighteq\fograph$.
Define $\fograph$ as \defn{valid} if $\formula$ is valid.
This is well-defined with respect to choice of $\formula$ since
every equality in Lemma\,\ref{lem:graph-surj} is a logical equivalence.
Define a coloured fograph as valid if its underlying uncoloured fograph is valid.

Write $\isvalid\chi$ to assert that a formula or fograph $\chi$ is valid,
and $\formula\assignment{\assign\variable\term}$ for the result of substituting a term $\term$ for all occurrences of the variable $\variable$ in a formula $\formula$,
where, without loss of generality (by renaming bound variables in $\formula$ as needed), no variable in $\term$ is a bound variable of $\formula$ \cite[\S1.1.2]{TS96}.
\begin{lemma}\label{lem:formula-inferences}
  Let $\formula$, $\formulaa$ and $\formulaaa$ be formulas.
  \begin{enumerate}
  \item\label{itm:and-inference} $\isvalid\formula\tightwedge\formulaa$ if and only if ($\isvalid\formula$ and $\isvalid\formulaa$).
  \item\label{itm:or-inference} $\isvalid\formula\tightvee\formulaa$ if ($\isvalid\formula$ or $\isvalid\formulaa$).
  \item\label{itm:distrib-inference} $\isvalid(\formula\tightvee\formulaa)\tightwedge\formulaaa$ implies $\isvalid(\formula\tightwedge\formulaaa)\vee(\formulaa\tightwedge\formulaaa)$.
  \item\label{itm:universal-inference} $\isvalid\forall x\mkern2mu\formula$ if and only if $\isvalid\formula$.
  \item\label{itm:solo-existential-inference} $\isvalid\formula\assignment{\assign x t}$ implies $\isvalid\exists x\mkern2mu\formula$.
  \item\label{itm:existential-inference} $\isvalid\formula\vee\formulaa\assignment{\assign x t}$ implies $\isvalid\formula\vee\exists x\mkern2mu\formulaa$.
  \item\label{itm:lindist-inference} $\isvalid(\formula\tightvee\formulaa)\tightwedge\formulaaa$ implies $\isvalid\formula\tightvee(\formulaa\tightwedge\formulaaa)$.
  \end{enumerate}
\end{lemma}
\begin{proof}
  \ref{itm:and-inference}--\ref{itm:existential-inference} are standard inferences and properties of validity in first-order classical logic. See \cite{TS96} and \cite{Joh87}, for example. Property \ref{itm:lindist-inference} follows from \ref{itm:and-inference} and \ref{itm:distrib-inference}.
\end{proof}
\subsection{Soundness of fonets}\label{sec:fonet-soundness}

In this section we prove that fonets are sound, \ie, every fonet is valid (Lemma~\ref{lem:fonet-soundness} below).

Let $\fograph$ be a fograph.
A set $\portion\tightsubseteq\verticesof\fograph$ is \defn{well-founded} if $\portion$ contains a binder only if $\portion$ contains a literal.
\begin{definition}
A \defn{portion} of a rectified fograph $\combulavar$ is a set $\portion\tightsubseteq\verticesof\fograph$ such that $\portion$ and  $\verticesof\fograph\tightsetminus\portion$ are well-founded, and $\portion$ is closed under adjacency and binding:
if $\vertex\vertexa\tightin\edgesof\fograph$ or $\diedge\vertex{\mkern-3mu\vertexa}\tightin\edgesof{\bindinggraphof\combulavar}$,
then $\vertex\tightin \portion$ if and only if $\vertexa\tightin \portion$.
\end{definition}
A variable $x$ in a fograph $\fograph$ is
\defn{bound} if $\fograph$ contains an $x$-binder,
and \defn{free} if $\fograph$ contains an $x$-literal but no $x$-binder.
Two fographs are
\defn{independent}
if
any variable in both is free in both.

\subsubsection{Fusion}\label{sec:fusion}

\begin{definition}\label{def:fusion}
Let $\fograph$ and $\fographp$ be independent rectified fographs with respective portions $\portion$ and $\portionp\mkern-2mu$.
The \defn{fusion} of $\fograph$ and $\fographp\mkern-1mu$ at $\portion$ and $\portionp\mkern-2mu$ is the union $\fograph\mkern-2mu\graphunion\mkern-2mu\fographp$ together with edges between every vertex in $\portion$ and every vertex in $\portionp\mkern-2mu$.
\end{definition}
For example, if $\fograph=\mkern2mu\namedsingletonleft x x \mkern12mu\namedsingletonright{ppx}{\ppx}\e x {ppx} \mkern7mu\singletonleft\py\mkern2mu$,
$\fographp=\mkern4mu\singletonq\mkern6mu\singletonqq\mkern6mu\singletonz\mkern2mu$,
$\portion=\setof{\singletonleft\py}$
and
$\portionp=\setof{\singletonq\mkern1mu,\mkern4mu\singletonqq}$,
then
the fusion of $\fograph$ and $\fographp$ at $\portion$ and $\portionp$ is
$\namedsingletonleft x x \mkern12mu\namedsingletonright{ppx}{\ppx}\e x {ppx} \mkern7mu \namedsingletonleft{py}{\py}
\mkern12mu
\namedsingletonright q q\e{py}{q}\mkern6mu\namedsingletonright{qq}{\qq}\nccurve[nodesep=0ex,angleA=25,angleB=145]{py}{qq} \mkern6mu\singletonz\mkern2mu$.
Colourings are inherited during fusion, since they are inherited during graph union $\graphunion$.
For example, if $\cover=\mkern2mu\namedsingletonleft x x \mkern12mu\singletonred{\ppx}\e x v \mkern7mu\singletonredleft{\py}\mkern2mu$,
$\coverp=\mkern4mu\inlinegreenvx{q}\mkern2mu q\mkern6mu\inlinegreenvx{qq} \mkern2mu \qq\mkern6mu\singletonz\mkern2mu$,
$\portion=\setof{\singletonredleft\py}$
and
$\portionp=\setof{\inlinegreenvx{q}\mkern2mu q\mkern1mu,\mkern4mu\inlinegreenvx{qq}\mkern2mu\qq}$,
then
the fusion of $\cover$ and $\coverp$ at $\portion$ and $\portionp$ is
$\namedsingletonleft x x \mkern12mu\singletonred{\ppx}\e x v \mkern7mu \py \inlineredvx{py}
\mkern12mu
\inlinegreenvx{q}\e{py}{q}\mkern.5mu q\mkern6mu\inlinegreenvx{qq}\nccurve[nodesep=0ex,angleA=25,angleB=145]{py}{qq} \mkern1mu \qq\mkern6mu\singletonz\mkern2mu$.

Write $\induced\graph W$ for the subgraph of $\graph$ induced by $W$.
\begin{lemma}\label{lem:fusion-sound}
  Every fusion of valid rectified fographs is valid.
\end{lemma}
\begin{proof}
  Let $\fusion$ be the fusion of valid rectified fographs $\fograph$ and $\fographp$ at portions $\portion$ and $\portionp$.
  We consider four cases.
  \begin{enumerate}
  \item $\portion$ or $\portionp$ is empty. Without loss generality, we may assume both are empty, since with one portion empty the fusion operation no longer depends on the other.
Thus $\fusion=\fograph\graphunion\fographp$ for rectified fographs $\fograph$ and $\fographp$, so by Lemma\,\ref{lem:graph-surj} there exist formulas $\formula$ and $\formulap$ with $\graphof\formula\tighteq\fograph$ and $\graphof\formulap\tighteq\fographp$.
Since $\fograph$ and $\fographp$ are independent, $\graphof{\formula\tightvee\formulap}=\fusion$. Since $\isvalid\fograph$ and $\isvalid\fographp$ we have $\isvalid\formula$ and $\isvalid\formulap$, hence $\isvalid\formula\tightvee\formulap$ by Lemma\,\ref{lem:formula-inferences}.\ref{itm:or-inference}. Thus $\isvalid\fusion$.
  \item $\portion\tighteq\verticesof\fograph$ and $\portionp=\verticesof\fographp$.  Thus $\fusion=\fograph\graphjoin\fographp$.
As in the previous case we have valid formulas $\formula$ and $\formulap$ with $\graphof\formula\tighteq\fograph$ and $\graphof\formulap\tighteq\fographp$. Thus $\isvalid\fusion$ since $\isvalid\formula\tightwedge\formulap$ by Lemma\,\ref{lem:formula-inferences}.\ref{itm:and-inference} and $\fusion=\graphof{\formula\tightwedge\formulap}$.
\item $\portion\tighteq\verticesof\fograph$ or $\portionp\tighteq\verticesof\fographp$, and the previous two cases do not hold. Without loss of generality assume $\portionp\tighteq\verticesof\fographp$, so $\emptyset\tightneq\portion\tightneq\verticesof\fograph$.
Let $\dualportion\tighteq\verticesof\fograph\mkern-1mu\tightsetminus\portion\neq\emptyset$.
Thus $\fusion=\induced{\fograph}{\dualportion}\graphunion(\induced{\fograph}{\portion}\graphjoin\fographp)$.
By Lemma\,\ref{lem:graph-surj} there exist formulas $\formula^*$, $\formula$ and $\formula'$ with $\graphof{\formula^*}\tighteq\induced{\fograph}{\dualportion}$,
$\graphof{\formula}\tighteq\induced{\fograph}{\portion}$ and $\graphof{\formulap}\tighteq\fographp$.
Since $\isvalid\fographp$ we have $\isvalid\formulap$, and since $\isvalid\fograph$ and $\graphof{\formula^*\tightvee\formula}=\fograph$, we have $\isvalid\formula^*\tightvee\formula$.
Thus $\isvalid(\formula^*\tightvee\formula)\tightwedge\formulap$ by Lemma\,\ref{lem:formula-inferences}.\ref{itm:and-inference}, so
$\isvalid\formula^*\tightvee(\formula\tightwedge\formulap)$ by Lemma\,\ref{lem:formula-inferences}.\ref{itm:lindist-inference}, hence $\isvalid\fusion$ since $\fusion=\graphof{\formula^*\tightvee(\formula\tightwedge\formulap)}$.
\item Otherwise $\emptyset\neq\portion\tightneq\verticesof\fograph$ and
$\emptyset\neq\portionp\mkern-1mu\neq\verticesof\fographp$.
Let $\dualportion=\verticesof\fograph\mkern-1mu\tightsetminus\portion\neq\emptyset$ and
$\dualportionp=\verticesof\fographp\mkern-1mu\tightsetminus\portionp\neq\emptyset$.
Thus the rectified fograph $\fusion$ is $\induced{\fograph}{\dualportion}\graphunion\induced{\fographp}{\dualportionp}\graphunion(\induced{\fograph}{\portion}\graphjoin\induced{\fographp}{\portionp})$.
By Lemma\,\ref{lem:graph-surj} there exist formulas $\formula^*$, $\formulap^*$, $\formula$ and $\formulap$ with $\graphof{\formula^*}=\induced{\fograph}{\dualportion}$,
$\graphof{\formulap^*}=\induced{\fographp}{\dualportionp}$,
$\graphof{\formula}=\induced{\fograph}{\portion}$, and
$\graphof{\formulap}=\induced{\fographp}{\portionp}$.
Since $\isvalid\fograph$ and $\graphof{\formula^*\tightvee\formula}=\fograph$ we have $\isvalid\formula^*\tightvee\formula$, and
since $\isvalid\fographp$ and $\graphof{\formulap^*\tightvee\formulap}=\fographp$ we have $\isvalid\formulap^*\tightvee\formulap$.
{\newcommand\bigformula{(\formula^*\tightvee\formulap^*)\tightvee(\formula\tightwedge\formulap)}%
Thus $\isvalid\bigformula$, so $\isvalid\fusion$ since $\fusion=\graphof\bigformula$.}
\end{enumerate}
\vskip-4ex
\end{proof}
\begin{lemma}\label{lem:pres-fusion}
  Every fusion of two rectified fonets is a rectified fonet.
\end{lemma}
\begin{proof}
  Let $\fusion$ be a fusion of rectified fonets $\cover$ and $\coverp$.
  Since each portion is closed under adjacency, $\fusion$ is a union of cographs,
  hence is a cograph.
  Every binder scope contains a literal, by inheritance from $\cover$ and $\coverp$.
  Since $\cover$ and $\coverp$ are rectified and (by the constraint on the definition of fusion) independent,
  and no links traverse between the two in $\fusion$,
  every union of dualizers for $\cover$ and $\coverp$ is a dualizer for $\fusion$, and vice versa.
  Thus the set of dependencies of $\fusion$ is the union of those of $\cover$ and $\coverp$, so any $W\tightsubseteq\verticesof\fusion$ inducing a bimatching in $\fusion$ would induce a bimatching in $\cover$ or $\coverp$.
Because $\cover$ and $\coverp$ are independent, $\fusion$ is rectified.
\end{proof}

\subsubsection{Universal quantification}\label{sec:universal}

\begin{definition}\label{def:universal}
Let $\fograph$ be a rectified fograph with no $\variable$-binder.
The \defn{universal quantification} of $\fograph$ by $x$
is $\singletonx\mkern-1mu\graphunion\mkern-1mu\fograph$.
\end{definition}
\begin{lemma}\label{lem:universal-sound}
  Every universal quantification of a valid rectified fograph is valid.
\end{lemma}
\begin{proof}
  Let $\fograph=\singletonx\graphunion\fographa$ be the universal quantification of a valid rectified fograph $\fographa$ by $x$. By Lemma\,\ref{lem:graph-surj} there exists a formula $\formula$ such that $\graphof\formula\tighteq\fographa$, and $\isvalid\formula$ since $\isvalid\fograph$. Thus $\graphof{\forall x\mkern2mu\formula}=\fograph$, hence $\isvalid\fograph$ since $\isvalid\forall x\mkern2mu\formula$ if and only if $\isvalid\formula$, by Lemma\,\ref{lem:formula-inferences}.\ref{itm:universal-inference}.
\end{proof}
If $\cover$ is a coloured rectified fograph, in the universal quantification $\singletonx\graphunion\cover$ we assume that the colouring of $\cover$ is inherited, while $\singletonx$ remains uncoloured.
\begin{lemma}\label{lem:pres-universal}
  Every universal quantification of a rectified fonet is a rectified fonet.
\end{lemma}
\begin{proof}
  Let $\coverp$ be the universal quantification $\singletonx\mkern-1mu\graphunion\mkern-1mu\cover$.
  Dualizers for $\cover$ are dualizers for $\coverp$, and vice versa, since if $x$ occurs in $\cover$, it has merely transitioned from free to bound.
  The leap graph of $\coverp$ is that of $\cover$ together with additional dependencies involving $\singletonx$.
  Since $\singletonx$ is in no edge, any $W\tightsubseteq\verticesof\coverp$ inducing a bimatching in $\coverp$ would induce a bimatching in $\cover$.
\end{proof}

\subsubsection{Existential quantification}\label{sec:existential}

\begin{definition}\label{def:existential}
Let $\fograph$ be a rectified fograph without the variable $x$, let $\portion$ be a non-empty portion of $\fograph$, and let $\occs$ be a set of occurrences of a term $t$ in labels of literals in $\portion$, such that $t$ contains no bound variable of $\fograph$.
The \defn{existential quantification} of $\fograph$ by $x$ at $\occs$ in $\portion$ is $\singletonx\mkern-1mu\graphunion\mkern-1mu\fograph\occsubst t \omega x$ together with an edge between $\singletonx$ and each vertex in $\portion$, where $\fograph\occsubst t \omega x$ is the result of substituting $x$ for every occurrence of $t$ in $\occs$.
\end{definition}
For example, if $\fograph=\mkern2mu\singleton{p \tightf g y} \mkern7mu\singleton{\pp \tightf gy}\mkern2mu$,
$\portion=\setof{\singleton{p \tightf g y}}$
and
$\occs$ is the occurrence of the term $gy$ in $\singleton{p \tightf g y}$,
the existential quantification of $\fograph$ by $x$ at $\occs$ is
$\mkern2mu\namedsingletonleft x x \mkern12mu\namedsingletonright v {p \tightf x}\e x v \mkern7mu\singleton{\pp\tightf g y}\mkern2mu$,
while if $\occs$ is empty
the existential quantification becomes
$\mkern2mu\namedsingletonleft x x \mkern12mu\namedsingletonright v {p \tightf g y}\e x v \mkern7mu\singleton{\pp\tightf g y}\mkern2mu$.
If $\portion=\setof{\singleton{p \tightf g y},\singleton{\pp \tightf gy}}$ and $\occs$ comprises both occurrences of the term $\tightf gy$ in $\portion$, then the existential quantification is
$\mkern2mu\namedsingletonleft x x \mkern12mu\namedsingletonright v{\px}\e x v \mkern7mu\namedsingletonright w {\ppx}\nccurve[nodesep=0ex,angleA=25,angleB=145] x w\mkern2mu$
\begin{lemma}\label{lem:existential-sound}
  Every existential quantification of a valid rectified fograph is valid.
\end{lemma}
\begin{proof}
  Let $\fographa$ be the existential quantification of a valid rectified fograph $\fograph$ by $x$ at a set $\occs$ of occurrences of the term $t$ in the non-empty portion $\portion$.
Thus $\fographa=\singletonx\graphunion\fograph\occsubst t \omega x$ plus edges from $\singletonx$ to every vertex in $\portion$.
We consider two cases.
\begin{enumerate}
\item
Suppose $\portion\tighteq\verticesof\fograph$.
Thus $\fographa=\singletonx\graphjoin\fograph\occsubst t \omega x$.
By Lemma\,\ref{lem:graph-surj} there exists a formula $\formula$ such that $\graphof\formula=\fograph\occsubst t \omega x$.
Therefore
$\fographa=\singletonx\graphjoin\fograph\occsubst t \omega x=\singletonx\graphjoin\graphof{\formula}=\graphof{\exists x\mkern2mu\formula}$.
Since $x$ does not occur in $\fograph$ we have $\graphof{\formula\assignment{\assign x t}}=\fograph$, and $\isvalid\formula\assignment{\assign x t}$ since $\isvalid\fograph$.
By Lemma\,\ref{lem:formula-inferences}.\ref{itm:solo-existential-inference} we have $\isvalid\exists x\mkern2mu\formula$ since
$\isvalid\formula\assignment{\assign x t}$, thus $\isvalid\fographa$.
\item
  Otherwise $\emptyset\neq\portion\neq\verticesof\fograph$.
  Let $\dualportion=\verticesof\fograph\mkern-1mu\tightsetminus\portion\neq\emptyset$.
  Since $\portion$ is a portion, it is well-founded and closed under adjacency and binding,
  $\fograph\occsubst t \omega x=\induced{\fograph\occsubst t \omega x}{\dualportion}\graphunion\induced{\fograph\occsubst t \omega x}{\portion}$ with \mbox{$\induced{\fograph\occsubst t \omega x}{\dualportion}$} and $\induced{\fograph\occsubst t \omega x}{\portion}$ both rectified fographs, and $\induced{\fograph\occsubst t \omega x}{\dualportion}=\induced\fograph\dualportion$ since $\omega$ does not intersect $\dualportion$. Thus $\fograph\occsubst t \omega x=\induced\fograph\dualportion\graphunion\induced{\fograph\occsubst t \omega x}{\portion}$.
  By Lemma\,\ref{lem:graph-surj} there exist formulas $\formulaa^*$ and $\formulaa$ with $\graphof{\formulaa^*}=\induced\fograph\dualportion$ and $\graphof\formulaa=\induced{\fograph\occsubst t \omega x}{\portion}$.
  Thus
  $$\fographa
  \hspace{1ex} = \hspace{1ex}
  \induced\fograph\dualportion
  \graphunion
  \singletonx\graphjoin\induced{\fograph\occsubst t \omega x}{\portion}
  \hspace{1ex} = \hspace{1ex}
  \graphof{\formulaa^*}
  \graphunion
  \graphof{\exists x\mkern2mu\formulaa}
  \hspace{1ex} = \hspace{1ex}
  \graphof{\formulaa^*\tightvee\exists x\mkern2mu\formulaa}$$
  Since $\graphof\formulaa=\induced{\fograph\occsubst t \omega x}{\portion}$ and $x$ does not occur in $\fograph$ we have $\graphof{\formulaa\assignment{\assign x t}}=\induced{\fograph}{\portion}$. Thus
  $$\fograph
  \hspace{1ex} = \hspace{1ex}
  \induced\fograph\dualportion
  \graphunion
  \induced{\fograph}{\portion}
  \hspace{1ex} = \hspace{1ex}
  \graphof{\formulaa^*}
  \graphunion
  \graphof{\formulaa\assignment{\assign x t}}
  \hspace{1ex} = \hspace{1ex}
  \graphof{\formulaa^*\tightvee\formulaa\assignment{\assign x t}}$$
  Since $\isvalid\fograph$ we have $\isvalid\formulaa^*\tightvee\formula\assignment{\assign x t}$,
  so by  Lemma\,\ref{lem:formula-inferences}.\ref{itm:existential-inference}
  we have $\isvalid\formulaa^*\tightvee\exists x\mkern2mu\formulaa$,
  hence $\isvalid\fographa$.
\end{enumerate}
\vskip-4ex\end{proof}
When quantifying a coloured rectified fograph existentially, the colouring is inherited, while the added binder remains uncoloured.
For example, if $\cover=\mkern2mu\singletonred{p \tightf g y} \mkern7mu\singletonred{\p\tightf gy}\mkern2mu$,
$\portion=\setof{\singletonred{p \tightf g y}}$
and
$\occs$ is the occurrence of
$y$ in $\singletonred{p \tightf g y}$,
the existential quantification of $\cover$ by $x$ at $\occs$ in $\portion$ is
$\mkern2mu\namedsingletonleft x x \mkern12mu\singletonred{p \tightf g x}\e x v \mkern7mu\singletonred{\pp\tightf g y}\mkern2mu$.
In the remainder of this section (\S\ref{sec:existential}) we prove that every existential quantification of a rectified fonet is a rectified fonet (Lemma\,\ref{lem:pres-existential}).

Let $\cover$ be a linked rectified fograph.
An \defn{existential} (resp.\ \defn{universal}) variable of $\cover$ is one labelling an existential (resp.\ universal) binder in $\cover$.
An \defn{output} of a function is any element of its image.
A \defn{stem} of a dualizer $\dualizer$ for $\cover$ is a variable in an output of $\dualizer$ but not in $\cover$.
For example, if $\cover\;=\;\stemeg\;$ and $\stem$ and $\stema$ are variables,
the dualizer $\stemegmin$ has one stem $\stem$,
$\assignment{\assign x {f\stem\stema},\assign y {f\stem\stema}}$ has two stems $\stem$ and $\stema$,
$\assignment{\assign x {f\stem z},\assign y {f\stem z}}$ has one stem $\stem$,
and $\stemeguniversal$ has no stem.
A dualizer $\dualizer$ \defn{generalizes} a dualizer $\dualizerp$ if $\dualizer$ yields $\dualizerp$ by substituting terms for stems, \ie, there exists a function $\subst$ from the stems of $\dualizer$ to terms such that $\dualizerp(x)=\dualizer(x)\subst$ for every existential variable $x$ of $\cover$, where $\termoratom\subst$ denotes the result of substituting $\subst(\stem)$ for $\stem$ in $\termoratom$, simultaneously for each stem $\stem$ of $\dualizer$.
For example, if $\cover\;=\;\stemeg\;$ and $\stem$ is a variable, the dualizer $\dualizer\mkern-1mu=\stemegmin$ generalizes $\dualizerp\mkern-2mu=\assignment{\assign x {f z a},\assign y {f z a}}$ via $\genegvia$ since $\dualizerp(x)=\dualizer(y)=\stem\genegvia=fza$.
A dualizer $\dualizer$ is \defn{most general} if it generalizes every other dualizer.
For example,
$\stemegmin$ is a most general dualizer for $\;\stemeg\;$ but
$\assignment{\assign x {f \zone},\assign y {f \zone}}$ and
$\stemeguniversal$ are not.
A linked rectified cograph is \defn{dualizable} if it has a dualizer.
\begin{lemma}\label{lem:mgd}
  Every dualizable linked rectified fograph has a most general dualizer.
\end{lemma}
\begin{proof}
  Let $\cover$ be the dualizable linked rectified fograph.
  Every dualizer for $\cover$ is, by definition,
  a unifier for the unification problem $\unirelof\cover$
  (binary relation on terms) \cite[\S7.2]{TS96}
  defined by $t_i\unirelof\cover t'_i$ for each link
  $\setof{\singleton{p t_1\ldots t_n},\singleton{\pp t'_1\ldots t'_n}}$ and $1\le i\le n$, solved for the existential variables.
  Let $\dualizer$ be a most general unifier of $\unirelof\cover$ \cite[\S7.2]{TS96}.
  By renaming variables as needed, we may assume that no output of $\dualizer$ contains an existential variable.
  Define $\dualizerp$ as the restriction of $\dualizer$ to existential variables.
  Since $\dualizer$ is a most general unifier, $\dualizerp$ is a most general dualizer.
\end{proof}
Let $\cover$ be a linked rectified fograph with dualizer $\dualizer$.
A pair $\dep$ is a \defn{dependency} of $\dualizer$ if
$\singletonx$ is existential, $\singletony$ is universal, and $\dualizer(x)$ contains $y$.

\begin{lemma}\label{lem:mgd-deps}
  Let $\cover$ be a linked rectified fograph with a most general dualizer $\dualizer$. A pair $\dep$ is a dependency of $\cover$ if and only if $\dep$ is a dependency of $\dualizer$.
\end{lemma}
\begin{proof}
  Since $\dualizer$ is most general, for any dualizer $\dualizerp$ every dependency of $\dualizer$ is a dependency of $\dualizerp$.
  By definition, $\dep$ is a dependency of $\cover$ if and only if it is a dependency of every dualizer for $\cover$.
  Thus $\dep$ is a dependency of $\cover$ if and only if it is a dependency of $\dualizer$.
\end{proof}
\begin{lemma}\label{lem:pres-existential}
  Every existential quantification of a rectified fonet is a rectified fonet.
\end{lemma}
\begin{proof}
Let $\coverp$ be the existential quantification of $\cover$ by $x$ at $\occs$ in $\portion$, where $\occs$ is a set of occurrences of the term $t$ in labels of literals in $\portion$.
Since $\portion$ is closed under adjacency, $\coverp$ is a cograph and every binder scope in $\coverp$ contains a literal.

In the following two paragraphs we will show that the dependencies of $\cover$ and $\coverp$ coincide.

For any dualizer $\dualizer$ for $\cover$, the function
$\dualizerp=\dualizer\cup\assignment{\assign x t}$ is a dualizer for $\coverp$, since the links of $\coverp$ are those of $\cover$ but for some occurrences of $t$ becoming $x$.
The dependencies of $\dualizer$ in $\cover$ are the same as those of $\dualizerp$ in $\coverp$,
since $t$ contains no binder variable of $\cover$.
Every dependency of $\coverp$ is a dependency of $\cover$:
a dependency of $\coverp$ is (by definition) a dependency of every dualizer of $\coverp$,
hence a dependency of $\dualizerp$ for every dualizer $\dualizer$ for $\cover$,
thus a dependency of $\cover$.

Conversely, to show that every dependency of $\cover$ is a dependency of $\coverp$,
we take a most general dualizer $\dualizera$ for $\coverp$ and construct a dualizer
$\hat\dualizera$ for $\cover$ with the same dependencies as $\dualizera$;
since a dependency of $\cover$ is (by definition) a dependency of every dualizer of $\cover$,
it is a dependency of $\hat\dualizera$ in $\cover$,
hence a dependency of $\dualizera$ in $\coverp$,
and therefore a dependency of $\coverp$ by \reflem{lem:mgd-deps} (since $\dualizera$ is most general).
Let $\dualizer$ be a most general dualizer for $\cover$.
By the argument in the previous paragraph, $\dualizerp=\dualizer\cup\assignment{\assign x t}$ is a dualizer for $\coverp$.
Since $\dualizera$ is most general for $\coverp$,
there exists a function $\subst$ from the stems of $\dualizera$ to terms such that
$t=\dualizerp(x)=\dualizera(x)\subst$.
Let $\tilde\subst$ be the restriction of $\subst$ to stems appearing in $\dualizera(x)$.
Define $\tilde\dualizera$ by $\tilde\dualizera(y)=\dualizera(y)\tilde\subst$, for every existential variable $y$ of $\coverp$.
In particular, $\tilde\dualizera(x)=t$.
The function $\tilde\dualizera$ is a dualizer for $\coverp$ (since it is $\dualizera$ with terms substituted for stems), and has the same dependencies as $\dualizera$ because $\dualizera(x)\tilde\subst=t$ so $\tilde\subst(z)$ is a sub-term of $t$ for every stem $z$ of $\dualizera$ in $\dualizera(x)$, and $t$ contains no bound variable of $\cover$, hence no bound variable of $\coverp$.
Define $\hat\dualizera$ as the restriction of $\tilde\dualizera$ to the existential variables of $\cover$
(thus $\tilde\dualizera=\hat\dualizera\cup\assignment{\assign x t}$).
The function $\hat\dualizera$ is a dualizer for $\cover$ since for every link $\setof{\singleton{p t_1\ldots t_n},\singleton{\pp u_1\ldots u_n}}$ in $\cover$ we have $t_i\hat\dualizera\tighteq u_i\hat\dualizera$, because for the corresponding link $\setof{\singleton{p t'_1\ldots t'_n},\singleton{\pp u'_1\ldots u'_n}}$ in $\coverp$ we have $t'_i\tilde\dualizera=u'_i\tilde\dualizera$ with
$t_i=t'_i\assignment{\assign x t}$
and
$u_i=u'_i\assignment{\assign x t}$, and by construction $\tilde\dualizera(x)=t$.
The dualizer $\hat\dualizera$ is a restriction of $\tilde\dualizera$, which has the same dependencies as $\dualizera$, thus $\hat\dualizera$ has the same dependencies as $\dualizera$.
Thus, by the argument at the start of this paragraph, every dependency of $\cover$ is a dependency of $\coverp$.

Since the dependencies of $\cover$ and $\coverp$ coincide, the leap graphs $\leapgraphof\cover$ and $\leapgraphof\coverp$ are identical but for an extra vertex $\singletonx$ in the latter which is not in any leap. Thus induced bimatchings of $\cover$ and $\coverp$ coincide, so $\coverp$ is a fonet because $\cover$ is a fonet.
\end{proof}

\subsubsection{Soundness of fonets}

An \defn{axiom} is a coloured rectified fograph comprising two dual literals of the same colour (\eg\ $\axiomeg$) or a single (uncoloured) $1$-literal.

\begin{lemma}\label{lem:construct-implies-net}
  Every coloured rectified fograph constructed from axioms by fusion and quantification is a rectified fonet.
\end{lemma}
\begin{proof}
  Every axiom is a rectified fonet, and fusion and quantification preserve the property of being a rectified fonet, by Lemmas\,\ref{lem:pres-fusion}, \ref{lem:pres-universal}, and \ref{lem:pres-existential}.
\end{proof}
A fonet is \defn{universal} if it has a binder in no edge (necessarily a universal binder).
\begin{lemma}\label{lem:split-universal}
  Every universal rectified fonet is a universal quantification of a rectified fonet.
\end{lemma}
\begin{proof}
Let $\net$ be a universal rectified fonet, with (universal) binder $\singletonx$ in no edge.
The result $\net^-$ of deleting $\singletonx$ from $\net$ is a fonet,
since $\net^-$ inherits all dualizers from $\net$ (because $x$ goes from being universal to being free) and if $W$ induces a bimatching in $\net^-$ then $W$ induces a bimatching in $\net$.
Since $\net^-$ is an induced subgraph of a rectified fograph, $\net^-$ is rectified.
Since $\net=\singletonx\graphunion\net^-$, the rectified fonet $\net$ is the universal quantification of the rectified fonet $\net^-$ by $x$.
\end{proof}
\begin{lemma}\label{lem:axiom-union}
  Every fonet with no edge and no binder is a union $\lambda_1\tightgraphunion\ldots\tightgraphunion\lambda_n$ of axioms $\lambda_i$ ($n\tightge 1$).
\end{lemma}
\begin{proof}
Since $\net$ has no edges, it has no existential binders, hence the empty dualizer.
Thus every link in $\net$ has literals with dual atoms, and every literal that is not in a link is $1$-labelled.
Since $\net$ has no edges, it is the union of axioms.
\end{proof}
\begin{lemma}\label{lem:split-fusion}
Let $\net$ be a rectified fonet with underlying uncoloured fograph
$\fographaa_1
\mkern1mu\graphunion\mkern1mu(\fographa_1\mkern-2mu\graphjoin\mkern-2mu\fographa_2)
\mkern1mu\graphunion\mkern1mu\fographaa_2$
for each $\fographa_i$ a fograph and each $\fographaa_j$ empty or a fograph.
Suppose no leap of $\net$ is between $\verticesof{\fographaa_1}\cup\verticesof{\fographa_1}$ and $\verticesof{\fographa_2}\cup\verticesof{\fographaa_2}$.
Then $\net$ is a fusion of rectified fonets.
\end{lemma}
\begin{proof}
  Since $\net$ is a fograph and no leap goes between $\verticesof{\fographaa_1}\cup\verticesof{\fographa_1}$ and $\verticesof{\fographa_2}\cup\verticesof{\fographaa_2}$, the graphs
\mbox{$\fographaa_1\graphunion\fographa_1$} and
\mbox{$\fographa_2\graphunion\fographaa_2$}
are well-defined fonets upon inheriting colouring from
$\net$
by restriction. Thus
$\net$ is a fusion of rectified fonets
$\fographaa_1\graphunion\fographa_1$ and $\fographa_2\graphunion\fographaa_2$ at portions
$\verticesof{\fographa_1}$ and $\verticesof{\fographa_2}$.
\end{proof}
\begin{lemma}\label{lem:no-leap}
Let $\net$ be a rectified fonet with underlying uncoloured fograph
$\fographaa_1\mkern1mu\graphunion\mkern1mu(\singletonx\graphjoin\fographa)\mkern1mu\graphunion\mkern1mu\fographaa_2$ for $\fographa$ a fograph and each $\fographaa_i$ empty or a fograph.
Suppose no leap of $\net$ is between $\verticesof{\fographaa_1}\cup\setof{\singletonx}$ and $\verticesof\fographa\cup\verticesof{\fographaa_2}$.
Then the binder $\singletonx$ is in no leap of $\net$.
\end{lemma}
\begin{proof}
In this proof \emph{leap supposition} refers to the supposition on leaps in the Lemma statement.
  Suppose for a contradiction that $\setof{\singletonx,\singletony}$ is a leap, hence dependency, of $\net$. By the leap supposition, the universal binder $\singletony$ is in $\fographaa_1$.
  Let $\dualizer$ be a most general dualizer for $\net$, which exists by \reflem{lem:mgd}. Since $\setof{\singletonx,\singletony}$ is a dependency, the term $\dualizer(\singletonx)$ contains $y$, by Lemma\,\ref{lem:mgd-deps}.
  There must be a link $\setof{v,w}$ such that the atom label of the literal $v$ contains $x$, otherwise $\dualizer(\singletonx)=z$ for a stem variable $z$ not occurring in $\net$, so $\dualizer(\singleton)$ would not contain $y$.
  Since $\net$ is rectified, the literal $v$ must be in the scope of $\singletonx$, thus $v$ is in $\fographa$.
  The atom label of $w$ cannot contain $y$, since $w$ would then be in $\fographaa_1$ (because $\net$ is rectified so $w$ must be in the scope of $\singletony$, which is in $\fographaa_1$), and $\setof{v,w}$ would be a link (hence leap) between $\fographa$ and $\fographaa_1$, contradicting the leap supposition.
  Thus, for $\dualizer(x)$ to be a term containing $y$, there must be a link $\setof{v,w}$ with the label of $v$ containing $x$ and the label of $w$ containing an existential variable $\xp$ such that the term $\dualizer(\xp)$ contains $y$.\todo{Elaborate by delving into unification problem equations?}
Therefore $\net$ has a leap $\setof{\singleton\xp,\singletony}$.
Since $v$ is in $\fographa$ and $\setof{v,w}$ is a link, hence a leap, by the leap supposition $w$ must be in $\fographa$ or $\fographaa_2$.  Because $\net$ is rectified, the literal $w$ must be in the scope of the existential binder $\singleton\xp$, so $\singleton\xp$ is in $\fographa$ or $\fographaa_2$. Since $\singletony$ is in $\fographaa_1$, the leap $\setof{\singleton\xp,\singletony}$ is between $\fographaa_1$ and $\fographa$ or $\fographaa_2$, contradicting the leap supposition.
\end{proof}
\begin{lemma}\label{lem:split-existential}
Let $\net$ be a rectified fonet with underlying uncoloured fograph
$\fographaa_1\mkern1mu\graphunion\mkern1mu(\singletonx\graphjoin\fographa)\mkern1mu\graphunion\mkern1mu\fographaa_2$ for $\fographa$ a fograph and each $\fographaa_i$ empty or a fograph.
Suppose no leap of $\net$ is between $\verticesof{\fographaa_1}\cup\setof{\singletonx}$ and $\verticesof\fographa\cup\verticesof{\fographaa_2}$.
Then $\net$ is an existential quantification of a rectified fonet by $x$.
\end{lemma}
\begin{proof}
  By Lemma\,\ref{lem:no-leap} the existential binder $\singletonx$ is in no leap of $\net$.
Let $\dualizer$ be a most general dualizer for $\net$ and let $t=\dualizer(x)$.
Define $\netp$ as the result of deleting $\singletonx$ from $\net$ and substituting $t$ for $x$ in the atom label of every literal.
Since $\net$ is a rectified fograph and $\singletonx$ is in no leap of $\net$, $\netp$ is a rectified fograph.\todo{elaborate}
Thus $\net$ is an existential quantification of $\netp$ by $x$ at $\occs$ in the portion $\verticesof{\fographa}$ for $\occs$ the set of occurrences of $t$ in $\netp$ which replaced occurrences of $x$ in $\net$ during the construction of $\netp$.
\end{proof}
The \defn{mate} of a literal in a link is the other literal in the link.
\begin{lemma}\label{lem:split-fusion-existential}
  Every non-universal rectified fonet with at least one edge is a fusion of rectified fonets or an existential quantification of a rectified fonet.
\end{lemma}
\begin{proof}
Let $\net$ be a non-universal fonet with an edge, and let $\fograph$ be its underlying uncoloured fograph.
Since $\fograph$ is a (labelled) cograph, it has the form
$\fograph=(\fograph_1\tightgraphjoin \fograph_2)\graphunion (\fograph_3\tightgraphjoin
\fograph_4)\graphunion\!\ldots\!\graphunion(\fograph_{n-1}\tightgraphjoin \fograph_n)\graphunion \fographaaa$ for (labelled) cographs $\fograph_i$
and $\fographaaa$, where $\fographaaa$ is a union of literals,
and $n\tightge 1$ since $\net$ (hence $\fograph$) has an edge.
Let $\megagraph$ be the graph whose vertices are the $\fograph_i$
with $\fograph_i\fograph_j\tightin E(\megagraph)$ if and only if $\net$ has an edge or leap
$\{v,w\}$ with $v\!\in \!V(\fograph_i)$ and $w\!\in\! V(\fograph_j)$.
A \emph{1-factor} is a set of pairwise disjoint edges whose
union contains all vertices.  Since $\net$ is a fonet, $Z=\{
\fograph_1\fograph_2,\fograph_3\fograph_4,\ldots,\fograph_{n-1}\fograph_n\}$ is the only 1-factor of
$\megagraph$.  For if $Z'$ is another 1-factor, then
$Z'\mkern-2mu\setminus\mkern-2mu Z$ determines a set of leaps in $\net$
whose union induces a bimatching in $\net$: for each
$\fograph_i\fograph_j\in Z'\mkern-2mu\setminus\mkern-2mu Z$ pick a leap
$\{v,w\}$ with $v\in V(\fograph_i)$ and $w\in V(\fograph_j)$.  Since $\megagraph$
has a unique 1-factor, some $\fograph_m\fograph_{m+1}\in Z$ is a bridge\todo{Related work: emphasize Retor\'e's use of bridge}
\cite{Kot59,LP86},
\textit{i.e.},
$\graphpairof{\verticesof \megagraph}{\edgesof \megagraph\mkern-1.5mu\setminus\mkern-1mu \fograph_m\fograph_{m+1}}=X\tightgraphunion
Y$ with $\fograph_m\!\in\! \verticesof X$ and $\fograph_{m+1}\!\in\! \verticesof Y$.\footnote{A similar construction of a unique 1-factor with a bridge is used in \cite{Hug06}, and \cite{Ret03} uses a related argument involving the existence of a bridge.}
Without loss of generality assume $G_i\tightin\verticesof X$ for $i\tightle m$ and
$G_j\tightin\verticesof Y$ for $j\tightge m+1$.
Let $\fographaaa_X$ be the restriction of $\fographaaa$ to literals with mate in a vertex of $X$, and let $\fographaaa_Y$ be the
restriction of $\fographaaa$ to literals not in $\fographaaa_X$.
Thus $\fographaaa\tighteq\fographaaa_X\tightgraphunion\fographaaa_Y$ since $\fographaaa$ contains only literals and no binders.
Define $\fographaa_1=\fographaaa_X\tightgraphunion(G_1\tightgraphjoin G_2)\tightgraphunion\ldots\tightgraphunion(G_{m-2}\tightgraphjoin G_{m-1})$
and $\fographaa_2=\fographaaa_Y\tightgraphunion(G_{m+2}\tightgraphjoin G_{m+3})\tightgraphunion\ldots\tightgraphunion(G_{n-2}\tightgraphjoin G_n)$,
so $\fograph=\fographaa_1\graphunion(\fograph_m\tightgraphjoin\fograph_{m+1})\graphunion\fographaa_2$.

Since $\fographaaa$ comprises literals only, each of $\fographaa_1$ and $\fographaa_2$ is either empty or a fograph.
If $\fograph_m$ and $\fograph_{m+1}$ both contain a literal, they are fographs, so we can appeal to Lemma\,\ref{lem:split-fusion} with $\fographa_1=\fograph_m$ and $\fographa_2=\fograph_{m+1}$ to conclude that $\net$ is a fusion of rectified fonets.
Otherwise one of $\fograph_m$ or $\fograph_{m+1}$, say $\fograph_m$, has no literal, thus $\fograph_m=\singletonx$. Then $\fograph_{m+1}$ must contain a literal, since $\fograph$ hence $\fograph_m\tightgraphjoin\fograph_{m+1}$ is a fograph, therefore $\fograph_{m+1}$ is a fograph.
Applying Lemma\,\ref{lem:split-existential}  with
$\fographa\tighteq\fograph_{m+1}$, we conclude that $\net$ is an existential quantification of a rectified fonet.
\end{proof}
\begin{lemma}\label{lem:fonet-constructible}
  Every rectified fonet can be constructed from axioms by fusion and quantification.
\end{lemma}
\begin{proof}
Let $\net$ be a rectified fonet.
We proceed by induction on the number of binders and edges in $\net$.
In the base case with no edge or binder, $\net$ is a union of axioms by Lemma\,\ref{lem:axiom-union}, hence a fusion of axioms since union is a special case of fusion (with empty portions).
If $\net$ is universal, apply Lemma\,\ref{lem:split-universal} then appeal to induction with one less binder.
Thus we may assume $\net$ is non-universal with a binder or edge.
Had $\net$ no edge, it would have no binder (since every existential binder must be in an edge, and
a universal binder would make $\net$ universal),
thus $\net$ has at least one edge.
Apply Lemma\,\ref{lem:split-fusion-existential} then appeal to induction with fewer edges.
\end{proof}

\begin{lemma}[Fonet soundness]\label{lem:fonet-soundness}
  Every fonet is valid.
\end{lemma}
\begin{proof}
  By Lemma\,\ref{lem:fonet-constructible} every fonet can be constructed from axioms by fusion and quantification. Since every axiom is valid, and fusion and quantification preserve validity by Lemmas\,\ref{lem:fusion-sound}, \ref{lem:universal-sound}, and \ref{lem:existential-sound}, every fonet is valid.
\end{proof}

\subsection{Soundness of skew bifibrations}\label{sec:soundness-bifibs}

In this section we no longer assume implicitly that every formula is rectified.

An \defn{intrusion} is a formula of the form
$\formula\tightvee\forall x\mkern2mu \formulaa$,
$(\forall x\mkern2mu \formulaa)\tightvee\formula$,
$\formula\tightwedge\exists x\mkern2mu \formulaa$, or
$(\exists x\mkern2mu \formulaa)\tightwedge\formula$.
A formula is \defn{extruded} if no subformula is an intrusion.
For any variable $x$, an \defn{$x$-quantifier} is a quantifier of the form $\forall x$ or $\exists x$.
A formula is \defn{unambiguous} if no $x$-quantifier is in the scope of another $x$-quantifier, for every variable $x$.
A formula is \defn{clear} if it is extruded and unambiguous.
\begin{definition}\label{def:xgraph}
The \defn{graph} $\xgraphof\formula$ of a clear formula
$\formula$ is the logical cograph defined inductively by:
\begin{center}\vspace{-2ex}\begin{math}
\begin{array}{c}
\xgraphof{\atom}
\;=\;
\singleton\atom \hspace{1ex} \text{ for every atom\/ }\atom
\\[2ex]
\def\eqgap{\hspace{3ex}}
\begin{array}{r@{\eqgap=\eqgap}l}
\xgraphof{\,\formula\tightvee\formulaa\,} & \xgraphof\formula\graphunion\xgraphof\formulaa
\\[1.5ex]
\xgraphof{\,\formula\tightwedge\formulaa\,} & \xgraphof\formula\graphjoin\xgraphof\formulaa
\\[1.5ex]
\end{array}
\hspace{16ex}
\begin{array}{r@{\eqgap=\eqgap}l}
\xgraphof{\,\forall x\, \formula\,} & \xgraphall {x\,} {\formula}
\\[1.5ex]
\xgraphof{\,\exists x\, \formula\,} & \xgraphex  {x\,} {\formula}
\\[1.5ex]
\end{array}\\[3ex]\end{array}\end{math}
\end{center}
\end{definition}
Note that $\xgraphofsymbol$ coincides with $\graphofsymbol$ (Def.\,\ref{def:graph}) on extruded rectified formulas.
\begin{lemma}\label{lem:xgraph-surj}
The function $\xgraphofsymbol$ is a surjection from clear formulas onto fographs.
Two clear formulas have the same graph if and only if they are equal modulo
  \begin{align*}
    \formula\tightwedge\formulaa &\fateq \formulaa\tightwedge\formula\;\;\;\;
    &
    \formula\wedge(\formulaa\tightwedge\formulaaa) &\fateq (\formula\tightwedge\formulaa)\wedge\formulaaa\;\;\;\;
    &
    \exists x\mkern2mu\exists y\mkern2mu \formula &\fateq \exists y\mkern2mu\exists x\mkern2mu\formula\;\;\;\;
    \\[1ex]
    \formula\tightvee\formulaa &\fateq \formulaa\tightvee\formula
    &
    \formula\vee(\formulaa\tightvee\formulaaa) &\fateq (\formula\tightvee\formulaa)\vee\formulaaa
    &
    \forall x\mkern2mu\forall y\mkern2mu \formula &\fateq \forall y\mkern2mu\forall x\mkern2mu\formula
\\[-2.5ex]
\end{align*}
\end{lemma}
\begin{proof}
  A routine induction, akin to the proof of Lemma\,\ref{lem:graph-surj}.\todo{elaborate in appendix}
\end{proof}
Let $\fograph$ be a fograph. Using the above Lemma,
choose a clear formula $\formula$ such that $\xgraphof\formula\tighteq\fograph$.
Define $\fograph$ as \defn{valid} if $\formula$ is valid.
This is well-defined with respect to choice of $\formula$ since every equality in Lemma\,\ref{lem:xgraph-surj} is a logical equivalence.

Fographs $\fograph$ and $\fographa$ are
\defn{$\wedge$-compatible} if
$\fograph\tightgraphjoin\fographa$ is a well-defined fograph and
$\widebindinggraphof{\fograph\mkern-2mu\tightgraphjoin\mkern-2mu\fographa}{7ex}=\bindinggraphof\fograph\graphunion\bindinggraphof{\fographa}$,
and
\defn{$\vee$-compatible} if
$\fograph\tightgraphunion\fographa$ is a well-defined fograph and
$\widebindinggraphof{\fograph\mkern-2mu\tightgraphunion\mkern-2mu\fographa}{7ex}=\bindinggraphof\fograph\graphunion\bindinggraphof{\fographa}$.
Thus $\vee$- and $\wedge$-compatibility ensure that no new bindings are created during graph union and join.
For any variable $x$, a fograph $\fograph$ is \defn{$x$-compatible} if $\fograph$ does not contain an $x$-binder $\singletonx$.
\begin{definition}\label{def:fograph-ops}
  Let $\fograph$ and $\fographa$ be fographs. Define the \defn{fograph connectives} $\wedge$, $\vee$, $\forall$ and $\exists$ by:
  \begin{itemize}
  \item if $\fograph$ and $\fographa$ are $\wedge$-compatible, define $\fograph\tightwedge\fographa\mkern2mu=\mkern2mu\fograph\tightgraphjoin\fographa$
  \item if $\fograph$ and $\fographa$ are $\vee$-compatible, define $\fograph\tightvee\fographa\mkern2mu=\mkern2mu\fograph\tightgraphunion\fographa$
  \item for any variable $x$, if $\fograph$ is $x$-compatible,
define $\forall x\mkern2mu\fograph\mkern2mu=\mkern2mu\singletonx \graphunion \fograph$
  \item for any variable $x$, if $\fograph$ is $x$-compatible,
define $\exists x\mkern2mu\fograph\mkern2mu=\mkern2mu\singletonx \graphjoin  \fograph$.
  \end{itemize}
\end{definition}
\begin{lemma}\label{lem:fograph-ops}
The fograph connectives $\wedge$, $\vee$, $\forall$ and $\exists$ are well-defined on fographs. In other words, given fographs as input(s),
each connective, when defined, produces a fograph as output.
\end{lemma}
\begin{proof}
  By the compatibility constraints, no $x$-binder of
$\fograph\tightwedge\fographa$,
$\fograph\tightvee\fographa$,
$\forall x\mkern2mu\fograph$, or
$\exists x\mkern2mu\fograph$ can be in the scope of another $x$-binder.
\end{proof}
\begin{lemma}\label{lem:xgraph-commute}
The following equalities hold for clear formulas:
\begin{center}\vspace{0ex}\begin{math}
\begin{array}{c}
\def\eqgap{\hspace{3ex}}
\begin{array}{r@{\eqgap=\eqgap}l}
\xgraphof{\,\formula\tightvee\formulaa\,} & \xgraphof\formula\mkern1mu\vee\mkern2mu\xgraphof\formulaa
\\[1.5ex]
\xgraphof{\,\formula\tightwedge\formulaa\,} & \xgraphof\formula\mkern1mu\wedge\mkern2mu\xgraphof\formulaa
\\[1.5ex]
\end{array}
\hspace{16ex}
\begin{array}{r@{\eqgap=\eqgap}l}
\xgraphof{\,\forall x\, \formula\,} & \forall x\mkern6mu \xgraphof\formula
\\[1.5ex]
\xgraphof{\,\exists x\, \formula\,} & \exists x\mkern6mu \xgraphof\formula
\\[1.5ex]
\end{array}\end{array}\end{math}
\end{center}
\end{lemma}
\begin{proof}
Since $\formula\tightvee\formulaa$ and $\formula\tightwedge\formulaa$ are clear, $\xgraphof\formula$ and $\xgraphof\formulaa$ are $\vee$- and $\wedge$-compatible, thus $\xgraphof\formula\vee\mkern1mu\xgraphof\formulaa$ and $\xgraphof\formula\wedge\mkern1mu\xgraphof\formulaa$ are well-defined.
Because $\forall x\mkern2mu \formula$ and $\exists x\mkern2mu\formula$ are clear, no $x$-quantifier occurs in $\formula$, so $\xgraphof\formula$ contains no binder $\singletonx$, thus $\forall x\mkern3mu \xgraphof\formula$ and $\exists x\mkern3mu \xgraphof\formula$ are well-defined.
\end{proof}
\begin{lemma}\label{lem:fograph-iff-constructed}
  A labelled graph is a fograph if and only if it can be constructed from literals by the fograph connectives $\wedge$, $\vee$, $\forall$ and $\exists$.
\end{lemma}
\begin{proof}
Let $\fograph$ be a fograph.
By Lemma\,\ref{lem:xgraph-surj} there exists a clear formula $\formula$
such that $\xgraphof\formula\tighteq\fograph$.
By Lemma\,\ref{lem:xgraph-commute} the $\graphjoin$ and $\graphunion$ operations in the inductive translation $\xgraphofsymbol$ of $\formula$ are well-defined $\wedge$, $\vee$, $\forall$ and $\exists$ operations on fographs. Thus $\fograph$ can be constructed from literals by
fograph connectives.
Conversely, any labelled graph constructed from literals by fograph connectives is a fograph, by repeated application of Lemma~\ref{lem:fograph-ops}, starting from the fact that any literal vertex is a fograph.
\end{proof}
A \defn{map} is a label-preserving graph homomorphism between fographs.
\begin{definition}
Extend the fograph connectives to maps $\map:\fograph\to\fographa$ and $\mapp:\fographp\to\fographap$
as follows:
\begin{itemize}
\item
if $\fograph\tightwedge\fographp$ and $\fographa\tightwedge\fographap$ are well-defined,
define $\map\tightwedge\mkern-1mu\mapp:\fograph\tightwedge\fographp\to\fographa\tightwedge\fographap$
as $\map\cup\mapp$
\item
if $\fograph\tightvee\fographp$ and $\fographa\tightvee\fographap$ are well-defined,
define $\map\tightvee\mkern-1mu\mapp:\fograph\tightvee\fographp\to\fographa\tightvee\fographap$
as $\map\cup\mapp$
\item
if $\forall x\mkern2mu \fograph$ and
if $\forall x\mkern2mu \fographa$ are well-defined,
define $\forall x\mkern2mu\map\mkern2mu:\forall x\mkern2mu\fograph\to\forall x\mkern2mu\fographa$
as $\map\cup\setof{\singletonx\mkern3mu\shortmapsto\mkern3mu\singletonx}$
\item
if $\exists x\mkern2mu \fograph$ and
if $\exists x\mkern2mu \fographa$ are well-defined,
define $\exists x\mkern2mu\map\mkern2mu:\exists x\mkern2mu\fograph\to\exists x\mkern2mu\fographa$
as $\map\cup\setof{\singletonx\mkern3mu\shortmapsto\mkern3mu\singletonx}$.
\end{itemize}
\end{definition}
\begin{lemma}\label{lem:fograph-connectives-bifibs}
The fograph connectives are well-defined on skew bifibrations: if $\map$ and $\mapp$ are skew bifibrations, then, when defined, each of the maps $\map\tightwedge\mapp$, $\map\tightvee\mapp$, $\forall x\mkern2mu\map$ and $\exists x\mkern2mu\map$ is a skew bifibration, where $x$ is any variable.
\end{lemma}
\begin{proof}
Due to the compatibility constraint in the definitions of the fograph connectives,
the skew fibration condition is preserved and the directed graph homomorphisms between binding graphs are fibrations.\todo{elaborate?}
In the $\wedge$ and $\exists$ connectives, additional requisite skew liftings are created across the corresponding graph join.
\end{proof}
\begin{lemma}\label{lem:bifibs-compose}
  Skew bifibrations between fographs compose: if $\bifib:\fograph\to\fographa$ and $\bifibp:\fographa\to\fographaa$ are skew bifibrations between fographs, their composite $\bifibp\mkern-1mu\tightcirc\bifib:\fograph\to\fographaa$ is a skew bifibration.
\end{lemma}
\begin{proof}
  Skew fibrations between cographs compose \cite[Cor.\:3.5]{Hug06i}, and directed graph fibrations compose \cite{Gro60}. Existential preservation is transitive.
\end{proof}
\begin{definition}
  If $\fograph$ is a fograph and $\fograph\vee\fograph$ is well-defined, define \defn{pure contraction} $\contractionof\fograph$ as the canonical map $\fograph\vee\fograph\to\fograph$.
  If $\fograph$ and $\fographa$ are fographs and $\fograph\vee\fographa$ is well-defined, define \defn{pure weakening} $\weakeningof\fograph\fographa$ as the canonical map $\fograph\to\fograph\vee\fographa$.
\end{definition}
\begin{lemma}\label{lem:pure-cw-are-bifibs}
  Every pure contraction and pure weakening is a skew bifibration.
\end{lemma}
\begin{proof}
  Immediate from the definitions of pure contraction and pure weakening.\todo{elaborate}
\end{proof}
\begin{definition}\label{def:contraction-weakening-map}
A \defn{contraction} is any map generated from a pure contraction by fograph connectives, and a \defn{weakening} is any map generated from a pure weakening by fograph connectives.
\end{definition}
\begin{lemma}\label{lem:cw-are-bifibs}
Every contraction and weakening is a skew bifibration.
\end{lemma}
\begin{proof}
Pure contraction and pure weakening are skew bifibrations by \reflem{lem:pure-cw-are-bifibs}, and
fograph connectives are well-defined on skew bifibrations by Lemma\,\ref{lem:fograph-connectives-bifibs}.
\end{proof}
\begin{definition}
  A \defn{structural map} is any map constructed from isomorphisms, contractions, and weakenings by composition.
\end{definition}
\begin{lemma}\label{lem:structural-map-is-bifib}
  Every structural map is a skew bifibration.
\end{lemma}
\begin{proof}
  Every isomorphism is a skew bifibration, and every contraction and weakening is a skew bifibration by Lemma\,\ref{lem:cw-are-bifibs}.
  Skew bifibrations compose by Lemma\,\ref{lem:bifibs-compose}.
\end{proof}
\begin{lemma}\label{lem:strucmap-sound}
Structural maps are sound: if $\fograph$ is a valid fograph and $\map:\fograph\to\fographa$ is a structural map, then $\fographa$ is valid.
\end{lemma}
\begin{proof}
  Isomorphisms, pure contraction, pure weakening, composition and fograph connectives are sound.
\end{proof}

\subsubsection{The image of a skew bifibration is a fograph}\label{sec:image}

We recall the \emph{modular decomposition} \cite{Gal67} of a cograph, called its \defn{cotree} \cite{CLS81}.

A directed graph $\defaultplustimestree$ is \defn{acyclic} if the transitive closure of ${}\child{}$ (viewed as a binary relation on $\nodeset$) is irreflexive.
A \defn{forest} is an acyclic directed graph $\defaultplustimestree$ such that for every $n\tightin \nodeset$ there exists at most one $m\tightin \nodeset$ with $\diredge nm\in{}\mkern-3mu\child{}$.
We refer to the vertices of a forest as \defn{nodes}.
Write $m\child n$ or $n\parent m$ for $\diredge n m\in{}\mkern-3mu\child{}$, and say that $m$ is a \defn{child} of $n$ and $n$ is the \defn{parent} of $m$.
A \defn{leaf} (resp.\ \defn{root}) is a node with no child (resp.\ parent).
A \defn{tree} is a forest with a unique root.
A \defn{\plustimestree} is a tree in which
a node is labelled $\graphunion$ or $\graphjoin$ if and only if it is not a leaf.
Each node labelled $\graphunion$ or $\graphjoin$ is a
\defn{\plustimesnode}.
An isomorphism $\iota:\defaultplustimestree\to\defaultplustimestreep$ of \plustimestrees is a bijection
$\iota:\nodeset\to\nodesetp$ such that
$m\child n$ if and only if $\iota(m)\mkern1mu\childp\mkern2mu\iota(n)$ and
$\iota(n)$ is a $\graphunion$ (resp.\ $\graphjoin$) node
if and only if $n$ is a $\graphunion$ (resp.\ $\graphjoin$) node.
We identify \plustimestrees up to isomorphism.

Given \plustimestrees $\tree_1,\ldots,\tree_n$ for $n\ge 1$ define $\graphunion\tree_1\ldots\tree_n$ (resp.\ $\graphjoin\tree_1\ldots\tree_n$) as the disjoint union of the $\tree_i$ together with a $\graphunion$ (resp.\ $\graphjoin$) root node $r$ and an edge to $r$ from the root of each $\tree_i$ ($1\tightle i\tightle n$).
Write $\singleton{}$ for the \plustimestree with a unique node.
For example, the \plustimestree
$\graphunion(\graphjoin\singleton{}\singleton{})\singleton{}\mkern-4mu\left(\rule{0ex}{1.7ex}\mkern-4mu\graphjoin\mkern-4mu\singleton{}\singleton{}(\graphunion\singleton{}\singleton{})\right)$
is below-left and
$\graphunion(\graphjoin\singleton{}\singleton{})(\graphunion\singleton{})\mkern-4mu\left(\rule{0ex}{1.7ex}\mkern-4mu\graphjoin\mkern-4mu\singleton{}(\graphjoin\singleton{}(\graphunion\singleton{}\singleton{}))\right)$
is below-right.
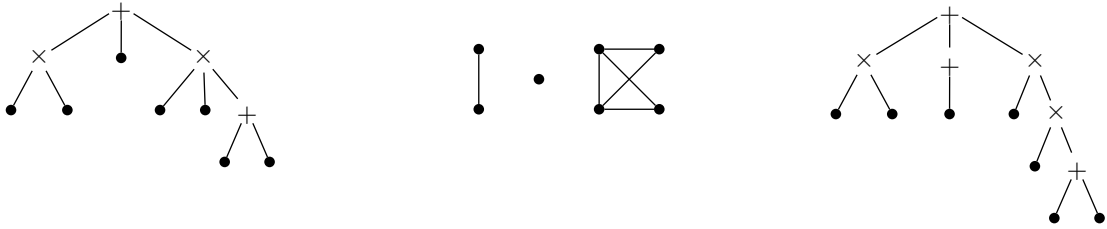
\begin{center}\begin{pspicture}(0,\twoedgelen)(0,-\twoedgelen)\begin{math}
\rput(-5.5,\halfedgelen){\defaultcotreeeg}
\rput(-.2,0){
  \vx{-\edgelen,\edgelen}{a}
  \vx{-\edgelen,0}{b}
  \e a b
  \vx{0,\halfedgelen}{c}
  \rput(\edgelen,0){
    \vx{0,\edgelen}{d}
    \vx{0,0}{e}
    \vx{\edgelen,\edgelen}{f}
    \vx{\edgelen,0}{g}
    \e d e
    \e d f
    \e e f
    \e d g
    \e e g
  }}
\rput(5.5,0){
  \plustreesep{.75}{
    \timestreesep{.75}{\lf\lf}
    \plustree{\lf}
    \timestreesep{.6}{
      \lf
      \tspace{-.235}
      \timestreesep{.6}{
        \lf
        \tspace{-.235}
        \plustreesep{.6}{\lf\lf}
      }
    }
  }
}
\end{math}\end{pspicture}\end{center}
\begin{definition}\label{def:cograph-of}
The \defn{cograph} $\cographof\tree$ of a \plustimestree $\tree$ is the cograph defined inductively by
\[
\cographof{\singleton{}}
\hspace{.5ex}=\hspace{.5ex}
\singleton{}
\hspace{6ex}
\cographof{\graphunion\tree_1\ldots\tree_n}
\hspace{1ex}=\hspace{1ex}
\cographof{\tree_1}\graphunion\ldots\graphunion\cographof{\tree_n}
\hspace{6ex}
\cographof{\graphjoin\tree_1\ldots\tree_n}
\hspace{1ex}=\hspace{1ex}
\cographof{\tree_1}\graphjoin\ldots\graphjoin\cographof{\tree_n}
\]
\end{definition}
For example, the cograph of the \plustimestree above-left is shown above-center; this cograph is also the cograph of the \plustimestree above-right.
\begin{lemma}\label{lem:leaf-to-vertex}
  The leaves of a \plustimestree $\tree$ are in bijection with the vertices of its cograph $\cographof\tree$.
\end{lemma}
\begin{proof}
  Induction on the number of vertices in $\cograph$, pattern-matching the three cases in
Def.\,\ref{def:cograph-of}.
\end{proof}
A \plustimesnode{} \defn{repeats} if it has a parent with the same label, and is \defn{unary} if it has a unique child.
A \plustimestree{} \defn{alternates} if it has no repeating \plustimesnode and
\defn{branches}
if it has no unary \plustimesnode.
\begin{definition}\label{def:cotree}
  A \defn{cotree} is a branching and alternating \plustimestree.
\end{definition}
For example, the \plustimestree above-left is a cotree, while the \plustimestree above-right is not (since it has a repeating $\graphjoin$ node and a unary and repeating $\graphunion$ node).
We recall the following definition from \cite{CLS81}.
\begin{definition}\label{def:cotree-of}
The \defn{cotree} $\cotreeof\cograph$ of a cograph $\cograph$ is the cotree defined inductively by
\begin{center}\begin{math}
\cotreeof{\singleton{}}
\hspace{1ex}=\hspace{1ex}
\singleton{}
\hspace{5ex}
\begin{array}{r@{\hspace{1ex}=\hspace{1ex}}l@{\hspace{2ex}}l}
\cotreeof{\cograph_1\graphunion\ldots\graphunion\cograph_n}
&
\graphunion\cotreeof{\cograph_1}\ldots\cotreeof{\cograph_n}
&
\text{if $\cograph_i$ is connected for $1\tightle i\tightle n\tightge 2$}
\\[2ex]
\cotreeof{\cograph_1\graphjoin\ldots\graphjoin\cograph_n}
&
\graphjoin\cotreeof{\cograph_1}\ldots\cotreeof{\cograph_n}
&
\text{if $\cograph_i$ is coconnected for $1\tightle i\tightle n\tightge 2\mkern2mu$.\!\!\!}
\end{array}\end{math}\end{center}\end{definition}
The following Lemma articulates a standard property of cotrees.
Recall from \S\ref{sec:proper} that a module in a graph is \emph{proper} if it has two or more vertices.
A module $\module$ of a cograph $\graph$ is \defn{connected} (resp.\ \defn{coconnected}) if the induced subgraph $\induced\graph\module$ is connected (resp.\ coconnected).
\begin{lemma}\label{lem:node-strong-module}
The nodes of the cotree $\cotreeof\graph$ of a cograph $\cograph$ correspond to the strong modules of $\graph$,
and the $\graphjoin$ (resp.\ $\graphunion$) nodes correspond to proper connected (resp.\ coconnected) strong modules.
\end{lemma}
\begin{proof}
  Induction on the number of vertices in $\graph$ \cite{CLS81}.
\end{proof}
The following Lemma is also a standard cotree property.
\begin{lemma}\label{lem:cotree-cograph}
The function $\cotreeof{-}$
is a bijection from cographs to cotrees.
\end{lemma}
\begin{proof}
  Induction on the number of vertices in the cograph \cite{CLS81}.
\end{proof}

\begin{lemma}\label{lem:unique-branching-alternating}
  The cotree $\cotreeof\cograph$ of a cograph $\cograph$ is the unique branching and alternating \plustimestree $\tree$ such that $\cographof\tree=\cograph$.
\end{lemma}
\begin{proof}
  A routine induction on the number of vertices in $\graph$.
\end{proof}
\begin{lemma}\label{lem:vertex-to-leaf}
  The vertices of a cograph $\cograph$ are in bijection with the leaves of its cotree $\cotreeof\cograph$.
\end{lemma}
\begin{proof}
  Lemmas\,\ref{lem:leaf-to-vertex} and \ref{lem:unique-branching-alternating}.
\end{proof}
Let $\node$ be a node in a tree $\tree=\defaultplustimestree$.
Define the \defn{absorption} $\absorption\tree\node$ of $\node$ in $\tree$ as the result of deleting $\node$ (and incident edges) from $\tree$ and, if $\node$ has a parent $\parentof\node$, adding an edge from each child of $\node$ to $\parentof\node$.
Thus
$\nodesetof{\absorptionsub\tree\node}=
\nodesetof\tree\tightsetminus\mkern1mu\setof{\mkern-2mu\node\mkern-2mu}$ and
$m{}\childin{\absorptionsub\tree\node}{}m\primed$ if and only if $m{}\,\childin\tree{}\,m\primed$
or
$m\,{}\childin\tree{}\,\node\,{}\childin\tree{}\,m\primed$.
\begin{definition}\label{def:cotree-of-tree}
Given a \plustimestree $\tree$ define its \defn{cotree} $\absorbed\tree$ as
the cotree obtained by iteratively and exhaustively absorbing unary \plustimesnodes and repeating \plustimesnodes in $\tree$.
\end{definition}
For example, if $\tree$ is the \plustimestree above-right of Def.\,\ref{def:cograph-of} then its cotree $\absorbed\tree$ is above-left of Def.\,\ref{def:cograph-of}.
\begin{lemma}\label{lem:cograph-absorb}
  $\cographof\tree=\cographof{\absorbed\tree}$ for every \plustimestree $\tree$.
\end{lemma}
\begin{proof}
  By induction on the number of nodes in $\tree$, pattern-matching the three cases in Def.\,\ref{def:cograph-of}, combined with the associativity and commutativity of the graph union $\graphunion$ and join $\graphjoin$ operations.
\end{proof}
Recall that $\induced\graph U$ is the subgraph of a graph $\graph$ induced by a set of vertices $U$.
Define the \plustimestree $\induced\tree U$ \defn{induced} by a non-empty set of leaves $U$ in a \plustimestree $\tree$ by deleting from $\tree$ every leaf not in $U$, and then iteratively and exhaustively deleting any resulting childless \plustimesnodes.
For example, if $\tree$ is the cotree below-left and $U$ comprises the left-most four leaves of $\tree$, then the \plustimestree $\induced\tree U$ is below-center, and the cotree $\absorbed{\induced\tree U}$ is below-right.
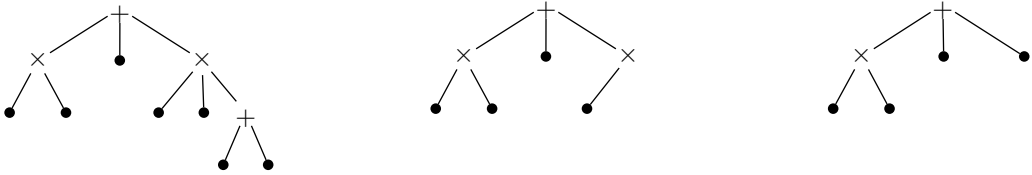
\begin{center}\begin{pspicture}(0,\twoedgelen)(0,-\twoedgelen)\begin{math}
\rput(-5,\halfedgelen){\defaultcotreeeg}
\rput(.5,\edgelen){\plustreesep{.9}{%
  \timestreesep{.75}{\lf\lf}%
    \lf%
    \timestreesep{.6}{%
      \lf%
     {}\hspace*{6.5ex}{}
    }%
  }%
}%
\rput(5.5,\edgelen){\plustreesep{.9}{
  \timestreesep{.75}{\lf\lf}%
    \lf%
    {}\hspace*{1ex}{}
    \lf%
  }%
}
\end{math}\end{pspicture}\end{center}

\begin{lemma}\label{lem:induced-subtree}
  If $U$ is a non-empty set of leaves in a cotree $\tree$, then
  $\cographof{\induced\tree U}\,=\,\induced{\cographof\tree} U$.
\end{lemma}
\begin{proof}
  Induction on the number of nodes in $\tree$.
\end{proof}
\begin{lemma}\label{lem:cograph-cotree-induced}
If $U$ is a non-empty set of vertices in a cograph $\cograph$, then
$\cographof{\absorbed{\induced{\cotreeof\graph} U}}=\induced\graph U$.
\end{lemma}
\begin{proof}
By \reflem{lem:cograph-absorb},
$\cographof{\absorbed{\induced{\cotreeof\graph} U}}
=
\cographof{\induced{\cotreeof\graph} U}$, which is
$\induced{\cographof{\cotreeof\graph}} U$
by \reflem{lem:induced-subtree},
hence
$\induced\graph U$
by \reflem{lem:unique-branching-alternating}.
\end{proof}
\begin{lemma}\label{lem:induced-cotree}
  If $U$ is a non-empty set of vertices in a cograph $\cograph$, then
  $\cotreeof{\induced\graph U}\,=\,\absorbed{\induced{\cotreeof\graph} U}$.
\end{lemma}
\begin{proof}
By \reflem{lem:unique-branching-alternating} it suffices to show that
$\cographof{\absorbed{\induced{\cotreeof\graph} U}}=\induced\graph U$, which is \reflem{lem:cograph-cotree-induced}.
\end{proof}
Write $\nodesof\tree$ for the set of nodes of a tree $\tree$, ${}\childin\tree{}$ for its set of directed edges, ${}\belowin\tree{}$ for the transitive closure of ${}\childin\tree{}$, and ${}\atorbelowin\tree{}$ for the reflexive closure of ${}\belowin\tree{}$.
Define
$\nodea\abovein\tree\node$ as
$\node\belowin\tree\nodea$,
and say that
$\nodea$ is \defn{above} $\node$ or
$\node$ is \defn{below} $\nodea$;
define
$\nodea\atorabovein\tree\node$ as
$\node\atorbelowin\tree\nodea$,
and say that
$\nodea$ is \defn{at or above} $\node$ or
$\node$ is \defn{at or below} $\nodea$.
Define
the \defn{meet} $\nodea\meet\node$ of nodes $\nodea$ and $\node$ in a tree $\tree$ as the $\atorbelowin\tree\mkern2mu$-least node $\nodeaa$ with $\nodea\atorbelowin\tree\nodeaa$ and $\node\atorbelowin\tree\nodeaa$.
\begin{lemma}\label{lem:meet}
  Let $\graph$ be a cograph and $v,w\in\verticesof\graph$. Then $vw\tightin\edgesof\graph$ if and only if $v\meet w$ in the cotree $\cotreeof\graph$ is a $\graphjoin$ node and $vw\tightnotin\edgesof\graph$ if and only if $v\meet w$ is a $\graphunion$ node or $v\tighteq w$.
\end{lemma}
\begin{proof}
  This follows directly from \reflem{lem:node-strong-module}.
\end{proof}
Write $v\jmeet w$ (resp.\ $v\umeet w$) for $v\meet w$ if it is a $\graphjoin$ (resp.\ $\graphunion$) node.
For a cograph $\cograph$ write
${}\childin\cograph{}$,
${}\belowin\cograph{}$,
and
${}\atorbelowin\cograph{}$ for
${}\childin{\cotreeof\cograph}{}$,
${}\belowin{\cotreeof\cograph}{}$,
and
${}\atorbelowin{\cotreeof\cograph}{}$, respectively.
\begin{lemma}\label{lem:vee}
If $\graph$ is a cograph with $vw,vu\tightin\edgesof\graph$, $wu\tightnotin\edgesof\graph$, $w\tightneq u$,
then
$v \belowin\cograph v\jmeet w \abovein\cograph w\umeet u \abovein\cograph w,u$.
\end{lemma}
\begin{proof}
By Lemma\,\ref{lem:meet} $v\meet w$ is a $\graphjoin$ node and $w\meet u$ is a $\graphunion$ node.
Since $v\jmeet w\abovein\cograph w\belowin\cograph w\umeet u$ and $\cotreeof\graph$ is a tree, either
$v\jmeet w\belowin\cograph w\umeet u$ or
$v\jmeet w\abovein\cograph w\umeet u$.
If $v\jmeet w\belowin\cograph w\umeet u$ then $v\meet u$ is a $\graphunion$ node, contradicting $vu\tightin\edgesof\fograph$ (by Lemma\,\ref{lem:meet}), so $v\jmeet w\abovein\cograph w\umeet u$.
\end{proof}
\begin{lemma}\label{lem:covee}
If $\graph$ is a cograph with $vw,vu\tightnotin\edgesof\graph$ and $wu\tightin\edgesof\graph$
then $v \belowin\cograph v\umeet w \abovein\cograph w\jmeet u \abovein\cograph w,u$.
\end{lemma}
\begin{proof}
Necessarily $w\tightneq u$ since $wu\tightin\edgesof\graph$, hence $v\tightneq w$ and $v\tightneq u$.
Thus we can apply Lemma\,\ref{lem:vee} to the complement of $\graph$.
\end{proof}
\begin{lemma}\label{lem:umeet-to-umeet}
  If $\map:\graph\to\grapha$ is a skew fibration between cographs and
  $v \, \childin\graph \, v \umeet w \, \abovein\graph \, w$
  for $v,w\tightin\verticesof\graph$ with $\map(v)\tightneq\map(w)$,
  then
  $\map(v) \, \belowin\grapha \, \map(v)\umeet\map(w) \, \abovein\grapha \, \map(w)$.
\end{lemma}
\begin{proof}
Since $\map(v)\tightneq\map(w)$ the meet $\map(v)\meet\map(w)$ in
$\cotreeof\grapha$ is a $\graphunion$ or $\graphjoin$ node.
If the former,
we have
\mbox{$\map(v) \, \belowin\grapha \, \map(v)\umeet\map(w) \, \abovein\grapha \, \map(w)$} as desired.
Otherwise
\mbox{$\map(v) \, \belowin\grapha \, \map(v)\jmeet\map(w) \, \abovein\grapha \, \map(w)$}.
By Lem.\,\ref{lem:meet}
\mbox{$vw\tightnotin\edgesof\graph$} since
\mbox{$v \, \childin\graph \, v\umeet w\, \abovein\graph \, w$},
and
\mbox{$\map(v)\mkern1mu\map(w)\tightin\edgesof\grapha$} since
\mbox{$\map(v) \, \belowin\grapha \, \map(v)\jmeet\map(w) \, \abovein\grapha \, \map(w)$}.
Because $\map$ is a skew fibration and \mbox{$\map(v)\mkern1mu\map(w)\tightin\edgesof\grapha$},
there exists $u\tightin\verticesof\graph$ with
$vu\tightin\edgesof\graph$ and \mbox{$\map(w)\mkern1mu\map(u)\tightnotin\edgesof\grapha$}.
Since $\map$ is a graph homomorphism,
$\map(v)\mkern1mu\map(u)\tightin\edgesof\grapha$ and
$wu\tightnotin\edgesof\graph$,
and $w\tightneq u$ (otherwise $vw\tightin\edgesof\graph$ since $vu\tightin\edgesof\graph$, contradicting $vw\tightnotin\edgesof\graph$).
Since $wv,wu\tightnotin\edgesof\graph$ and $vu\tightin\edgesof\graph$,
by Lemma\,\ref{lem:covee} we have
$w \belowin\cograph w\umeet v \abovein\cograph v\jmeet u \abovein\cograph v$,
hence
$v\belowin\graph v\jmeet u\belowin\graph v\umeet w$, contradicting $v\childin\graph v\umeet w$.
\end{proof}
The following Lemma refines
$\map(v)\belowin\fographa \map(v)\umeet\map(w)$
in the above Lemma to
$\map(v)\childin\fographa \map(v)\umeet\map(w)$.
\begin{lemma}\label{lem:preserve-parent-plus}
  If $\map:\graph\to\grapha$ is a skew fibration between cographs and
  $v \, \childin\graph \, v \umeet w \, \abovein\graph \, w$
  for $v,w\tightin\verticesof\graph$ with $\map(v)\tightneq\map(w)$,
  then
  $\map(v) \, \childin\grapha \, \map(v)\umeet\map(w) \, \abovein\grapha \, \map(w)$.
\end{lemma}
\begin{proof}
By Lemma\,\ref{lem:umeet-to-umeet} we have
$\map(v) \, \belowin\grapha \, \map(v)\umeet\map(w) \, \abovein\grapha \, \map(w)$.
Suppose not $\map(v)\childin\fographa\map(v)\umeet\map(w)$. Then
\mbox{$\map(v) \, \belowin\grapha \, \map(v)\jmeet u \, \childin\grapha \, \map(v)\umeet\map(w) \, \abovein\grapha \, \map(w)$}
for some $u\tightin\verticesof\grapha$.
Since $\map(v) \, \belowin\grapha \, \map(v)\jmeet u \, \abovein\grapha u$
we have $\map(v)\mkern1mu u\tightin\edgesof\grapha$ by \reflem{lem:meet}.
Because $\map$ is a skew fibration and $\map(v)\mkern1mu u\tightin\edgesof\grapha$, there exists $\tilde u\tightin\verticesof\graph$ with $v\tilde u\tightin\edgesof\graph$ and $\map(u)\mkern-1mu\map(\tilde u)\tightnotin\edgesof\grapha$.
Necessarily $\tilde u w\tightin\edgesof\graph$, otherwise \reflem{lem:covee} applied to
$vw,\tilde u w\tightnotin\edgesof\graph$
and
$v\tilde u\tightin\edgesof\graph$
yields
$w \belowin\cograph w\umeet v \abovein\cograph v\jmeet \tilde u \abovein\cograph v$
so
$v \, \belowin\graph \, v\jmeet\tilde u \, \belowin\graph \, w\umeet v$
contradicting $v\,\childin\graph\,v\umeet w$.
\end{proof}
\begin{lemma}\label{lem:balance}
If $\bifib:\fograph\to\fographa$ is a skew fibration of cographs and $m\,\childin\fographa\, m\jmeet n\,\parentin\fographa\, n$, then
$\map(v)\atorbelowin\fographa m$ for some $v\tightin\verticesof\fograph$
if and only if
$\map(w)\atorbelowin\fographa n$ for some $w\tightin\verticesof\fograph$.
\end{lemma}
\begin{proof}
Assume $m\tightneq n$, otherwise the result is immediate.
Suppose $\map(v)\atorbelowin\fographa m$ for $v\tightin\verticesof\fograph$.
Choose $u\tightin\verticesof\fographa$ with $u\atorbelowin\fographa n$.
Thus $\map(v)\,\atorbelowin\fographa\, m \, \childin\fographa \, m\jmeet n \, \parentin\fographa\, n \,\atorabovein\fographa\, u\,$.
Since $m\jmeet n=\map(v)\jmeet u$ is a $\graphjoin$ node,
we have $\map(v)\mkern1mu u\tightin\edgesof\fographa$ by Lemma\,\ref{lem:meet}.
Because $\bifib$ is a skew fibration and $\map(v)\mkern1mu u\tightin\edgesof\fographa$,
there exists
$\widehat u\tightin\verticesof\fograph$
with
$v\widehat u\tightin\edgesof\fograph$
and
$\map(\widehat u)\mkern1mu u\tightnotin\edgesof\fographa$.
Since $v\widehat u\tightin\edgesof\fograph$ and $\bifib$ is a graph homomorphism, we have
$\map(v)\mkern1mu\map(\widehat u)\tightin\edgesof\fograph$.
If $u\tighteq\map(\widehat u)$ then since
$n \,\atorabovein\fographa\, u\,$
we have $\map(\widehat u)\atorbelowin\fographa n$ as desired.
Otherwise $u\tightneq\map(\widehat u)$, so applying Lemma~\ref{lem:vee} to
$\map(v)\mkern1mu u,\map(v)\mkern1mu\map(\widehat u)\tightin\edgesof\fographa$,
$u\mkern1mu \map(\widehat u)\tightnotin\edgesof\fographa$,
$u\tightneq\map(\widehat u)$
yields
$\map(v) \, \belowin\fographa \, \map(v)\jmeet u \, \abovein\fographa \, u\umeet \map(\widehat u) \, \abovein\fographa \, u,\map(\widehat u)$.
Thus since $\map(v)\jmeet u=m\jmeet n$ is the parent of $n$,
both
\mbox{$\map(v) \jmeet u \, \parentin\fographa \, n \, \atorabovein\fographa \, u$}
and
\mbox{$\map(v) \jmeet u \, \abovein\fographa  \, w\umeet\map(\widehat u) \, \abovein\fographa \, u$},
so because $\cotreeof\fographa$ is a tree,
we have
$n \, \atorabovein\fographa \, w\mkern1mu\umeet\mkern1mu\map(\widehat u)$,
hence
$\map(\widehat u) \, \atorbelowin\fographa \, n$.
\end{proof}
Given a function $\map:V\to W$ write $\map(V)$ for $\setst{\map(v)}{v\tightin V}\subseteq W$.
\begin{definition}
Let $\map:\fograph\to\fographa$ be a graph homomorphism between cographs. Define the \defn{image}
$\im\map$ as the subgraph $\induced\fographa{\map(\verticesof\fograph)}$ of $\fographa$ induced by $\map(\verticesof\fograph)$.
\end{definition}
Define $v\parentabovein\tree w$ if $v$ and $w$ are leaves and $\parentof v\,\abovein\tree\, w$ for the parent $\parentof v$ of $v$.
Recall that a cograph is \emph{logical} if every vertex is a binder or literal (\ie, is labelled by a variable or atom), and at least one vertex is a literal.
Write $\scopeofin \binder \fograph$ for the scope of a binder $\binder$ in a logical cograph $\fograph$.
The following Lemma shows that the scope of $\binder$ is the set of leaves below the parent of $\binder$ in the cotree $\cotreeof\cograph$.
\begin{lemma}\label{lem:parent-scope}\label{lem:scope-parent}
  For any vertex $\vertex$ and binder $\binder$ in a logical cograph $\fograph$,
$\vertex\tightin\scopeofin\binder\fograph$
if and only if
$\binder\parentabovein{\cotreeof\fograph}\vertex$.
\end{lemma}
\begin{proof}
By definition the scope of $\binder$ is the smallest proper strong module containing $\binder$,
which corresponds to the parent of $\binder$ in the cotree $\cotreeof\fograph$ by \reflem{lem:node-strong-module}.
\end{proof}
\begin{lemma}\label{lem:absorption-parent-below}
  For any \plustimestree $\tree$ and non-root \plustimesnode $\node$ in $\tree$,
  if $v\parentabovein\tree w$ then $v\parentabovein{\absorptionsub\tree\node} w$.
\end{lemma}
\begin{proof}
  Immediate from the definition of $\absorption\tree\node$.
\end{proof}
\begin{lemma}\label{lem:absorbed-parent-below}
  For any \plustimestree $\tree$, if $v\parentabovein\tree w$ then $v\parentabovein{\absorbed\tree} w$.
\end{lemma}
\begin{proof}
  Iterate \reflem{lem:absorption-parent-below} for every absorption step in the construction of $\absorbed\tree$ in Def.\,\ref{def:cotree-of-tree}.
\end{proof}
\begin{lemma}\label{lem:induced-parent-below}
  For any \plustimestree $\tree$ and non-empty set $U$ of leaves in $\tree$,
  if $v\parentabovein\tree w$ and $v,w\tightin U$ then $v\parentabovein{\induced\tree U} w$.
\end{lemma}
\begin{proof}
  The edge relation $\,\childin{\induced\tree U}\,$ is a subset of $\,\childin\tree\,$.
\end{proof}
\begin{lemma}\label{lem:induced-scope}
Let $\fographa$ be a logical cograph which is an induced subgraph of a logical cograph $\fograph$.
  For every vertex $\vertex$ and binder $\binder$ in $\fographa$,
  if $\vertex\in\scopeofin \binder \fograph$ then $\vertex\in\scopeofin \binder \fographa$.
\end{lemma}
\begin{proof}
\hspace{-2pt}Let $\vertex\in\scopeofin \binder \fograph$.
By \reflem{lem:scope-parent}, $\binder\parentabovein{\cotreeof\fograph}\vertex$.
By \reflem{lem:induced-parent-below}, $\binder\parentabovein{\induced{\cotreeof\fograph}{\verticesof\fographa}}\vertex$.
By \reflem{lem:absorbed-parent-below},
$\binder\parentabovein{\absorbed{\induced{\cotreeof\fograph}{\verticesof\fographa}}}\vertex$.
By \reflem{lem:induced-cotree},
$\absorbed{\induced{\cotreeof\fograph}{\verticesof\fographa}}
=
\cotreeof{\induced\fograph{\verticesof\fographa}}$,
and $\cotreeof{\induced\fograph{\verticesof\fographa}}=
\cotreeof{\fographa}$,
thus
$\binder\parentabovein{\cotreeof\fographa}\vertex$.
Therefore $\vertex\in\scopeofin \binder \fographa$ by \reflem{lem:scope-parent}.
\end{proof}
A fograph map $\map:\fograph\to\fographa$ \defn{preserves universals} if every universal binder $\binder$ in $\fograph$ maps to a universal binder $\map(\binder)$ in $\fographa$.
\begin{lemma}\label{lem:pres-universals}
  Every skew fibration between fographs preserves universals.
\end{lemma}
\begin{proof}
  By  \reflem{lem:node-strong-module}, a binder is universal if and only if its parent in the cotree is a $\graphunion$ node. Thus the result follows from \reflem{lem:preserve-parent-plus}.
\end{proof}
A fograph map \defn{preserves binders} if every universal (resp.\ existential) binder $\binder$ in $\fograph$ maps to a universal (resp.\ existential) binder $\map(\binder)$ in $\fographa$.
\begin{lemma}\label{lem:pres-binders}
  Every skew bifibration between fographs preserves binders.
\end{lemma}
\begin{proof}
  Skew bifibrations preserve existentials by definition, and universals by \reflem{lem:pres-universals}.
\end{proof}

\begin{lemma}\label{lem:im-logical}
For every skew bifibration $\bifib:\fograph\to\fographa$ of fographs, the image $\im\bifib$ is a logical cograph.
\end{lemma}
\begin{proof}
Every vertex of $\im\bifib$ is inherited from $\fographa$, and is therefore a binder or literal.
Since $\im\bifib$ is an induced subgraph of a cograph $\fographa$, it is a cograph.
Because $\fograph$ is a logical cograph, it contains a literal $\literal$, thus $\im\bifib$ contains the literal $\bifib(\literal)$ (a literal since $\bifib$ preserves labels).
\end{proof}
\begin{lemma}\label{lem:target-universal-scope}
For every skew bifibration $\bifib:\fograph\to\fographa$ between fographs and universal binder $\binder$ in $\im\bifib$,
the scope $\scopeofin \binder \fographa$ contains a literal in $\im\bifib$.
\end{lemma}
\begin{proof}
Choose $\tilde\binder$ in $\fograph$ with $\bifib(\tilde\binder)=\binder$.
By \reflem{lem:pres-binders}, $\tilde\binder$ is universal.
Since $\fograph$ is a fograph there exists a literal $\literal\in\scopeofin {\tilde\binder} \fograph$,
so
$\tilde\binder\,\childin\fograph\,\tilde b\umeet\literal\,\abovein\fograph\,\literal$ by \reflem{lem:parent-scope}.
Since $\bifib(\tilde\binder)\tightneq\bifib(\literal)$ by label preservation,
by \reflem{lem:preserve-parent-plus}
we have
$\binder\,\childin\fographa\,b\umeet\bifib(\literal)\,\abovein\fographa\,\bifib(\literal)$,
so $\bifib(\literal)\in\scopeofin \binder \fographa$ by \reflem{lem:parent-scope}.
\end{proof}

\begin{lemma}\label{lem:image-universal-scope}
For every skew bifibration $\bifib:\fograph\to\fographa$ between fographs and universal binder $\binder$ in $\im\bifib$,
the scope $\scopeofin \binder {\im\bifib}$ contains a literal.
\end{lemma}
\begin{proof}
By \reflem{lem:target-universal-scope}, the scope $\scopeofin \binder \fographa$ contains a literal $\bifib(\literal)$.
Since $\im\bifib$ is an induced subgaph of $\fographa$, by \reflem{lem:induced-scope} we have $\bifib(\literal)\in\scopeofin \binder {\im\bifib}$.
\end{proof}
\begin{lemma}\label{lem:image-existential-scope-prep}\label{lem:target-existential-scope}
For every skew bifibration $\bifib:\fograph\to\fographa$ between fographs and existential binder $\binder$ in $\im\bifib$,
the scope $\scopeofin \binder\fographa$ contains a literal in $\im\bifib$.
\end{lemma}
\begin{proof}
Since $\fographa$ is a fograph there exists a literal $\literala$ in $\fographa$ in the scope of $\binder$.
Thus
$\binder\,\childin\fographa\,b\jmeet\literala\,\abovein\fographa\,\literala$ by \reflem{lem:parent-scope}.
Therefore
$\binder\,\childin\fographa\, \binder\jmeet\literala\,\parentin\fographa\,n\,\atorabovein\fographa\,\literala$ for some child $n$ of $\binder\jmeet\literala$.
Since $\binder$ is in $\im\map$ there exists $\tilde\binder\in\verticesof\fograph$
with
$\bifib(\tilde\binder) \tighteq \binder$, so
we may apply \reflem{lem:balance} with
$m\tighteq \binder$ and $v\tighteq\tilde\binder$ to obtain
$w\in\verticesof\fograph$ with $\bifib(w)\atorbelowin\fographa n$.
If $w$ is a literal, then the literal $\map(w)$ is in $\scopeofin\binder\fographa$, and the Lemma holds.

Otherwise $w$ is a binder, hence $\bifib(w)$ is a binder.
We proceed by induction on the number of vertices in the scope $\scopeofin\binder\fographa$.
Since $\bifib(w) \,\atorbelowin\fographa\, n \,\atorabovein\fographa\, \literala$ for $\bifib(w)$ a binder and $\literala$ a literal,
$n$ must be a $\graphunion$ or $\graphjoin$ node,
and since $\graphunion$ and $\graphjoin$ alternate in a cotree,
$n$ is a $\graphunion$ node because its parent $\binder\jmeet\literala$ is a $\graphjoin$ node.
Let $n\primed$ be the parent of $\bifib(w)$.
Thus
$\binder \,\childin\fographa\, \binder\jmeet\literala \,\parentin\fographa\, n \,\atorabovein\fographa\, n\primed \,\parentin\fographa\, \bifib(w)$.
If $n\primed$ is a $\graphunion$ node, then $\bifib(w)$ is universal so by \reflem{lem:image-universal-scope} the scope of $\bifib(w)$ contains a literal in $\im\bifib$, \ie, a literal $\bifib(\literal)$ for some literal $\literal$ in $\fograph$.
Therefore
$\binder \,\childin\fographa\, \binder\jmeet\literala \,\parentin\fographa\, n \,\atorabovein\fographa\, n\primed \,\parentin\fographa\, \bifib(\literal)$, so $\bifib(\literal)$ is also in $\scopeofin\binder\fographa$.
Otherwise $n\primed$ is a $\graphjoin$ node, so $\bifib(w)$ is existential. Since $\scopeofin{\bifib(w)}\fographa$ is strictly contained in $\scopeofin\binder\fographa$, by induction there exists a literal $\literal$ in $\fograph$ such that $\bifib(\literal)$ is in $\scopeofin{\bifib(w)}\fographa$, thus the literal $\bifib(\literal)$ is in $\scopeofin\binder\fographa$.
\end{proof}
\begin{lemma}\label{lem:target-scope}
For every skew bifibration $\bifib:\fograph\to\fographa$ between fographs and binder $\binder$ in $\im\bifib$,
the scope $\scopeofin \binder \fographa$ contains a literal in $\im\bifib$.
\end{lemma}
\begin{proof}
  If $\binder$ is universal (resp.\ existential) apply \reflem{lem:target-universal-scope} (resp.\ \ref{lem:target-existential-scope}).
\end{proof}
\begin{lemma}\label{lem:image-existential-scope}
For every skew bifibration $\bifib:\fograph\to\fographa$ between fographs and existential binder $\binder$ in $\im\bifib$,
the scope $\scopeofin \binder {\im\bifib}$ contains a literal.
\end{lemma}
\begin{proof}
  By \reflem{lem:image-existential-scope-prep} there exists a literal $\literal$ with $\bifib(\literal)$ in $\scopeofin\binder\fographa$. Since $\im\bifib$ is an induced subgraph of $\fographa$, we have $\bifib(\literal)$ in $\scopeofin\binder{\im\bifib}$ by \reflem{lem:induced-scope}.
\end{proof}

\begin{definition}\label{def:fair}
  A logical cograph $\fograph$ is \defn{fair} if binders $\binder$ and $\binderp$ have the same variable only if $\binder\binderp\tightnotin\edgesof\fograph$.
\end{definition}
Note that every rectified fograph is fair.
\begin{lemma}\label{lem:image-fograph}
  For every skew bifibration $\bifib:\fograph\to\fographa$ between fographs with $\fographa$ fair, $\im\bifib$ is a fair fograph.\todo{\ldots is a fair fograph needed anywhere?}
\end{lemma}
\begin{proof}
By \reflem{lem:im-logical} $\im\bifib$ is a logical cograph, and $\im\bifib$ is fair since if $\binder\binderp\tightin\edgesof{\im\bifib}$ for binders $\binder$ and $\binderp$ with the same variable, then $\binder\binderp\tightin\edgesof\fographa$ since $\im\bifib$ is an induced subgraph, contradicting the fairness of $\fographa$. It remains to show that
(1) for every binder $\binder$ in $\im\bifib$ the scope $\scopeofin\binder{\im\bifib}$ contains a literal, and (2) for every variable $x$ and every $x$-binder $\binder$ in $\im\bifib$, the scope $\scopeofin\binder{\im\bifib}$ contains no other $x$-binder.

(1) If $\binder$ is universal (resp.\ existential), then by \reflem{lem:image-universal-scope} (resp.\ \ref{lem:image-existential-scope}), the scope
$\scopeofin\binder{\im\bifib}$ contains a literal.

(2)
Suppose $\binderp$ were another $x$-binder with $\binderp\tightin\scopeofin\binder{\im\bifib}$,
\ie, $\binder\childin{\im\bifib}\binder\meet\binderp\abovein{\im\bifib}\binderp$.
If $\binder$ is universal,
then
$\binder\meet\binderp\tighteq\binder\umeet\binderp$,
so
$\binder\childin{\fographa}\binder\umeet\binderp\abovein\fographa\binderp$,
whence $\binderp\tightin\scopeofin\binder\fographa$, contradicting the fact that $\fographa$ is a fograph.
Otherwise, $\binder$ is existential,
and $\binder\meet\binderp\tighteq\binder\jmeet\binderp$, so $\binder\binderp\tightin\edgesof{\im\bifib}$.
Since $\im\bifib$ is an induced subgraph of $\fographa$, we have $\binder\binderp\tightin\edgesof\fographa$, contradicting the fairness of $\fographa$.
\end{proof}
The following example illustrates why fairness of $\fographa$ is required to ensure no $x$-binder is in the scope of another in \reflem{lem:image-fograph}.
Let $\graph=\xgraphof{(\exists x\mkern2mu\px)\tightvee(\exists x\mkern2mu\px)}
=
\namedsingletonleft{x}{x}\hspace{1ex}\namedsingletonright{px}{px}\e{x}{px}
\hspace{1ex}
\namedsingletonleft{x}{x}\hspace{1ex}\namedsingletonright{px}{px}\e{x}{px}$,
let
$\graphaa=\xgraphof{q\mkern-.5mu\vee\exists x\mkern2mu\px}
=
\singletonq
\hspace{1ex}
\namedsingletonleft{x}{x}\hspace{1ex}\namedsingletonright{px}{px}\e{x}{px}$,
and let $\map$ be the unique label-preserving graph homomorphism $\graph\to\graphaa$, which is a skew bifibration between fographs.
Then
$\map\tightwedge\map:\graph\tightwedge\graph\to\grapha=\graphaa\tightwedge\graphaa$
is a skew bifibration between fographs, with $\grapha$ not fair, and $\im\map$ is
$
(\namedsingletonleft{x}{x}\hspace{1ex}\namedsingletonright{px}{px}\e{x}{px})
\tightwedge
(\namedsingletonleft{x}{x}\hspace{1ex}\namedsingletonright{px}{px}\e{x}{px})
$, which is not well-defined fograph since each $\singletonx$ is in the scope of the other.

\subsubsection{Marking and pruning}
Let $\cograph$ be a cograph and let $U\tightsubseteq\verticesof\cograph$.
A node $\node$ in the cotree $\cotreeof\cograph$ is \defn{over} $U$ if
$n\atorabovein\cograph u$
for some vertex $u\tightin U$.
Define
the \defn{support}
$\markingof U\subseteq\nodesof{\cotreeof\cograph}$
as the set of nodes over $U$,
and say that $U$ is \defn{balanced} for $\cograph$ if, for every $\graphjoin$ node $\node$ in $\cotreeof\cograph$ and child $\nodea$ of $\node$, we have $\nodea\tightin\markingof U$ if $\node\tightin\markingof U$.
\begin{lemma}\label{lem:im-balanced}
\hspace{-.4pt}If $\map:\graph\mkern-2mu\to\mkern-2mu\grapha$ is a skew fibration between cographs
then $\map(\verticesof\graph\mkern-2mu)\tightsubseteq\verticesof\grapha$ is balanced for $\grapha$.
\end{lemma}
\begin{proof}
A corollary of \reflem{lem:balance}.
\end{proof}
Let $\cograph$ be a cograph and let $U\tightsubseteq\verticesof\cograph$.
A \plustimesnode $\node$ in $\markingof U$ is \defn{literal-supported}
if there exists a literal $\literal\tightin U$ with $\node\atorabovein\cograph\literal$.
We say that $U$ is \defn{binding-closed} if, for every literal $\literal\tightin U$ and binder $\binder$ in $\cograph$
such that $\binder$ binds $\literal$, we have $\binder\tightin U$.
\begin{definition}\label{def:marking}
Let $\fograph$ be a fograph. A set $U\tightsubseteq\verticesof\fograph$ is a \defn{marking} for $\fograph$ if it is balanced, every \plustimesnode of $\markingof U$ is literal-supported, and $U$ is binding-closed.
\end{definition}
\begin{lemma}\label{lem:im-marking}
If $\bifib:\fograph\to\fographa$ is a skew bifibration
between fographs
then
$\bifib(\verticesof\fograph)$ is a marking for $\fographa$.
\end{lemma}
\begin{proof}
Let $U=\bifib(\verticesof\fograph)$.
By \reflem{lem:balance}, $U$ is balanced, by \reflem{lem:target-scope}, every node $\node$ in $\markingof U$ is literal-supported,\todo{more}
and $U$ is binding-closed since $\bifib:\bindinggraphof\fograph\to\bindinggraphof\fographa$ is a directed graph fibration.
\end{proof}
Let $\node$ be the child of a $\graphunion$ node $\nodea$ in a \plustimestree $\tree$.
The node $\node$ is \defn{critical} to $\nodea$ if $\node$ is the only child of $\nodea$ which is at or above a literal.
If $\node$ is an $x$-binder for some variable $x$, then $\node$ is \defn{vacuous} if it is the unique node in the subtree rooted at $\nodea$ whose label contains $x$.\footnote{Thus in the cograph $\cographof\tree$, the binder $\node$ (which is universal since its parent is a $\graphunion$ node) binds no literal.}

\begin{definition}\label{def:pareable}\label{def:pare}
A node $\node$ in a \plustimestree $\tree$ is \defn{pareable} if:
\begin{enumerate}
\item $\node$ has a parent $\graphunion$ node $\nodea$,
\item $\node$ is not critical to $\nodea$, and
\item if $\node$ is a binder (necessarily universal) then it is vacuous.
\end{enumerate}
To \defn{pare} a pareable node $\node$ in a \plustimestree $\tree$ is to delete the subtree rooted at $\node$.
\end{definition}
\begin{definition}
A \defn{pruning} is any result of iteratively paring zero or more pareable $\graphunion$ nodes.
\end{definition}
\begin{lemma}\label{lem:marking-pruning}
  Let $\fograph$ be a fograph with marking $U$,
  and let $\tree$ be a \plustimestree such that $\cographof\tree=\fograph$.
There exists a pruning $\treep$ of $\tree$ with $\cographof\treep=\induced\fograph U$.
\end{lemma}
\begin{proof}
  A routine induction on the number of nodes in $\tree$.\todo{elaborate: see sheet 22 of notepad Apr'19}
\end{proof}

\subsubsection{Decomposition of skew bifibrations}\label{sec:decomp-bifib}

\begin{definition}
  If $\fograph$ is a connected fograph without the variable $x$, define \defn{slackening} $\slackeningof\fograph x$ as the canonical inclusion map $\fograph\to\forall x\mkern2mu\fograph$.
\end{definition}
\begin{lemma}\label{lem:slackening-structural}
  Every slackening is a structural map.
\end{lemma}
\begin{proof}
  Weaken $\fograph$ to $\fograph\vee\forall x\mkern2mu\fograph$, which is $\forall x(\fograph\vee\fograph)$, then contract under $\forall x$ to $\forall x\mkern2mu\fograph$. (Note that $\fograph\vee\fograph$ is well-defined because $\fograph$ is connected.)
\end{proof}
\begin{definition}
  A \defn{\wsmap} is any map constructed from isomorphisms, weakenings and slackenings by composition and fograph connectives.
\end{definition}
\begin{lemma}\label{lem:ws-structural}
  Every \wsmap is a structural map.
\end{lemma}
\begin{proof}
  Iterate \reflem{lem:slackening-structural}.
\end{proof}
\begin{lemma}\label{lem:paring-ws}
  Let $\fograph$ be a fograph,
  let $\tree$ be a \plustimestree such that $\cographof\tree=\fograph$,
  let $\treep$ be the result of paring a pareable node in $\tree$, and let $\fographp=\cographof\treep$.
There exists a \wsmap $\fographp\to\fograph$.
\end{lemma}
\begin{proof}
  If the paring is of a vacuous binder (condition 3 in Def.\,\ref{def:pareable}), then we obtain a slackening in the context of a fograph connective, otherwise (condition 2 in Def.\,\ref{def:pareable}) we obtain a weakening in the context of a fograph connective.
\end{proof}
\begin{lemma}\label{lem:pruning-ws}
  Let $\fograph$ be a fograph,
  let $\tree$ be a \plustimestree such that $\cographof\tree=\fograph$,
  let $\treep$ be a pruning of $\tree$, and let $\fographp=\cographof\treep$.
There exists a \wsmap $\fographp\to\fograph$.
\end{lemma}
\begin{proof}
  Apply \reflem{lem:paring-ws} to each paring in the pruning.
\end{proof}
\begin{lemma}\label{lem:image-inclusion-wsmap}
  Let $\bifib:\fograph\to\fographa$ be a skew bifibration with $\fographa$ fair. The inclusion $\im\bifib\to\fographa$ is a \wsmap.
\end{lemma}

\begin{proof}
By \reflem{lem:image-fograph}, $\im\bifib$ is a fograph.
  Let $U=\bifib(\verticesof\fograph)$, thus $\im\bifib$ is the induced subgraph $\induced\fographa U$.
  By \reflem{lem:im-marking}, $U$ is a marking.
  Let $\tree$ be the cotree $\cotreeof\fographa$.
  By \reflem{lem:im-marking}, there exists a pruning $\treep$ of $\tree$ with $\cographof\treep=\induced\fographa U=\im\bifib$.
  By \reflem{lem:pruning-ws}, there exists a \wsmap $\cographof\treep\to\cographof\tree$, \ie,\todo{G(T(G))=G lemma earlier?}
$\im\bifib\to\fographa$.
\end{proof}
\noindent Let $\bifib:\fograph\to\fographa$ be a skew fibration and let $K$ be a connected component of $\fographa$. The \defn{multiplicity} of $K$ is the number of connected components of $\bifib^{-1}(K)$, and the \defn{weight} of $K$ is one more than its multiplicity.
The \defn{weight} of $\bifib$ is the sum of the weights of the connected components of $\fographa$.
A skew bifibration is \defn{shallow} if the multiplicity of every connected component of $\fographa$ is at most one.
\begin{lemma}\label{lem:bifib-cw}
  Every skew bifibration into a fair fograph is a structural map.
\end{lemma}
\begin{proof}\todo{elaborate?}
  By induction on the weight of the skew bifibration $\bifib:\fograph\to\fographa$ and its multiplicity.
By \reflem{lem:image-inclusion-wsmap}
(and the fact that every \wsmap is a structural map by \reflem{lem:ws-structural})
we may assume $\bifib$ is a surjection, and by pre-composing with contractions
we may assume $\bifib$ is shallow.
  If $\fographa=\singletonx\graphunion\fographap$ then $\fograph=\singletonx\graphunion\fographp$ since $\bifib$ is a shallow surjection, hence $\bifib=\forall x\mkern2mu\bifibp$, and by induction $\bifibp$ is a structural map.
  Otherwise if $\fographa=\fographa_1\graphunion\fographa_2$ then $\fographa=\fographa_1\tightvee\fographa_2$ (since $\fographa$ is not of the form $\singletonx\graphunion\fographap$).
  Since $\bifib$ is a shallow surjection, $\fograph=\fograph_1\tightvee\fograph_2$, so $\bifib=\bifib_1\tightvee\bifib_2$ for $\bifib_i:\fograph_i\to\fographa_i$. By induction each $\bifib_i$ is a structural map, hence $\bifib$ is structural.
  Otherwise $\fographa$ is connected. If $\fographa$ has no edge then $\bifib$ is an isomorphism from a literal to a literal, hence is a structural map. Thus we may assume $\fographa$ has an edge.
If $\fographa=\singletonx\graphjoin\fographap$ then $\fograph=\singletonx\graphjoin\fographp$ since $\bifib$ is a shallow surjection, hence $\bifib=\exists x\mkern2mu\bifibp$, and by induction $\bifibp$ is a structural map.
Otherwise $\fographa=\fographa_1\graphjoin\fographa_2$ for fographs $\fographa_i$, with $\fographa_i$ not of the form $\singleton x\graphjoin\fographa_i'$. Thus $\fographa=\fograph_1\wedge\fograph_2$, hence $\fograph=\fograph_1\wedge\fograph_2$ with $\bifib(\verticesof{\fograph_i})\subseteq\verticesof{\fographa_i}$. Therefore $\bifib=\bifib_1\vee\bifib_2$ for skew bifibrations $\bifib_i:\fograph_i\to\fographa_i$, and by induction each $\bifib_i$ is a structural map, so $\bifib$ is a structural map.
\end{proof}
\begin{lemma}[Soundness of skew bifibrations]\label{lem:bifib-sound}
  If $\fograph$ is a valid fograph and $\bifib:\fograph\to\fographa$ is a skew bifibration with $\fographa$ fair, then $\fographa$ is valid.
\end{lemma}
\begin{proof}
  By Lemma\,\ref{lem:bifib-cw}, $\bifib$ is a structural map, which is sound by \reflem{lem:strucmap-sound}.
\end{proof}

\subsection{Proof of the Soundness Theorem}\label{subsec:proof-of-soundness}

\begin{proof}[Proof of the Soundness Theorem (Theorem~\ref{thm:soundness})]
Let $\bifib:\net\to\graphof\formula$ be a combinatorial proof of a formula $\formula$. By \reflem{lem:fonet-soundness} $\net$ is valid,
thus by \reflem{lem:bifib-sound} $\graphof\formula$ is valid (applicable since $\graphof\formula$ is rectified, hence fair), therefore $\formula$ is valid.
\end{proof}

\section{Proof of the Completeness Theorem}\label{sec:completeness}\label{sec:cp-completeness}

In this section we prove the Completeness Theorem, Theorem~\ref{thm:completeness}.
Our strategy will be to show that
every syntactic proof of a formula $\formula$ in Gentzen's classical sequent calculus \cite{Gen35}
generates a combinatorial proof of $\formula$, so completeness follows from that of Gentzen's system.

A \defn{sequent} is a finite sequence $\sequentsequence$ of formulas, $n\tightge 0$.
We identify a formula $\formula$ with the single-formula sequent containing $\formula$.
Let $\sequent$ be the sequent $\sequentsequence$.
Its \defn{formula} $\formulaof\sequent$ is $\formulaofsequent$, and $\sequent$ is
\defn{valid}
(resp.\ \defn{rectified})
if $\formulaof\sequent$ is
valid
(resp.\ rectified).
As with formulas, we shall assume, without loss of generality, that every sequent is rectified (by renaming bound variables as needed).

The \defn{graph} $\graphof\sequent$ of a sequent $\sequent$ is the graph
$\graphof{\formulaof\sequent}$
of its formula.
For example,
$\graphof{\px\com\exists y\mkern2mu\ppy}\mkern2mu=\mkern4mu\singletonleft\px\hspace{1ex}\namedsingletonleft y y\hspace{1ex}\namedsingletonright{ppy}\ppy\e y{ppy}\mkern2mu$.
A \defn{combinatorial proof} of $\sequent$ is a combinatorial proof of its formula $\formulaof\sequent$.

For technical convenience we will use a right-sided formulation $\R$ of Gentzen's sequent calculus $\LK$ \cite{Gen35},
comprising the following rules.
Here
$\sequent$ and $\sequenta$
are arbitrary sequents,
$\formula$ and $\formulaa$ are arbitrary formulas,
$\atom$ is any atom which is not a logical constant,
and $\formulap\cong\formula$ denotes that $\formulap$ is
equal to $\formula$ up to bound variable renaming (\eg,
$
\forall x\mkern2mu pxy
\cong
\forall z\mkern2mu pzy
$).\footnote{For the $\exists$ rule, recall
(from \S\ref{sec:soundness})
that $\formula\assignment{\assign\variable\term}$ denotes
the result of substituting a term $\term$ for all occurrences of the variable $\variable$ in $\formula$, where,
without loss of generality (by renaming bound variables in $\formula$ as needed),
no variable in $\term$ is a bound variable of $\formula$.}

\begin{center}\v{2.5}\(
\newcommand\gap{\hspace{8ex}}\hh{14}\begin{array}{c}
\Raxiomrule{\atom}
\gap
\Ronerule
\gap
\Rxrule\Rxrulehyp\Rxruleconc
\gap
\Rwrule\Rwrulehyp\Rwruleconc
\gap
\Rcrule\Rcrulehyp\Rcruleconc
\rput[lb](.3,.3){\text{($\formulap\cong\formula$)}}
\\[5ex]
\renewcommand\gap{\hspace{8ex}}
\Randrule\Randrulehypone\Randrulehyptwo\Randruleconc
\gap
\Rorrule\Rorrulehyp\Rorruleconc
\gap
\Rexistsrule\Rexistsrulehyp\Rexistsruleconc
\gap
\Rforallrule\Rforallrulehyp\Rforallruleconc
\rput[lb](.3,.3){\text{($x$ not free in $\Gamma$)}}
\end{array}\)\v{2.5}\end{center}
The rules $\xname$,  $\cname$ and $\wname$ are called \defn{exchange}, \defn{contraction} and \defn{weakening}.
Each sequent above a rule is a \defn{hypothesis} of the rule, and the sequent below a rule is the \defn{conclusion} of the rule.
\begin{lemma}[\R soundness \& completeness]\label{lem:r-sound-complete}
  A sequent is valid if and only if it has an \R proof.
\end{lemma}
\begin{proof}
  System \R is equivalent to \gsone, which is sound and complete \cite[\S3.5.2]{TS96}.
  It differs in that \R retains Gentzen's explicit formula-exchange (or \emph{permutation}) rule $\xname$,
  while \cite{TS96} leaves exchange implicit by formulating sequents as multisets rather than sequences.
\end{proof}

\subsection{Interpreting rules as operations on combinatorial proofs}\label{sec:interp-rules}

We interpret each rule of \R with hypothesis sequents $\sequent_1,\ldots,\sequent_n$ and conclusion sequent $\sequenta$ as an operation taking combinatorial proofs $\skewfib_i$ of $\sequent_i$ as input to produce a combinatorial proof $\bifiba$ of $\sequenta$.
\begin{itemize}
\item $\inlineaxiom\atom$ rule. Define $\bifiba$ as the identity on $\singletonleft\atom\hspace{.7ex}\singletonright\dualatom$, with both vertices the same colour in the source.
\item $\dual1$ rule. Define $\bifiba$ as the identity on $\singleton1$, with no colour in the source.
\item $\vee$ rule
with hypothesis $\hyp=\Rorrulehyp$ and conclusion $\conc=\Rorruleconc$.
Let $\bifib$ be the combinatorial proof of $\hyp$.
Note that $\graphof\hyp=\graphof\conc$.
Define $\bifiba=\bifib$.
\item $\xname$ rule
with hypothesis $\hyp=\Rxrulehyp$ and conclusion $\conc=\Rxruleconc$.
Let $\bifib$ be the combinatorial proof of $\hyp$, and let $\textsf{x}$ be the canonical isomorphism $\graphof\hyp\to\graphof\conc$.
Define $\bifiba=\textsf{x}\circ\bifib$.
\item $\wname$ rule with hypothesis $\hyp=\Rwrulehyp$ and conclusion $\conc=\Rwruleconc$.
Let $\bifib$ be the combinatorial proof of $\hyp$, and let $\textsf{w}$ be the canonical injection $\graphof\hyp\to\graphof\conc$.
Define $\bifiba=\textsf{w}\circ\bifib$.
\item $\forall$ rule
with hypothesis $\hyp=\Rforallrulehyp$ and conclusion $\conc=\Rforallruleconc$.
Let $\bifib:\cover\to\graphof{\hyp}$ be the combinatorial proof of $\hyp$ and
let $\textsf{a}$ be the canonical injective graph homomorphism \mbox{$\graphof\hyp\to\graphof\conc$}.
Note $\graphof\hyp=\graphof{\sequent}\graphunion\graphof{\formula}$.
If $\bifib^{-1}(\verticesof{\graphof\formula})$ is empty define $\bifiba=\textsf{a}\circ\bifib:\cover\to\graphof{\conc}$.
Otherwise,
define \mbox{$\bifiba:\singleton x\mkern-1mu\graphunion\mkern-2mu \cover\to\graphof{\conc}$} as
the extension of $\textsf{a}\circ\bifib:\cover\to\graphof{\conc}$
which maps
$\singleton x$ to the (unique) $x$-binder in $\graphof{\conc}$.
\item $\cname$ rule
with hypothesis $\hyp=\Rcrulehyp$ and conclusion $\conc=\Rcruleconc$.
Let $\bifib:\cover\to\graphof{\hyp}$ be the combinatorial proof of $\hyp$.
Let $\skel\cover$ be the result of dropping the labels from $\cover$ and
let \mbox{$\skel\bifib:\skel\cover\to\graphof{\hyp}$} be the skeleton of $\bifib$.
Let $\textsf{c}$ be the canonical surjective graph homomorphism \mbox{$\graphof\hyp\to\graphof\conc$}.
Define $h=\textsf{c}\circ\skel\bifib:\skel\cover\to\graphof\conc$.
A universal binder in $\graphof\conc$ is \defn{outer} if it is in no edge.
A \defn{duplicator} is an outer universal binder $\binder$ in $\graphof\conc$ such that
$h^{-1}(\binder)$ contains two vertices.
To \defn{collapse} a duplicator $\binder$ is to
delete one of the two vertices of $h^{-1}(\binder)$ from $\skel\cover$.
Define $\skel\cover^-$ as the derivative of $\cover$ obtained by collapsing every duplicator.
Define the skeleton $\skel\bifiba:\skel\cover^-\to\graphof{\conc}$ of $\bifiba$ as
the restriction of $h$ to $\skel\cover^-$.
Thus $\bifiba:\cover^-\to\graphof{\conc}$ for $\cover^-$ the result of adding the canonical labels to $\skel\cover^-$
(\ie, the label of $\vertex$ in $\cover^-$ is the label of $\skel\bifiba(\vertex)$ in $\graphof\conc$).\footnote{For example,
let $\hyp=\forall x\mkern1mu 1\com\mkern3mu\forall y\mkern1mu 1$,
let $\conc=\forall x\mkern1mu 1$,
let $\cover=
\graphof{\hyp}=
\singleton x\hspace{.7ex}\singleton 1
\hspace{.8ex}
\singleton y \hspace{.7ex}\singleton 1$,
and let $\bifib$ be the identity $\cover\to\cover$.
\newcommand\coverminus{\singleton x\hspace{.7ex}\singleton 1
\hspace{.8ex}
\singleton 1}%
\newcommand\badcover{\singleton x\hspace{.7ex}\singleton 1
\hspace{.8ex}
\singleton x \hspace{.7ex}
\singleton 1}%
Then $\graphof{\conc}=\singleton x\hspace{.7ex}\singleton 1$
and $\singleton x$ in $\graphof{\conc}$ is a duplicator.
Thus
$\cover^-=\coverminus$ and the combinatorial proof
$\bifiba$ of $\conc$ is the unique label-preserving function from
$\coverminus$ to $\graphof{\conc}=\singleton x\hspace{.7ex}\singleton1$.
We must delete $\singleton y$ from $\cover$ to form $\cover^-$ because otherwise $\bifiba$ will be
from
$\badcover$
to
$\singleton x\hspace{.7ex}\singleton1$, whose source $\badcover$ is not a well-defined fograph since it has two universal $x$-binders in the scope of one another.}
\item \begin{samepage}$\wedge$ rule
with hypotheses $\hyp_1=\Randrulehypindex1$ and $\hyp_2=\Randrulehypindex2$ and conclusion $\conc=\Randruleconcindexed$.
Let $\bifib_i:\cover_i\to\graphof{\hyp_i}$ be the combinatorial proof of $\hyp_i$ and let $\textsf{z}_i$ be the canonical injective graph homomorphism $\graphof{\hyp_i}\to\graphof\conc$.
Note that $\graphof{\hyp_i}=\graphof{\sequent_i}\graphunion\graphof{\formula_i}$.
Let $\portion_i=\bifib_i^{\mkern3mu-1}(\verticesof{\graphof{\formula_i}})$
and define $\bifib_i$ as \defn{weak} if $\portion_i$ is empty, else \defn{strong}.
Define $\bifiba$ according to the following cases.
\vspace{-.6ex}\begin{itemize}
\item $\bifib_1$ and $\bifib_2$ are both strong or both weak.
Let $\cover$ be the fusion of $\cover_1$ and $\cover_2$ at the portions $P_1$ and $P_2$.
Define $\bifiba:\cover\to\graphof{\conc}$ as the union of $\textsf{z}_1\circ\bifib_1$ and $\textsf{z}_2\circ\bifib_2$.
\item $\bifib_i$ is weak and $\bifib_{3-i}$ is strong.
Define $\bifiba:\cover_i\to\graphof{\conc}$ as $\textsf{z}_i\circ\bifib_i$.
\end{itemize}\end{samepage}
\item $\exists$ rule
with hypothesis $\hyp=\Rexistsrulehyp$ and conclusion $\conc=\Rexistsruleconc$.
Let $\bifib:\cover\to\graphof{\hyp}$ be the combinatorial proof of $\hyp$.
Let $\skel\cover$ be the result of dropping the labels from $\cover$ and let $\skel\bifib:\skel\cover\to\graphof{\hyp}$ be the skeleton of $\bifib$.
Let $\textsf{e}$ be the canonical injective graph homomorphism \mbox{$\graphof\hyp\to\graphof\conc$}.
Note $\graphof{\hyp}=\graphof{\sequent}\graphunion\graphof{\formula\assignment{\assign\variable\term}}$.
Let $P=\skel\bifib^{-1}(\verticesof{\graphof{\formula\assignment{\assign\variable\term}}})$.
Define the skeleton $\skel\bifiba$ of $\bifiba$ according to the following cases.
\vspace{-.6ex}\begin{itemize}
\item $P$ is empty. Define $\skel\bifiba$ as $\textsf{e}\circ\bifib:\skel\cover\to\graphof{\conc}$.
\item $P$ is non-empty. Let $\skel\cover^+$ be the
extension of $\skel\cover$ with an additional vertex $v$ and edges from $v$ to every vertex in $P$.
Define $\skel\bifiba:\skel\cover^+\to\graphof{\conc}$ as
the extension of $\textsf{e}\circ\bifib:\skel\cover\to\graphof{\conc}$ to $\skel\cover^+$ which maps $\vertex$ to the $x$-binder of $\graphof{\conc}$.
\end{itemize}
\end{itemize}
\begin{lemma}\label{lem:interp-well-defined}
The interpretation of each rule of $\R$ defined above
produces a well-defined combinatorial proof.
\end{lemma}
\begin{proof}
A routine verification of the fograph and skew bifibration conditions defining a combinatorial proof.
\end{proof}

\subsection{Proof of the Completeness Theorem}\label{subsec:cp-completeness}

\begin{proof}[Proof of the Completeness Theorem, Theorem~\ref{thm:completeness}]
Let $\formula$ be a valid formula.
By \reflem{lem:r-sound-complete} there exists an \R proof $\Pi$ of $\formula$.
By \reflem{lem:interp-well-defined} we obtain a combinatorial proof of $\formula$ from $\Pi$ by interpreting each rule of $\Pi$ as an operation on combinatorial proofs, as defined in \S\ref{sec:interp-rules}.
\end{proof}

\section{Homogeneous soundness and completeness proofs}

\subsection{Propositional homogeneous soundness and completeness proof}\label{sec:proof-of-propositional-homogeneous-soundness-completeness}

In this section we prove the propositional homogeneous soundness and completeness theorem, Theorem\,\ref{thm:prop-soundness-completeness}.
We begin by observing that the function $\dualizinggraphofpropsymbol$ from simple propositions to dualizing graphs (Def.\,\ref{def:dualizing-graph-of-prop}) factorizes through simple propositional fographs.
A fograph is \defn{propositional} if every predicate symbol is nullary, and \defn{simple} if it has no $1$- or $0$-labelled literal.
For example, the middle row of Fig.\,\ref{fig:dualizing-graphs} (p.\,\pageref{fig:dualizing-graphs}) shows four simple propositional fographs.
\begin{definition}\label{def:prop-fograph-to-dualizing-graph}
  The \defn{dualizing graph} $\dualizinggraphof\fograph$ of a simple propositional fograph is the dualizing graph $\dualizinggraph$ with $\verticesof\dualizinggraph=\verticesof\fograph$, $\edgesof\dualizinggraph=\edgesof\fograph$, and $\vertex\vertexa\in\dualitiesof\dualizinggraph$ if and only if
$\vertex$ and $\vertexa$ have dual predicate symbols.\footnote{In \S\ref{sec:surjections} we will show that $\dualizinggraphofsymbol$ is a surjection from simple propositional fographs onto dualizing graphs (\reflem{lem:surj-prop-fographs-to-dualizing-graphs}).}
\end{definition}
For example, for each simple propositional fograph $\fograph$ in the middle row of Fig.\,\ref{fig:dualizing-graphs} (p.\,\pageref{fig:dualizing-graphs}),
the corresponding dualizing graph $\dualizinggraphof\fograph$ is shown below $\fograph$.

\begin{lemma}\label{lem:dualizing-graph-of-prop-fograph-well-defined}
$\dualizinggraphof\fograph$ is a well-defined dualizing graph for every simple propositional fograph $\fograph$.
\end{lemma}
\begin{proof}
Let
$\dualizinggraph=\dualizinggraphof\fograph$.
Since
$\verticesof\dualizinggraph=\verticesof\fograph$,
$\edgesof\dualizinggraph=\edgesof\fograph$
and
$\fograph$ is a fograph,
$\dualizinggraph$ is a cograph.
By reasoning as in the proof of \reflem{lem:dualizing-graph-of-prop-well-defined}, $\graphpairof{\verticesof\fograph}{\dualitiesof\dualizinggraph}$ is $P_4$- and $\cthree$-free.
\end{proof}
\begin{lemma}\label{lem:prop-factorization}
The function $\dualizinggraphofpropsymbol$ from simple propositions to dualizing graphs (Def.\,\ref{def:dualizing-graph-of-prop}) factorizes through simple propositional fographs:
$\dualizinggraphofprop\prop=\dualizinggraphof{\graphof\prop}$ for every simple proposition $\prop$.
\end{lemma}
\begin{proof}
A routine induction on the structure of $\prop$.\todo{todo}
\end{proof}
\begin{lemma}[Propositional homogeneous soundness]\label{lem:prop-soundness}
A simple proposition is valid if it has a homogeneous combinatorial proof.
\end{lemma}
\begin{proof}
Suppose
$\bifib:\net\to\dualizinggraphofprop\prop=\dualizinggraph$
is a homogeneous combinatorial proof of the simple proposition $\prop$.
By \reflem{lem:prop-factorization}, $\dualizinggraph=\dualizinggraphof{\graphof\prop}$.
Define $\netp$ as the cograph $\graphpairof{\verticesof\net}{\edgesof\net}$
with a link $\setof{\vertex,\vertexa}$ for each $\vertex\vertexa\in\dualitiesof\net$
and the label of a vertex $\vertex$ in $\netp$ defined as the label of $\bifib(\vertex)$ in $\graphof\prop$,
where $\bifib(\vertex)\tightin\verticesof\dualizinggraph$
can be viewed as a vertex of $\graphof\prop$ since
$\verticesof\dualizinggraph=\verticesof{\graphof\prop}$
by definition of $\dualizinggraphofsymbol$.
Since $\net$ is a dualizing net, $\netp$ is a fonet: (a) every colour is a pre-dual pair of literals, since $\bifib:\graphpairof{\verticesof\net}{\dualitiesof\net}\to\graphpairof{\verticesof\dualizinggraph}{\dualitiesof\dualizinggraph}$ is an undirected graph homomorphism,
(b) $\netp$ trivially has a dualizer, the empty assignment, since it is propositional, with no existential variables,
and
(c) $\netp$ has no induced bimatching, since the leap graphs $\leapgraphof\netp$ and $\leapgraphof\net$ are equal and
$\net$ has no induced bimatching.
We claim that $\bifib:\netp\to\graphof\prop$ is a skew bifibration.
Since $\bifib:\net\to\dualizinggraph$ is a skew fibration,
$\graphpairof{\verticesof\netp}{\edgesof\netp}=\graphpairof{\verticesof\net}{\edgesof\net}$, and
$\graphpairof{\verticesof{\graphof\prop}}{\edgesof{\graphof\prop}}=\graphpairof{\verticesof\dualizinggraph}{\edgesof\dualizinggraph}$,
we know $\bifib:\netp\to\graphof\prop$ is a skew fibration.
Because the label of $\vertex$ in $\netp$ is
that of $\bifib(\vertex)$ in $\graphof\prop$, $\bifib:\netp\to\graphof\prop$ preserves labels. Since there are no binders, existentials are preserved trivially and $\bifib:\bindinggraphof\netp\to\bindinggraphof{\graphof\prop}$ is trivially a directed graph fibration.
Thus $\bifib:\netp\to\graphof\prop$ is a skew bifibration, hence a combinatorial proof (since $\netp$ is a fonet).
By Theorem\,\ref{thm:soundness}, $\prop$ is valid.
\end{proof}
\begin{lemma}[Propositional homogeneous completeness]\label{lem:prop-completeness}
Every valid simple proposition has a homogeneous combinatorial proof.
\end{lemma}
\begin{proof}
Let $\prop$ be a valid simple proposition.
By Theorem\,\ref{thm:completeness} there exists a (standard) combinatorial proof $\bifib:\net\to\graphof\prop$.
Let $\netp$ be the dualizing graph obtained from $\net$ by replacing each link (colour) $\setof{\vertex,\vertexa}$ by a duality $\vertex\vertexa\tightin\dualitiesof\netp$.
Since, by definition of a linked fograph, every literal is in exactly one link,
$\graphpairof{\verticesof\netp}{\dualitiesof\netp}$ is a matching,
and since $\net$ is a simple propositional fonet, $\netp$ has no induced bimatching; thus $\netp$ is a dualizing net.
Let
$\dualizinggraph=\dualizinggraphofprop\prop$.
By \reflem{lem:prop-factorization},
$\dualizinggraphofprop\prop=\dualizinggraphof{\graphof\prop}$, thus
$\graphpairof{\verticesof\dualizinggraph}{\edgesof\dualizinggraph}=\graphpairof{\verticesof{\graphof\prop}}{\edgesof{\graphof\prop}}$.
We claim that $\bifib:\netp\to\dualizinggraph$ is a homogeneous combinatorial proof, \ie, (1) $\bifib:\netp\to\dualizinggraph$ is a skew fibration and (2) \mbox{$\bifib:\graphpairof{\verticesof\netp}{\dualitiesof\netp}\to\graphpairof{\verticesof\dualizinggraph}{\dualitiesof\dualizinggraph}$} is a graph homomorphism.
By definition, (1) holds if \mbox{$\bifib:\graphpairof{\verticesof\netp}{\edgesof\netp}\to\graphpairof{\verticesof\dualizinggraph}{\edgesof\dualizinggraph}$} is a skew fibration, which is true because $\bifib$ is a skew bifibration,
$\graphpairof{\verticesof\netp}{\edgesof\netp}=\graphpairof{\verticesof\net}{\edgesof\net}$, and
$\graphpairof{\verticesof\dualizinggraph}{\edgesof\dualizinggraph}=\graphpairof{\verticesof{\graphof\prop}}{\edgesof{\graphof\prop}}$.
For (2), suppose $\vertex\vertexa\tightin\dualitiesof\netp$. Since $\net$ is a fonet, it has a dualizer, so the labels of $\vertex$ and $\vertexa$ are dual, say, $p$ and $\pp$, respectively.
Because $\bifib$ preserves labels, $\bifib(\vertex)$ and $\bifib(\vertexa)$ are labelled $p$ and $\pp$, thus $\bifib(\vertex)\mkern2mu\bifib(\vertexa)\in\dualitiesof\dualizinggraph$, and (2) holds.
\end{proof}
\begin{proofof}{Theorem\,\ref{thm:prop-soundness-completeness} (Propositional homogeneous soundness and completeness)}\\
Lemmas\,\ref{lem:prop-soundness} and \ref{lem:prop-completeness}.
\end{proofof}

\subsection{Monadic homogeneous soundness and completeness proof}\label{sec:proof-of-monadic-soundness-completeness}

In this section we prove the monadic homogeneous soundness and completeness theorem, Theorem\,\ref{thm:monadic-soundness-completeness}.
The proof of completeness is similar to that of the propositional case,
\reflem{lem:prop-completeness}:
transform a standard first-order combinatorial proof of a monadic formula into a homogeneous combinatorial proof.
The proof of soundness is more subtle. In the propositional case,
\reflem{lem:prop-soundness},
we transformed a homogeneous combinatorial proof directly into a standard one, with the same vertices in both source and target.
The monadic case involves quotienting indistinguishable vertices in the source monet.

\subsubsection{Factorization through closed monadic fographs}

A fograph is \defn{closed} if it contains no free variables,
and \defn{monadic} if its predicate symbols are unary and it has no function symbols or logical constants ($1$ or $0$).
\begin{samepage}\begin{definition}\label{def:mograph-of-fograph}
The \defn{mograph} $\mographof\fograph$ of a closed monadic fograph $\fograph$ is the mograph $\mograph$ with
\begin{itemize}
\item $\verticesof\mograph=\verticesof\fograph$,
\item $\edgesof\mograph=\edgesof\fograph$,
\item $\vertex\vertexa\tightin\dualitiesof\mograph$ if and only if
$\vertex$ and $\vertexa$ are literals whose predicate symbols are dual, and
\item $\diredge\vertex\vertexa\tightin\bindingsof\mograph$ if and only if
$\vertex$ binds $\vertexa$.
\end{itemize}
\end{definition}\end{samepage}
For example, the closed monadic fograph $\fograph$ in Fig.\,\ref{fig:mograph} (centre)
has the mograph $\mographof\fograph$ to its right.
\begin{lemma}\label{lem:mograph-of-fograph-well-defined}
$\mographof\fograph$ is a well-defined mograph for every closed monadic fograph $\fograph$.
\end{lemma}
\begin{proof}
The underlying cograph $\graphpairof{\verticesof\mograph}{\edgesof\mograph}$ is inherited directly from $\fograph$.
By reasoning as in the proof of \reflem{lem:dualizing-graph-of-prop-well-defined}, $\graphpairof{\verticesof\fograph}{\dualitiesof\dualizinggraph}$ is $P_4$- and $\cthree$-free.
It remains to show
(a) every target of a binding in $\bindingsof\mograph$ is in no other binding,
(b) no binder is in a duality,
(c) the scope of every binder $\binder$ is non-empty, and
(d) $\diredge\binder\literal\tightin\bindingedgesof\mograph$ only if $\literal$ is in the scope of $\binder$.

(a) Since $\fograph$ is monadic, every literal label contains exactly one variable, hence is bound by at most one binder in $\fograph$. By definition of $\bindinggraphof\fograph$, no literal binds any other vertex, thus every literal target of a binding is in no other binding.

(b) Dualities are defined as pairs of literals in $\fograph$, which become literals in $\mograph$ since $\fograph$ is closed.
(Every literal in $\fograph$ is bound by a binder in $\fograph$, so becomes a literal in $\mograph$.)

(d) By definition of fograph binding, $\literal$ is bound by a binder $\binder$ only if $\literal$ is in the scope of $\binder$.
\end{proof}
\begin{lemma}\label{lem:monadic-factorization}
The function $\mographofformulasymbol$ from closed monadic formulas to mographs (Def.\,\ref{def:mograph-of-formula}) factorizes through closed monadic fographs:
$\mographofformula\monadicformula=\mographof{\graphof\monadicformula}$ for every closed monadic formula $\monadicformula$.
\end{lemma}
\begin{proof}
A routine induction on the structure of $\monadicformula$.\todo{todo}
\end{proof}

\subsubsection{Collapsing indistinguishable vacuous universal binders}

Given an equivalence relation $\sim$ on a set $\vertices$ write $\equivclassof\vertex\sim$
for the $\sim$-equivalence class $\setst{\vertexa\tightin\vertices}{\vertexa\sim\vertex}$ and
$\quotientof\vertices\sim$ for the set of $\sim$-equivalence classes $\setst{\equivclassof\vertex\sim}{\vertex\tightin\vertices}$.
For a set $\edges$ of edges on $\vertices$ define $\quotientof\edges\sim$ as the set $\setst{\equivclassof\vertex\sim\equivclassof\vertexa\sim}{\vertex,\vertexa\in\edges}$ of edges on $\quotientof\vertices\sim$.
Given a mograph $\mograph$ and an equivalence relation $\sim$ on $\verticesof\mograph$ define the \defn{quotient} mograph $\quotientgraphof\mograph\sim$ by
$\verticesof{\quotientgraphof\mograph{\mkern3mu\sim}}=\quotientof{\verticesof\mograph}\sim$,
$\edgesof{\quotientgraphof\mograph{\mkern3mu\sim}}=\quotientof{\edgesof\mograph}\sim$,
$\dualitiesof{\quotientgraphof\mograph{\mkern3mu\sim}}=\quotientof{\dualitiesof\mograph}\sim$,
and
$\bindingsof{\quotientgraphof\mograph{\mkern3mu\sim}}=\quotientof{\bindingsof\mograph}\sim$.

A binder in a mograph is \defn{vacuous} if it binds no literal.
Let $\bifib:\net\to\mograph$ be a skew bifibration of mographs.
Vacuous universal binders $\binder$ and $\bindera$ in $\net$ are \defn{indistinguishable}
if their images and neighbourhoods are equal, \ie,
$\bifib(\binder)=\fib(\bindera)$ and $\neighbourhoodof\binder=\neighbourhoodof\bindera$.
Define $\indist$ as the equivalence relation on $\verticesof\net$ generated by indistinguishability,
and the \defn{collapse} $\collapseof\bifib:\quotientgraphof\net\indist\to\mograph$
as the canonical function on the quotient,
\ie, $\collapseof\bifib(\equivclassof\binder\indist)=\fib(\vertex)$, a well-defined function since $\binder\indist\bindera$ implies $\bifib(\binder)\tighteq\bifib(\bindera)$.
\begin{lemma}\label{lem:collapse}
Let $\mograph$ be a mograph and $\net$ a monet.
If
$\bifib:\net\to\mograph$ is a homogeneous combinatorial proof then its collapse
$\collapseof\fib:\quotientgraphof\monet\indist\to\grapha$ is a homogeneous combinatorial proof.
\end{lemma}
\begin{proof}
$\quotientgraphof\monet\indist$ is a monet because if
$W\tightsubseteq\verticesof{\quotientgraphof\monet{\mkern2mu\indist}}$ induces a bimatching in
$\quotientgraphof\monet\indist$ then it induces a bimatching in $\monet$:
since indistinguishable vertices are vacuous binders, they cannot be in both a leap and an edge of
$\quotientgraphof{\edgesof\monet}\indist$, so cannot occur in $W$.
The function
$\collapseof\bifib:\graphpairof{\quotientof{\verticesof\monet}\indist}{\quotientof{\edgesof\monet}\indist}\to\graphpairofgraph\mograph$
is a skew fibration because
$\bifib:\graphpairofgraph\monet\to\graphpairofgraph\mograph$ is a skew fibration and indistinguishable vertices have the same image and neighbourhood, and
\mbox{$\collapseof\bifib:\graphpairof{\quotientof{\verticesof\monet}\indist}{\quotientof{\dualities\monet}\indist}\to\dualizinggraphofgraph\mograph$}
is a homomorphism because
\mbox{$\bifib:\graphpairofgraph\monet\to\dualizinggraphofgraph\mograph$}
is a homomorphism and no binder is in a duality edge.
Finally,
$\collapseof\bifib:\graphpairof{\quotientof{\verticesof\monet}\indist}{\quotientof{\bindingsof\monet}\indist}\to\bindinggraphofgraph\mograph$
is a fibration because
\mbox{$\bifib:\graphpairofgraph\monet\to\graphpairofgraph\mograph$} is a
fibration and indistinguishable binders are vacuous, therefore absent from bindings.
\end{proof}

\subsubsection{Monadic fonets without dualizers}

Monets were defined (\S\ref{sec:monets}) without need for dualizers, in terms of the binder equivalence relation $\brel\mograph{\mkern-4mu}$.
In this section we take an analogous approach with monadic fonets (\S\ref{sec:fonets}).

Let \defn{rmf} abbreviate \emph{rectified monadic fograph}.
\begin{definition}
Let $\cover$ be a linked rmf.
\defn{Variable equivalence}
$\vrel\cover$ is the equivalence relation on binders generated by
$x\vrel\cover y$
for each link $\setof{\singletonpx,\singletonppy}$ in $\cover$.
\end{definition}
In the above definition $p$ is any predicate symbol (necessarily unary, since $\cover$ is monadic).

A \defn{conflict} in $\cover$ is a pair $\setof{x,y}$ of distinct non-existential variables $x$ and $y$ such that $x\vrel\cover y$.
\begin{definition}\label{def:linked-rmf-consistent}
A linked rmf is \defn{consistent} if
it has no conflict.
\end{definition}
\begin{lemma}\label{lem:consistent-dualizable}
A linked rmf has a dualizer if and only if it is consistent.
\end{lemma}
\begin{proof}
Let $\cover$ be the linked rmf.

Suppose $\cover$ has a dualizer. By \reflem{lem:mgd} $\cover$ has a most general dualizer $\dualizer$.
Thus for every colour $\monadiclink$ we have $(\px)\dualizer$ dual to $(\ppy)\dualizer$.
(Recall that $\atom\dualizer$ denotes the result of substituting $\dualizer(x)$ for $x$ in $\atom$.)
For a contradiction, suppose $\setof{z_1,z_2}$ were a conflict in $\cover$, \ie,
$z_1\mkern-2mu\vrel\cover\mkern-2mu z_2$ for
non-existential variables $z_1\tightneq z_2$.
Since
$z_1\vrel\cover z_2$
we have variables
$x_1,\ldots, x_n$ for $n\tightge 1$ with $x_1=z_1$, $x_n=z_2$, and for $1\tightle i\tightlt n$ there exists a link $\setof{\singleton{p_ix_i}, \singleton{\pp_ix_{i+1}}}$.
Since $\dualizer$ is a dualizer we have $(p_ix_i)\dualizer$ dual to $(\pp_ix_{i+1})\dualizer$, so $x_i\dualizer=x_{i+1}\dualizer$.
Thus $x_1\dualizer=x_n\dualizer$ so $z_1\dualizer=z_2\dualizer$.
Since $z_1$ and $z_2$ are non-existential, we have $z_1\dualizer=z_1$ and $z_2\dualizer=z_2$, hence $z_1=z_2$, contradicting $z_1\tightneq z_2$.

Conversely, suppose $\cover$ is consistent.
Let $e_1,\ldots,e_n$ be the equivalence classes of $\vrel\cover$.
Define $y_i$ as the unique non-existential variable in $e_i$, if it exists (unique since $\cover$ is consistent), and otherwise define $y_i$ as a fresh variable, where \emph{fresh} means not in $\cover$ and distinct from $y_j$ for $1\tightle j\tightlt i$.
Given an existential variable $x$, define $\dualizer(x)=y_i$ if $e_i$ is the equivalence class containing $x_i$.

We must show that for every link $\monadiclink$ in $\cover$ we have $(px)\dualizer$ dual to $(\ppy)\dualizer$.
Thus it remains to show that $x\dualizer=y\dualizer$.
Since $x$ and $y$ are in the same link, they are in the same equivalence class $e_i$ (for some $i$).
We consider three cases.
\begin{enumerate}
\item Both $x$ and $y$ are existential. Since $x$ and $y$ are in $e_i$, we have $\dualizer(x)=\dualizer(y)=y_i$.
\item Both $x$ and $y$ are non-existential. Therefore $x\dualizer=x$ and $y\dualizer=y$, so we require $x\tighteq y$.
This holds because $x\tightneq y$ would imply that $\setof{x,y}$ is a conflict, contradicting the consistency of $\cover$.
\item Exactly one of $x$ and $y$ is existential, say $x$. Since $y$ is non-existential, $y\dualizer=y$, and $y$ is the unique $y_i$ non-existential variable in $e_i$. Since $x$ is also in $e_i$, we have $\dualizer(x)=y_i$.
\end{enumerate}\vspace{-3.3ex}\end{proof}
\begin{lemma}\label{lem:monadic-deps}
  Let $\singletonx$ and $\singletony$ be binders in a consistent linked rmf $\cover$, with $\singletonx$ existential and $\singletony$ universal.
  The pair $\setof{\singletonx,\singletony}$ is a dependency of $\cover$ if and only if $x\vrel\cover y$.
\end{lemma}
\begin{proof}
  By \reflem{lem:mgd-deps}, the dependencies of $\cover$ are those of a most general dualizer $\dualizer$, so it suffices
  to show that $x\vrel\cover y$ if and only if $\dualizer(x)=y$.
  Since every predicate symbol in $\cover$ is unary, $\vrel\cover$ is the transitive closure of the unification problem $\unirelof\cover$ (see the proof of \reflem{lem:mgd}).
Thus the dualizer $\dualizer$ defined in the proof of \reflem{lem:consistent-dualizable} is most general, and by construction $x\vrel\cover y$ if and only if $\dualizer(x)=y$.
\end{proof}
Note that the above lemmas simplify the definition of (standard, non-homogeneous) monadic combinatorial proof $\bifib:\cover\to\base$:
\begin{itemize}
\item Instead of checking for the existence of a dualizer for (the rectified form of) $\cover$, we merely check that $\cover$ is consistent, via the variable relation $\vrel\cover$, using \reflem{lem:consistent-dualizable}.
\item Instead of building the leap graph $\leapgraphof\cover$ with dependencies via a dualizer,
we read dependencies directly from $\vrel\cover$, using \reflem{lem:monadic-deps}.
\end{itemize}

\subsubsection{The linked mograph of a linked closed monadic fograph}

\begin{definition}\label{def:linked-mograph-of-linked-fograph}
The \defn{linked mograph} $\linkedmographof\cover$ of a linked closed monadic fograph $\cover$ is the linked mograph $\mograph$ with
\begin{itemize}
\item $\verticesof\mograph=\verticesof\cover$,
\item $\edgesof\mograph=\edgesof\cover$,
\item $\vertex\vertexa\tightin\dualitiesof\mograph$
if and only if
$\setof{\vertex,\vertexa}$ is a link
\item $\diredge\vertex\vertexa\tightin\bindingsof\mograph$ if and only if
$\vertex$ binds $\vertexa$.
\end{itemize}
\end{definition}
\begin{lemma}
$\linkedmographof\cover$ is a well-defined linked mograph for every linked closed monadic fograph $\cover$.
\end{lemma}
\begin{proof}
\hspace{-1.48pt}Let $\lambda_1,\ldots,\lambda_n$ be the links of $\cover$ for $\lambda_i=\setof{\literal_i,\literala_i}$.
Choose distinct predicate symbols $p_1,\ldots,p_n$, and define $\coverp$ by replacing the predicate symbols in the labels of $\literal_i$ and $\literala_i$ by $p_i$ and $\pp_i$, respectively.
By construction, $\linkedmographof\coverp=\linkedmographof\cover$, and
since two literals in $\coverp$ are pre-dual in $\coverp$ if and only if they constitute a link,
we have $\linkedmographof\coverp=\mographof\coverp$.
Thus $\linkedmographof\cover=\mographof\coverp$ which is a well-defined mograph by \reflem{lem:mograph-of-fograph-well-defined}.
Since the $p_i$ are distinct, every literal of $\coverp$ is in a unique duality, so $\coverp$ is linked.
\end{proof}
\begin{lemma}\label{lem:fograph-mograph-net}
A linked closed monadic fograph $\cover$ is a fonet if and only if its linked mograph $\linkedmographof\cover$ is a monet.
\end{lemma}
\begin{proof}
Without loss of generality we may assume $\cover$ is rectified.
By \reflem{lem:consistent-dualizable}, $\cover$ has a dualizer if and only if it is consistent in the sense of Def.\,\ref{def:linked-rmf-consistent}, and consistency of $\cover$ coincides with consistency of $\linkedmographof\cover$ (Def.\,\ref{def:mograph-consistent}).
By \reflem{lem:monadic-deps} the dependencies of $\cover$ are those pairs $\setof{x,y}$ of variables with $x$ existential, $y$ universal and $x\brel\cover y$,
which, by definition of $\linkedmographofsymbol$, correspond to pairs $\setof{\binder_x,\binder_y}$ of binders in $\linkedmographof\cover$ with $\binder_x$ and $\binder_y$ the unique binders corresponding to the variables $x$ and $y$, and $\binder_x\brel{\linkedmographof\cover}\binder_y$.
Thus the leap graphs of $\cover$ and $\linkedmographof\cover$ are the same, so $\cover$ has an induced bimatching if and only if $\linkedmographof\cover$ has an induced bimatching.
\end{proof}

\subsubsection{Proof of monadic homogeneous combinatorial soundness}

Recall that, by definition of $\mographofsymbol$, $\verticesof{\mographof\fograph}=\verticesof\fograph$ for every closed monadic fograph $\fograph$.
\begin{lemma}\label{lem:mograph-of-fograph-preserves-literals}
Let $\fograph$ be a closed monadic fograph.
A vertex is a literal in $\fograph$ if and only if it is a literal in the mograph $\mographof\fograph$.
\end{lemma}
\begin{proof}
Immediate from the definition of the binding set $\bindingsof{\mographof\fograph}$ (Def.\,\ref{def:mograph-of-fograph}) and that, by definition,
a vertex is a literal in a mograph if and only if it is the target of a binding.
\end{proof}
\begin{lemma}\label{lem:mograph-of-fograph-preserves-existentials}
Let $\fograph$ be a closed monadic fograph.
A binder is universal in $\fograph$ if and only if it is universal in the mograph $\mographof\fograph$.
\end{lemma}
\begin{proof}
By definition $\graphpairofgraph{\fograph}=\graphpairofgraph{\mographof\fograph}$, and in both cases,
a binder is universal if and only if its scope contains no edge.
\end{proof}
Define the \defn{type}
$\typeof\vertex\graph\in\vertextypes$ of a vertex $\vertex$
in a mograph or fograph $\graph$ as $\literaltype$ if $\vertex$ is a literal, $\universaltype$ if $\vertex$ is a universal binder, and $\existentialtype$ if $\vertex$ is an existential binder.
\begin{lemma}\label{lem:types}
For every closed monadic fograph $\fograph$, $\;\typeof\vertex\fograph=\typeof\vertex{\mographof\fograph}$ for every vertex $\vertex$.
\end{lemma}
\begin{proof}
Lemmas\,\ref{lem:mograph-of-fograph-preserves-literals} and \ref{lem:mograph-of-fograph-preserves-existentials}.
\end{proof}
\begin{lemma}\label{lem:mograph-bifib-preserves-literals}
Every mograph skew bifibration $\bifib:\net\to\mograph$ preserves vertex type, \ie, $\typeof\vertex\net=\typeof{\bifib(\vertex)}\mograph$ for every vertex $\vertex$ in $\net$.
\end{lemma}
\begin{proof}
A vertex is a literal if and only if it is the target of a binding, and
since $\bifib:\bindinggraphofgraph\net\to\bindinggraphofgraph\mograph$ is a fibration,
a vertex $\vertex$ in $\verticesof\net$ is the target of a binding if and only if $\bifib(\vertex)$ is the target of a binding.
Thus $\bifib$ maps literals to literals and binders to binders.
By definition (Def.\,\ref{def:monadic-skew-bifib}) a skew bifibration maps existential binders to existential binders,
so it remains to show that universal binders map to universal binders.
This follows from the proof of \reflem{lem:pres-universals}, which applies in the homogeneous setting because it does not depend on labels.
\end{proof}
\begin{lemma}[Monadic homogeneous soundness]\label{lem:monadic-soundness}
A closed monadic formula is valid if it has a homogeneous combinatorial proof.
\end{lemma}
\begin{proof}
Suppose
$\bifib:\net\to\mographofformula\monadicformula=\mograph$
is a homogeneous combinatorial proof of the monadic formula $\monadicformula$.
Without loss of generality, we may assume $\bifib$ is collapsed, by \reflem{lem:collapse}.
Define $\netp$ as the
coloured labelled cograph with
$\verticesof\netp=\verticesof\net$,
$\edgesof\netp=\edgesof\net$,
a colour $\setof{\vertex,\vertexa}$ for each $\vertex\vertexa\in\dualitiesof\net$,
and
the label of $\vertex$ in $\netp$ defined as
the label of $\bifib(\vertex)$ in $\graphof\monadicformula$,
where $\bifib(\vertex)\in\verticesof{\mographofformula\monadicformula}$ can be viewed as a vertex in
$\verticesof{\graphof\monadicformula}$ since
$\mographofformula\monadicformula=\mographof{\graphof\monadicformula}$ by \reflem{lem:prop-factorization} and, by definition of $\mographofsymbol$ (Def.\,\ref{def:mograph-of-fograph}),
$\verticesof{\mographof\fograph}=\verticesof{\fograph}$ for any closed monadic fograph $\fograph$.

We claim that $\netp$ is a well-defined fograph (Def.\,\ref{def:fograph}, p.\,\pageref{def:fograph}).
By \reflem{lem:types},
$\vertextypein{\mographof\monadicformula}=\vertextypein{\graphof\monadicformula}$
so by \reflem{lem:mograph-bifib-preserves-literals},
$\vertextypein{\netp}=\vertextypein{\net}$ (since the label of $\vertex$ in $\netp$ is that of $\bifib(\vertex)$ in $\graphof\monadicformula$).
Thus $\netp$ has a literal since $\net$ has one (because it is a mograph), so $\netp$ is a logical cograph.
We must show, for all variables $x$, that every $x$-binder $\binder$ is legal,
\ie, the scope of $\binder$ contains
(a) at least one literal and (b) no other $x$-binder.
For (a), the scope $\scopeofin\binder\graph$ of $\binder$ in a fograph or mograph $\graph$ depends only on the underlying cograph $\graphpairofgraph\graph$, so
$\scopeofin\binder\netp=\scopeofin\binder\net$.
Thus $\scopeofin\binder\netp$ has a literal because $\scopeofin\binder\net$ does (since $\net$ is a mograph).
For a contradiction to (b), Suppose $\bindera\tightneq\binder$ were an $x$-binder in $\scopeofin\binder\netp$.
Let $\treep$ be the cotree $\cotreeof\netp$ and let $\parentof\binder$ be the parent of $\binder$ in $\treep$.
If $\binder$ is existential, then $\binder\bindera\tightin\edgesof\netp$ (since, by \reflem{lem:scope-parent}, all distinct vertices in the scope of an existential binder are in an edge, since $\parentof\binder$ is a $\graphjoin$ node in $\treep$),
contradicting $\bifib(\binder)=\bifib(\bindera)$ (which holds because, without loss of generality, $\monadicformula$ is rectified, so there is a unique $x$-binder in $\graphof\monadicformula$).
Otherwise $\binder$ is universal.
Since $\bindera\in\scopeofin\binder\netp$, by \reflem{lem:scope-parent}
we have
$\binder\parentabovein\treep\bindera$,
\ie,
$\binder\childin\treep\parentof\binder\abovein\treep\bindera$, with $\parentof\binder$ a $\graphunion$-node, since $\binder$ is universal.
Because $\bifib$ is collapsed and $\bifib(\binder)=\bifib(\bindera)$,
we cannot have
$\parentof\binder\parentin\treep\bindera$
(otherwise $\binder$ and $\bindera$ would be indistinguishable, contradicting $\bifib$ being collapsed),
thus
$\binder\childin\treep\parentof\binder\parentin\treep\bindera\jmeet\vertex$ for some vertex $\vertex$.
Since $\bindera\vertex\tightin\edgesof\netp$ and $\bifib$ is a graph homomorphism we have
$\bifib(\bindera)\mkern1mu\bifib(\vertex)\tightin\edgesof{\graphof\monadicformula}$,
so
$\bifib(\binder)\mkern1mu\bifib(\vertex)\tightin\edgesof{\graphof\monadicformula}$
(because $\bifib(\binder)=\bifib(\bindera)$).
Since $\bifib$ is a skew fibration, there exists $\vertexa\tightin\verticesof\netp$ such that
$\vertexa\binder\in\edgesof\netp$ and $\bifib(\vertexa)\mkern1mu\bifib(\vertex)\notin\edgesof{\graphof\monadicformula}$.
Because $\vertexa\binder\in\edgesof\netp$, the meet $\binder\meet\vertexa$ is a \mbox{$\graphjoin$-node},
\ie, $\binder\meet\vertexa=\binder\jmeet\vertexa$,
and since the parent $\parentof\binder$ of $\binder$ is a \mbox{$\graphunion$-node},
we have $\binder\jmeet\vertexa\abovein\treep\parentof\binder$,
hence
$\vertexa\abovein\treep\binder\jmeet\vertexa\abovein\treep\vertex$.
Therefore $\vertexa\vertex\tightin\edgesof\netp$, so $\bifib(\vertexa)\mkern1mu\bifib(\vertex)\in\edgesof{\graphof\monadicformula}$ a contradiction.
Thus we have proved that $\netp$ is a well-defined fograph.
Since every literal label in $\netp$ comes from $\graphof\monadicformula$, $\netp$ is monadic, and since $\bifib$ is a directed graph fibration
$\bindinggraphofgraph\net\to\bindinggraphofgraph\mograph$, $\netp$ is closed.

By construction, $\netp=\linkedmographof\net$ (Def.\,\ref{def:linked-mograph-of-linked-fograph}),
so by \reflem{lem:fograph-mograph-net}, $\netp$ is a fonet.
Since $\bifib:\net\to\mograph$ is a skew bifibration of mographs,
$\bifib:\netp\to\graphof\monadicformula$ is a skew bifibration of fographs,
hence a (standard) combinatorial proof,
so $\formula$ is valid by Theorem\,\ref{thm:soundness}.
\end{proof}\todo{Give non-collapsed example}
The crux of the soundness proof above is to transform a collapsed monadic homogeneous combinatorial proof into a standard combinatorial proof.
The following example shows why collapse occurs before this transformation.
A monadic homogeneous combinatorial proof of the closed monadic formula $\collapsedegformula$ is shown below-left.
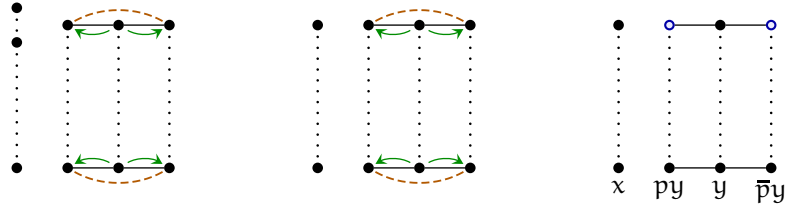
\begin{pic}{-.3}{2.4}
\rput(-4,0){\uncollapsedeg}
\rput(0,0){\collapsedeg}
\rput(4,0){\labelledcollapsedeg}
\end{pic}
Its collapse, also a monadic homogeneous combinatorial proof (by \reflem{lem:collapse}), is shown above-centre.
Above-right is the standard combinatorial proof constructed from the collapse in the soundness proof above.
Observe that, were we to attempt to construct a standard combinatorial proof directly from the uncollapsed form,
it would have two source vertices above $\singletonx$ in the target,
each implicitly labelled $x$ (implicit since we are drawing the skeleton),
so the source would have a (universal) $x$-binder in the scope of another $x$-binder and therefore fail to be a well-defined fograph.
Collapse is directly related to the deletion of select universal binders in the interpretation of the $\cname$ rule as an operation in \S\ref{sec:interp-rules}.
\begin{lemma}[Monadic homogeneous completeness]\label{lem:monadic-completeness}
Every valid closed monadic formula has a homogeneous combinatorial proof.
\end{lemma}
\begin{proof}
Let $\monadicformula$ be a valid closed monadic formula.
By Theorem\,\ref{thm:completeness} there exists a (standard) combinatorial proof $\bifib:\net\to\graphof\monadicformula$.
Let $\netp$ be the linked mograph obtained from $\net$ with $\verticesof\netp=\verticesof\net$, $\edgesof\netp=\edgesof\net$,
$\vertex\vertexa\tightin\dualitiesof\netp$ if and only if and $\setof{\vertex,\vertexa}$ is a link (colour) in $\net$,
and
$\bindingsof\netp=\edgesof{\bindinggraphof\net}$ (\ie, $\vertex\vertexa\in\bindingsof\netp$ if and only if $\vertex$ binds $\vertexa$ in $\net$).
Since there are no logical constants in $\monadicformula$ there are none in $\net$ (by label-preservation of $\bifib$),
so every literal in $\net$ is in exactly one link.
Thus every literal of $\netp$ is in a unique duality, no binder of $\netp$ is in a duality,
and $\netp$ has no induced bimatching because $\net$ is a fonet;
thus $\netp$ is a monet.
Let $\mograph=\mographof{\graphof\monadicformula}=\mographofformula\monadicformula$.

We claim that $\bifib:\netp\to\mograph$ is a homogeneous combinatorial proof, \ie,
(1) $\bifib$ preserves existential binders,
(2) $\bifib:\netp\to\mograph$ is a skew fibration,
(3) $\bifib:\graphpairof{\verticesof\netp}{\dualitiesof\netp}\to\graphpairof{\verticesof\mograph}{\dualitiesof\mograph}$ is an undirected graph homomorphism,
and
(4) $\bifib:\graphpairof{\verticesof\netp}{\bindingsof\netp}\to\graphpairof{\verticesof\mograph}{\bindingsof\mograph}$ is a directed graph fibration.
(1) holds because $\bifib:\net\to\graphof\monadicformula$ preserves existential binders, and by construction the existential binders of $\netp$ and $\net$ coincide, as do those of $\graphof\monadicformula$ and $\mograph$.
By definition (2) holds if $\bifib:\graphpairof{\verticesof\netp}{\edgesof\netp}\to\graphpairof{\verticesof\mograph}{\edgesof\mograph}$ is a skew fibration, which is true because $\bifib$ is a skew bifibration,
$\graphpairof{\verticesof\netp}{\edgesof\netp}=\graphpairof{\verticesof\net}{\edgesof\net}$, and
$\graphpairof{\verticesof\mograph}{\edgesof\mograph}=\graphpairof{\verticesof{\graphof\monadicformula}}{\edgesof{\graphof\monadicformula}}$.
For (3), suppose $\vertex\vertexa\tightin\dualitiesof\netp$. Since $\net$ is a fonet, it has a dualizer, so the labels of $\vertex$ and $\vertexa$ are dual, say, $p$ and $\pp$, respectively.
Because $\bifib$ preserves labels, $\bifib(\vertex)$ and $\bifib(\vertexa)$ are labelled $p$ and $\pp$, thus $\bifib(\vertex)\mkern2mu\bifib(\vertexa)\in\dualitiesof\mograph$, and (3) holds.
(4) holds because $\bifib:\net\to\graphof\monadicformula$ is a skew bifibration, thus $\bifib:\bindinggraphof\net\to\bindinggraphof{\graphof\monadicformula}$ is a directed graph fibration, and by construction $\graphpairof{\verticesof\netp}{\bindingsof\netp}=\bindinggraphof\net$ and
$\graphpairof{\verticesof\mograph}{\bindingsof{\graphof\monadicformula}}=\bindinggraphof{\graphof\monadicformula}$.
\end{proof}

\begin{proof}[Proof of Theorem\,\ref{thm:monadic-soundness-completeness} (Monadic homogeneous soundness and completeness)]
\mbox{}\\
Lemmas~\ref{lem:monadic-soundness} and \ref{lem:monadic-completeness}.
\end{proof}

\section{Homogeneous surjections}\label{sec:surjections}

Recall from \S\ref{sec:soundness} that $\graphofsymbol$ is a surjection from rectified formulas onto rectified fographs (Lem.\,\ref{lem:graph-surj}),
and that $\xgraphofsymbol$ is a surjection from clear formulas onto fographs (Lem.\,\ref{lem:xgraph-surj}).
This section exhibits similar surjections onto duality graphs and mographs.
\begin{lemma}\label{lem:surj-prop-fographs-to-dualizing-graphs}
$\dualizinggraphofsymbol$ is a surjection from simple propositional fographs onto dualizing graphs.
\end{lemma}
\begin{proof}
  Let $\dualizinggraph$ be a dualizing graph.
We construct a fograph $\fograph$ such that $\dualizinggraphof\fograph=\dualizinggraph$.
Define $\verticesof\fograph=\verticesof\dualizinggraph$ and $\edgesof\fograph=\edgesof\dualizinggraph$, with a nullary predicate symbol label on each vertex defined as follows.
Since $\graphpairof{\verticesof\dualizinggraph}{\dualitiesof\dualizinggraph}$ is $P_4$-free and $\cthree$-free, it is a disjoint union of complete bipartite graphs\footnote{Recall that a complete bipartite graph is one of the form $\graph\graphjoin\grapha$ for edgeless graphs $\graph$ and $\grapha$.} $\cbg_1,\ldots,\cbg_n$.
Choose distinct nullary predicate symbols $p_1,\ldots,p_n$ such that $\pp_i\tightneq p_j$ ($1\tightle i,j\tightle n$).
If $\cbg_i$ has no edges, it has a single vertex $v_i$; assign $p_i$ as the label of $v_i$.
Otherwise, $\cbg_i=\cbg_i'\graphjoin\cbg_i''$ for $\cbg_i'$ and $\cbg_i''$ without edges.
Assign the label $p_i$ to every vertex in $\cbg_i'$ and the label $\pp_i$ to every vertex in $\cbg_i''$.
The graph $\fograph$ is a non-empty cograph with vertices labelled by nullary predicate symbols, hence $\fograph$ is a simple propositional fograph.
By construction, $\dualizinggraphof\fograph=\dualizinggraph$.
\end{proof}
\begin{lemma}\label{lem:prop-surj}
The function $\dualizinggraphofpropsymbol$ from simple propositions to dualizing graphs (Def.\,\ref{def:dualizing-graph-of-prop}) is a surjection.
\end{lemma}
\begin{proof}
By \reflem{lem:graph-surj} $\graphofsymbol$ is a surjection from (rectified) formulas onto fographs.
The restriction of $\graphofsymbol$ to simple propositions is a surjection onto simple propositional fographs.
Since $\dualizinggraphofpropsymbol=\dualizinggraphofsymbol\circ\graphofsymbol$ by \reflem{lem:prop-factorization}, and $\dualizinggraphofsymbol$ is a surjection by \reflem{lem:surj-prop-fographs-to-dualizing-graphs}, $\dualizinggraphofpropsymbol$ is a surjection.
\end{proof}

\begin{lemma}\label{lem:surj-closed-monadic-fographs-to-mographs}
$\mographofsymbol$ is a surjection from closed monadic fographs onto mographs.
\end{lemma}
\begin{proof}
Let $\mograph$ be a mograph.
We will construct a closed monadic fograph $\fograph$ with $\mographof\fograph=\mograph$.
Define $\verticesof\fograph=\verticesof\mograph$ and $\edgesof\fograph=\edgesof\mograph$, and define the predicate symbol in the label of each vertex of $\verticesof\fograph$ exactly as in the proof of \reflem{lem:surj-prop-fographs-to-dualizing-graphs}, only this time we shall make each such predicate symbol $p$ unary rather than nullary by adding a variable after $p$. For each binder $\binder$ in $\mograph$, choose a distinct variable $x_\binder$, set the label of $\binder$ to $x_\binder$, and for every literal $\literal$ with $\diredge{\binder}{\literal}\in\bindingsof\mograph$, add the variable $x_\binder$ to the label of $\literal$ as the argument of the predicate symbol already assigned to $\literal$.
Since every binder in $\mograph$ has non-empty scope, every binder in $\fograph$ has non-empty scope.
By construction every literal label is a unary predicate symbol followed by a variable, so $\fograph$ is monadic.
Because every variable $x_\binder$ is distinct for each binder $\binder$, no literal in $\fograph$ can be bound by two binders in $\fograph$.
Thus $\fograph$ is a rectified monadic fograph.
Since, by definition of a literal in a mograph, every literal in $\mograph$ is the target of a binding in $\bindingsof\mograph$, every literal in $\fograph$ is bound, so $\fograph$ is closed.
By construction, $\mographof\fograph=\mograph$.
\end{proof}

\begin{lemma}\label{lem:surj-closed-monadic-formulas-to-mographs}
The function $\mographofformulasymbol$ from closed monadic formulas to mographs (Def.\,\ref{def:mograph-of-formula}) is a surjection.
\end{lemma}
\begin{proof}
By \reflem{lem:graph-surj} $\graphofsymbol$ is a surjection from (rectified) formulas onto fographs.
The restriction of $\graphofsymbol$ to closed monadic formulas is a surjection onto closed monadic fographs.
Since $\mographofformulasymbol=\mographofsymbol\circ\graphofsymbol$ by \reflem{lem:monadic-factorization}, and
$\mographofsymbol$ is a surjection by \reflem{lem:surj-closed-monadic-fographs-to-mographs}, $\mographofformulasymbol$ is a surjection.
\end{proof}

\section{Polynomial-time verification}\label{sec:ptime}

In this section we show that a combinatorial proof can be verified in polynomial time.
Thus combinatorial proofs constitute a formal \emph{proof system} \cite{CR79}.

The \defn{size} of a graph $\graph$ is the sum of the number of vertices in $\graph$ and the number of edges in $\graph$.
\begin{lemma}\label{lem:deps-ptime}
The dependencies of a linked rectified fograph $\cover$ can be constructed in time polynomial in the size of $\cover$.
\end{lemma}
\begin{proof}
Let $x_1,\ldots,x_n$ be the existential variables in $\cover$.
The main unification algorithm of \cite{MM76} provides in linear time an assignment $\assignopen\assign {x_1} {u_1},\ldots,\assign {x_n} {u_n}\assignclose$ with $x_i$ not in $u_j$ for $i\tightle j$, such that the most general unifier $\sigma$
is $\assignopen\assign{x_1}{t_1},\ldots,\assign{x_n}{t_n}\assignclose$ for
$t_i=u_i
\assignopen
\assign{x_{i+1}}{u_{i+1}}
\assignclose
\ldots
\assignopen
\assign{x_n}{u_n}
\assignclose$
(the sequential composition of $n\tightminus i$ one-variable substitutions applied to $u_i$).
Let $\setof{y_{i1},\ldots,y_{im_i}}$ be the set of variables occurring in $u_i$,
and define $u'_i$ as $f_iy_{i1}\ldots y_{im_i}$ for a fresh $m_i$-ary function symbol $f_i$.
The assignment $\sigma'=\assignopen\assign{x_1}{t'_1},\ldots,\assign{x_n}{t'_n}\assignclose$ for
$t'_i=
u'_i
\assignopen
\assign{x_{i+1}}{u'_{i+1}}
\assignclose
\ldots
\assignopen
\assign{x_n}{u'_n}
\assignclose$
has the same dependencies as $\sigma$ but can be constructed in polynomial time since each $x_j$ appears at most once in each $u'_i$.
\end{proof}
The above proof is essentially the first part of the proof of Theorem~3 in \cite{Hug18}.
\begin{lemma}\label{lem:fonet-ptime}
  The correctness of fonet can be verified in time polynomial in its size.
\end{lemma}
\begin{proof}
Let $\net$ be a fonet of size $n$.
By \reflem{lem:deps-ptime} we can construct all dependencies of $\net$ in polynomial time, hence the leap graph $\leapgraphof\net$ in polynomial time.
By \reflem{lem:fonet-constructible} every fonet is constructible from axioms by fusion and quantification.
Since there can be at most $n$ fusions and/or quantifications, it suffices to show that each step in the inductive decomposition of a fonet in the proof of \reflem{lem:fonet-constructible} can be performed in polynomial time.
In the first case the proof of \reflem{lem:fonet-constructible}, $\net$ has no edges (which can be determined in polynomial time), and to confirm that $\net$ is a union of axioms takes polynomial time.
In the second case, $\net$ is universal, and the universal binder can be found and deleted in polynomial time, by inspecting each vertex of $\net$ in succession.

In the final case, $\net$ is not universal and has at least one edge, and we seek to decompose $\net$ as a fusion or existential quantification via \reflem{lem:split-fusion-existential}.
Henceforth we follow the proof of \reflem{lem:split-fusion-existential} closely.
The graph $\megagraph$ in the proof of \reflem{lem:split-fusion-existential} can be constructed in polynomial time from the cotree,
which can built in polynomial time \cite{CLS81}.
The bridge $\graph_m\graph_{m+1}$ can be located in polynomial time (by iterating through the edges of $\megagraph$),
and $\cover_1$ and $\cover_2$ can be determined in polynomial time by traversing edges.
The underlying fograph $\fograph$ of
$\net$
is
$\cover_1\graphunion(\fograph_m\graphjoin\fograph_{m+1})\graphunion\cover_2$.
Depending on whether both $\fograph_m$ and $\fograph_{m+1}$ both contain literals, the proof of \reflem{lem:split-fusion-existential} now provides either $\net$ as a fusion of $\cover_1\graphunion\fograph_m$ and $\fograph_{m+1}\graphunion\cover_2$, and we recurse with each half of the fusion, or $\net=\singletonx\graphunion\netp$, and we delete the existential binder $\singletonx$ and recurse with $\netp$.
\end{proof}
Define the \defn{size} of a combinatorial proof $\bifib:\net\to\fograph$ as the sum of the size of $\net$ and the size of $\fograph$.
\begin{theorem}
  The correctness of a combinatorial proof can be verified in time polynomial in its size.
\end{theorem}
\begin{proof}
  Let $\bifib:\net\to\fograph$ be a combinatorial proof.
  By \reflem{lem:fonet-ptime} the fonet $\net$ can be verified in polynomial time.
  Verifying that $\bifib$ is a skew bifibration is polynomial time because the skew fibration and directed graph fibration conditions apply to pairs of vertices, one in $\net$ and one in $\fograph$, seeking the existence of a vertex in $\net$, which can be found be iterating through each vertex of $\net$ in turn.
\end{proof}

\section{Cut combinatorial proofs}

Just as sequent calculus proofs may include cuts \cite{Gen35}, combinatorial proofs can be extended with cuts.
Define an \defn{$n$-cut combinatorial proof} of a formula $\formula$ as a combinatorial proof of
\mbox{$\formula\vee(\formulaa_1\tightwedge\neg{\formulaa_1})\vee\ldots\vee(\formulaa_n\tightwedge\neg{\formulaa_n})$} for (arbitrary) formulas $\formulaa_1,\ldots,\formulaa_n$.
Each formula $\formulaa_i\tightwedge\neg\formulaa_i$ is a \defn{cut}.
A \defn{cut combinatorial proof} is an $n$-cut combinatorial proof for some $n\tightge 0$; if $n=0$ the combinatorial proof is \defn{cut-free}.\footnote{We can define a cut combinatorial proof of a sequent similarly.}
\begin{theorem}
A formula is valid if and only if it has a cut combinatorial proof.
\end{theorem}
\begin{proof}
Since
$\formula\vee(\formulaa_1\tightwedge\neg{\formulaa_1})\vee\ldots\vee(\formulaa_n\tightwedge\neg{\formulaa_n})$ is valid if and only if $\formula$ is valid,
the result follows from Theorem\,\ref{thm:soundness-completeness}.
\end{proof}

\section{Semi-combinatorial proofs}\label{sec:semicps}

Using the surjections
$\graphofsymbol$ (Def.\,\ref{def:graph}) from (implicitly rectified) formulas onto rectified fographs
and
$\xgraphofsymbol$ (Def.\,\ref{def:xgraph}) from clear formulas onto fographs,
given a combinatorial proof $\bifib:\cover\to\fograph$, such as the one whose skeleton is drawn below-left (copied from the Introduction),
by choosing
a rectified formula $\formula$ with $\graphof\formula=\fograph$
and
a clear formula $\formulaa$ with $\xgraphof\formulaa$ equal to the underlying uncoloured fograph of $\cover$,
we can render $\bifib$ in the form below-centre.
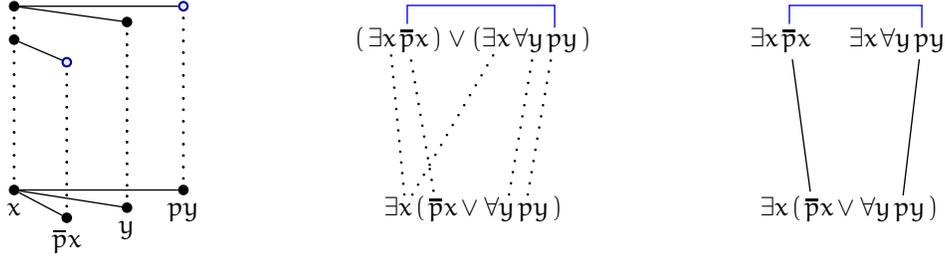
\begin{center}\begin{pic}{-.9}{2.9}
\rput(-5,0){\drinkerfibcoloured}
\rput(0,-.25){
\rput(0,0){\drinkerSemicombinatorialVerbose}
\rput(5,0){\drinkerSemicombinatorial}
}
\end{pic}\end{center}
We have drawn the bifibration between the quantifier variables and predicate symbols of the formulas $\formulaa=\drinkerSemicombinatorialSourceVerboseInline$ and
$\formula=\veedrinkerformula$ corresponding to the vertices of $\cover$ and $\fograph$,
and replaced the link (coloured pair of vertices) on $\cover$ with a three-segment edge between the dual predicate symbols $\pp$ and $p$ in $\formulaa$, in the style of proof nets for linear logic \cite{Gir87}.
Above-right we have simplified the presentation further by removing redundant bifibration edges between quantifier variables (since they can be left implicit due to label-preservation, \eg, both occurrences of the existential quantifier variable $x$ in the source map to the (unique) existential quantifier variable $x$ in the target), and we have drawn non-dotted edges.
We have also replaced the formula $\drinkerSemicombinatorialSourceVerboseInline$ with the corresponding sequent
$\drinkerSemicombinatorialSourceInline$, and suppressed the comma of the sequent.
We call this presentation of a combinatorial proof a \defn{semi-combinatorial proof}, a first-order generalization of the propositional case in \cite{Hug06i}.\footnote{The source sequent, with links, can be viewed as generalization of a unification net \cite{Hug18}. See \S\ref{sec:related} for details.}

\section{Conclusion and related work}\label{sec:conclusion}\label{sec:related}

This paper reformulated classical first-order logic with combinatorial rather than syntactic proofs (\S\ref{sec:fographs}--\S\ref{sec:cps}),
extending the propositional case of \cite{Hug06} to quantifiers.
The proof of soundness (\S\ref{sec:soundness}) was more intricate than
that of the propositional case \cite[\S5]{Hug06}.
In the logical-constant-free propositional, monadic and S5-modal special cases, labels
can be removed from a combinatorial proof, and colouring from the source,
for a homogeneous form (\S\ref{sec:prop-cps}--\S\ref{sec:modal-cps}).

Propositional combinatorial proofs are related to sequent calculus \cite{Gen35} in \cite{Hug06i} and \cite{Car10},
and to other syntactic systems (including resolution and analytic tableaux) in \cite{Str17} and \cite{AS18}.
Skew fibrations are decomposed as propositional structural maps (composites of contraction and weakening maps) in \cite{Hug06i} and \cite{Str07}.
Combinatorial proofs may provide an avenue towards tackling Hilbert's 24th problem \cite{TW02,Thi03,Hug06i,Str19h24}.

Combinatorial proofs for non-classical logics are being pursued actively. For example,
combinatorial proofs for propositional intuitionistic logic are presented in \cite{HHS19ipws}.
A potential topic of future research is first-order intuitionistic combinatorial proofs.
Cut elimination procedures for propositional cut combinatorial proofs are presented in \cite{Hug06i} and \cite{Str17}.
Natural open questions include the extension of propositional intuitionistic combinatorial proofs to first-order,
and cut elimination procedures for first-order combinatorial proofs (classical and intuitionistic).

The function $\graphofsymbol$ from first-order formulas to fographs (Def.\,\ref{def:graph})
is a first-order extension of the propositional translation \textsl{G} of \cite[\S3]{Hug06}.
The latter is well-known in graph theory, as the function from a (prime-free) modular decomposition tree \cite{Gal67} or cotree \cite{Ler71,CLS81} to a cograph, and is employed in logic and category theory. For example, \cite{Gir87} uses \textsl{G}\/ with
$\wedge\tighteq\with$ and $\vee\tighteq\oplus$ in linear logic, \cite{Hu99} uses \textsl{G} with $\wedge$ and $\vee$ as product and coproduct for free bicompletion (and emphasizes the $\pfour$-freeness of the image), and \cite{Ret03} uses \textsl{G} with $\wedge\tighteq\otimes$ and $\vee\tighteq\parr$ in linear logic.
That cographs are exactly the $P_4$-free graphs is proved in \cite{Sum73}.

Links between pre-dual literals in fonets, which become dual only after applying a dualizer or unifier,
are akin to the
first-order \emph{connections} or \emph{matings} employed in automated theorem proving \cite{Bib81, And81}.
Bibel in \cite[p.\,4]{Bib81} proposed \emph{link} as an alternative name for a connection,
and we have adopted that terminology.
Since combinatorial proofs can be verified in polynomial time (\S\ref{sec:ptime}), they constitute a formal \emph{proof system} \cite{CR79},
in contrast to the connection and mating methods.
The roots of first-order connections/matings
go back to Herbrand \cite{Her30}, Quine \cite{Qui55}, Robinson \cite{Rob65} and Prawitz \cite{Pra70}, amongst others.
Propositional links between dual literals can be found in
\cite{Dav71,Bib74,And76},
and sets of such propositional links form a category \cite{LS05} via path composition.
The pairing of dual propositional occurrences can be found in the study of other forms of syntax, such as
closed categories \cite{KM71} (see also \cite{EK66}),
contraction-free predicate calculus \cite{KW84} and linear logic \cite{Gir87}.

A fonet can be viewed as a graph-theoretic abstraction and generalization of
a unification net \cite{Hug18},
which in turn abstracts proof nets for first-order multiplicative linear logic \cite{Gir87,Gir91}.
The sense of generalization is that fographs admit the multiplicative mix rule $\frac{\sequent\hspace{1ex}\sequenta}{\sequentcom\sequenta}$, interpreted as the fusion operation with empty portions (\ie, disjoint union).
The relationship with a unification net is made clearer when rendering a combinatorial proof in semi-combinatorial form,  as in \S\ref{sec:semicps}. (For example, in the right-most example at the start of \S\ref{sec:semicps}, the source is exactly a unification net in the sense of \cite{Hug18}.)
Unification nets are also available for first-order additive linear logic \cite{HHS19ALL}.
Upon forgetting vertex labels, propositional fonets correspond to the \emph{nicely coloured cographs} of \cite{Hug06},
and nicely coloured cographs without singleton colours are in bijection with the \emph{even-length alternating elementary acyclic R\&B cographs} of \cite{Ret03}.
An additional constraint on fonets can be applied to reject the mix rule, and retain soundness and completeness (\cf\ the alternating elementary connectedness condition of \cite{Ret03}).

Abstract representations of first-order quantifiers with explicit witnesses are in \cite{Hei10} (extending expansion
trees \cite{Mil84}) and \cite{McK10} (for classical logic) and \cite{HHS19ALL} (for additive linear logic).
Composition of witnesses is analysed in \cite{Mim11} and \cite{ACHW18}.

Proof nets \cite{Gir87} were extended to propositional classical
logic in \cite{Gir91} (developed in detail in \cite{Rob03}).
The paper \cite{McK13} fixes issues of redundancy due to contraction and weakening nodes
and relates classical propositional proof nets to propositional combinatorial proofs \cite{Hug06,Hug06i}.

Peirce \cite[vol.\,4:2]{Pei} provides an early graphical representation of propositional formulas.

\appendix

\section{Fograph examples}\label{app:fograph-examples}

\figGraphExamples\figMoreGraphExamples{}Figures~\ref{fig:graph-examples} and \ref{fig:more-graph-examples}
show a progression of examples of fographs $\graphof{\formula}$ for various formulas $\formula$.\clearpage

\small
\bibliographystyle{alpha}
\bibliography{../bib/main}

\begin{thebibliography}{ACHW18}

\bibitem[ACHW18]{ACHW18}
A.~Alcolei, P.~Clairambault, J.~M.~E. Hyland, and G.~Winskel.
\newblock The true concurrency of herbrand's theorem.
\newblock In {\em CSL'18}, 2018.

\bibitem[And76]{And76}
P.~B. Andrews.
\newblock Refutations by matings.
\newblock {\em IEEE Transactions on Computers}, 25(8), 1976.

\bibitem[And81]{And81}
P.~B. Andrews.
\newblock Theorem proving via general matings.
\newblock {\em J.\ ACM}, 28, 1981.

\bibitem[AS18]{AS18}
M.~Acclavio and L.~Stra{\ss}burger.
\newblock From syntactic proofs to combinatorial proofs.
\newblock In {\em Int.\ Joint Conf.\ on Automated Reasoning}. Springer, 2018.

\bibitem[Bib74]{Bib74}
W.~Bibel.
\newblock An approach to a systematic theorem proving procedure in first-order
  logic.
\newblock {\em Computing}, 12(1), 1974.

\bibitem[Bib81]{Bib81}
W.~Bibel.
\newblock On matrices with connections.
\newblock {\em J.\ ACM}, 28(4), 1981.

\bibitem[BK92]{BK92}
G.~Bellin and J.~Ketonen.
\newblock A decision procedure revisited: Notes on direct logic, linear logic
  and its implementation.
\newblock {\em Theor.\ Comp.\ Sci.}, 95:115--142, 1992.

\bibitem[Car10]{Car10}
A.~Carbone.
\newblock A new mapping between combinatorial proofs and sequent calculus
  proofs read out from logical flow graphs.
\newblock {\em Inf.\ Comput.}, 208(5), 2010.

\bibitem[CLS81]{CLS81}
D.~G. Corneil, H.~Lerchs, and L.~K. {Stewart-Burlingham}.
\newblock Complement reducible graphs.
\newblock {\em Disc.\ Appl.\ Math.}, 1981.

\bibitem[CR79]{CR79}
S.~A. Cook and R.~A. Reckhow.
\newblock The relative efficiency of propositional proof systems.
\newblock {\em J.\ Symb.\ Logic}, 44:36--50, 1979.

\bibitem[Dav71]{Dav71}
G.~V. Davydov.
\newblock The synthesis of the resolution method with the inverse method.
\newblock {\em Zapiski Nauchnykh Seminarov Lomi.}, 20:24--35, 1971.
\newblock Translation in J.\ Sov.\ Math.\ 1(1) 1973 12-–18.

\bibitem[EK66]{EK66}
S.~Eilenberg and G.~M. Kelly.
\newblock A generalization of the functorial calculus.
\newblock {\em J.\ Algebra}, 3:366--375, 1966.

\bibitem[Fre79]{Fre}
G.~Frege.
\newblock {\em Begriffsschrift, eine der arithmetischen nachgebildete
  {F}ormelsprache des reinen {D}enkens}.
\newblock Halle, 1879.

\bibitem[Gal67]{Gal67}
T.~Gallai.
\newblock Transitiv orientierbare graphen.
\newblock {\em Acta Math.\ Acad.\ Sci.\ Hungar}, 18:25--66, 1967.

\bibitem[Gen35]{Gen35}
G.~Gentzen.
\newblock Untersuchungen {\"u}ber das logische {S}chlie{\ss}en. {II}.
\newblock {\em Mathematische Zeitschrift}, 39:405--431, 1935.

\bibitem[Gir87]{Gir87}
J.-Y. Girard.
\newblock Linear logic.
\newblock {\em Theor.\ Comp.\ Sci.}, 50:1--102, 1987.

\bibitem[Gir89]{Gir89}
J.-Y. Girard.
\newblock Towards a geometry of interaction.
\newblock In {\em Categories in Computer Science and Logic}, volume~92 of {\em
  Contemporary Mathematics}, pages 69--108, 1989.
\newblock Proc.\ of June '87 meeting in Boulder, Colorado.

\bibitem[Gir91]{Gir91}
J.-Y. Girard.
\newblock A new constructive logic: Classical logic.
\newblock {\em Mathematical Structures in Comp.\ Sci.}, 1:255--296, 1991.

\bibitem[Gra66]{Gra66}
J.~W. Gray.
\newblock Fibred and cofibred categories.
\newblock In {\em Proc.\ Conf.\ on Categorical Algebra '65}, pages 21--83.
  Springer, 1966.

\bibitem[Gro60]{Gro60}
A.~Grothendieck.
\newblock Technique de descente et th\'eor\`emes d'existence en g\'eom\'etrie
  alg\'ebrique. {I}. {G}\'en\'eralit\'es. {D}escente par morphismes
  fid\`element plats.
\newblock In {\em S\'eminaire Bourbaki: ann\'ees 1958/59--1959/60, expos\'es
  169-204}. Soci\'et\'e math\'ematique de France, 1960.

\bibitem[Hei10]{Hei10}
W.~Heijltjes.
\newblock Classical proof forestry.
\newblock {\em Annals of Pure and Applied Logic}, 161(11), 2010.

\bibitem[Her30]{Her30}
J.~Herbrand.
\newblock {\em Recherches sur la th\'eorie de la d\'emonstration}.
\newblock PhD thesis, Sorbonne, Paris, 1930.

\bibitem[HHS19a]{HHS19ipws}
W.~Heijltjes, D.~J.~D. Hughes, and L.~Stra{\ss}burger.
\newblock Intuitionistic proofs without syntax.
\newblock In {\em Proc.\ LICS'19}, 2019.

\bibitem[HHS19b]{HHS19ALL}
W.~Heijltjes, D.~J.~D. Hughes, and L.~Stra{\ss}burger.
\newblock Proof nets for first-order additive linear logic.
\newblock In {\em Proc.\ FSCD'19}, 2019.

\bibitem[Hu99]{Hu99}
H. Hu.
\newblock Contractible Coherence Spaces and Maximal Maps.
\newblock {\em Electr. Notes Theor. Comput. Sci.}, 20:309--319, 1999.

\bibitem[Hug06a]{Hug06}
D.~J.~D. Hughes.
\newblock Proofs without syntax.
\newblock {\em Annals of {M}athematics}, 143:1065--1076, 2006.

\bibitem[Hug06b]{Hug06i}
D.~J.~D. Hughes.
\newblock Towards {H}ilbert's 24$^\text{th}$ {P}roblem: {C}ombinatorial {P}roof
  {I}nvariants.
\newblock In {\em Proc.\ {WOLLiC}'06}, volume 165 of {\em Lec.\ Notes in Comp.\
  Sci.}, 2006.

\bibitem[Hug18]{Hug18}
D.~J.~D. Hughes.
\newblock Unification nets: canonical proof net quantifiers.
\newblock In {\em Proc.\ LICS '18}, 2018.

\bibitem[Joh87]{Joh87}
P.~T. Johnstone.
\newblock {\em Note on Logic and Set Theory}.
\newblock Cambridge University Press, 1987.

\bibitem[KM71]{KM71}
G.~M. Kelly and S.~Mac{ }Lane.
\newblock Coherence in closed categories.
\newblock {\em J.\ Pure Appl.\ Algebra}, 1:97--140, 1971.

\bibitem[Kot59]{Kot59}
A.~Kotzig.
\newblock On the theory of finite graphs with a linear factor ii.
\newblock {\em Mat.-Fyz. Casopis. Slovensk. Akad. Vied}, 9(3), 1959.

\bibitem[KW84]{KW84}
J.~Ketonen and R.~Wehyrauch.
\newblock A decidable fragment of predicate calculus.
\newblock {\em Theor.\ Comp.\ Sci.}, 32:297--307, 1984.
\newblock Error in proof of main theorem corrected in \cite{BK92}.

\bibitem[Ler81]{Ler71}
H.~Lerchs.
\newblock On cliques and kernels.
\newblock Tech.\ report, U.\ Toronto, 1981.

\bibitem[LP86]{LP86}
L.~Lov\'asz and M.~D. Plummer.
\newblock {\em Matching Theory}.
\newblock North-Holland, 1986.

\bibitem[LS05]{LS05}
F.~Lamarche and L.~Stra{\ss}burger.
\newblock Naming proofs in classical propositional logic.
\newblock In {\em Proc.\ Typed Lambda Calculus '05}, Lec.\ Notes in Comp.\
  Sci., 2005.

\bibitem[McK10a]{McK10}
R.~McKinley.
\newblock Expansion nets: Proof nets for for propositional classical logic.
\newblock In {\em LPAR 17}, volume 6397 of {\em LNCS}, 2010.

\bibitem[McK10b]{McK10h}
R.~McKinley.
\newblock Proof nets for herbrand's theorem.
\newblock {\em ACM Trans.\ Comp.\ Logic}, 14, 2010.

\bibitem[McK13]{McK13}
R.~McKinley.
\newblock Canonical proof nets for classical logic.
\newblock {\em Annals of Pure and Applied Logic}, 164(6):702--732, 2013.

\bibitem[Mil84]{Mil84}
D.~A. Miller.
\newblock Expansion tree proofs and their conversion to natural deduction
  proofs.
\newblock {\em Lec.\ Notes in Comp.\ Sci.}, 170, 1984.

\bibitem[Mim11]{Mim11}
S.~Mimram.
\newblock The structure of first-order causality.
\newblock {\em Mathematical Structures in Comp.\ Sci.}, 21, 2011.

\bibitem[Min92]{Min92}
G.~Mints.
\newblock {\em A short introduction to modal logic}.
\newblock CSLI, 1992.

\bibitem[MM76]{MM76}
A.~Martelli and U.~Montanari.
\newblock Unification in linear time and space: a structured presentation.
\newblock Technical report, U.\ Pisa, 1976.

\bibitem[Pei33]{Pei}
C.~S. Peirce.
\newblock Description of a notation for the logic of relatives, resulting from
  an amplification of the conceptions of {B}oole's calculus of logic.
\newblock In {\em Collected Papers of Charles Sanders Peirce. {III}. {E}xact
  Logic}. Harvard University Press, 1933.

\bibitem[Pra70]{Pra70}
D.~Prawitz.
\newblock A proof procedure with matrix reduction.
\newblock {\em Lecture Notes in Mathematics}, 125, 1970.

\bibitem[Qui55]{Qui55}
W.~V. Quine.
\newblock A proof procedure for quantification theory.
\newblock {\em J.\ Symbolic Logic}, 20(2), 1955.

\bibitem[Ret03]{Ret03}
C.~Retor\'e.
\newblock Handsome proof-nets: perfect matchings and cographs.
\newblock {\em Theor.\ Comp.\ Sci.}, 2003.

\bibitem[Rob65]{Rob65}
J.~A. Robinson.
\newblock A machine-oriented logic based on the resolution principle.
\newblock {\em J.\ ACM}, 12(1), 1965.

\bibitem[Rob03]{Rob03}
E.~Robinson.
\newblock Proof nets for classical logic.
\newblock {\em J.\ Logic and Computation}, 13(5):777--797, 2003.

\bibitem[Str07]{Str07}
L.~Stra\ss{}burger.
\newblock A characterisation of medial as rewriting rule.
\newblock In {\em Term Rewriting and Applications, RTA07}, volume 4533. Lec.\
  Notes in Comp.\ Sci., 2007.

\bibitem[Str17]{Str17}
L.~Stra{\ss}burger.
\newblock {Combinatorial Flows and Their Normalisation}.
\newblock In {\em Proc.\ FSCD 2017}, volume~84 of {\em Leibniz Int.\ Proc.\
  Informat.}, 2017.

\bibitem[Str19]{Str19h24}
L.~Stra{\ss}burger.
\newblock The problem of proof identity, and why computer scientists should
  care about hilbert's 24th problem.
\newblock {\em Phil.\ Trans.\ Royal Soc.\ A}, 377(2140), 2019.

\bibitem[Sum73]{Sum73}
D.~P. Sumner.
\newblock Graphs indecomposable with respect to the \textit{x}-join.
\newblock {\em Discr.\ Math.}, 6, 1973.

\bibitem[Thi03]{Thi03}
R.~Thiele.
\newblock Hilbert's twenty-fourth problem.
\newblock {\em The American Mathematical Monthly}, 110(1), 2003.

\bibitem[TS96]{TS96}
A.~S. Troelstra and H.~Schwichtenberg.
\newblock {\em Basic Proof Theory}.
\newblock Cambridge Tracts in Theoretical Computer Science, Cambridge
  University Press, Cambrige, U.K., 1996.

\bibitem[TW02]{TW02}
R.~Thiele and L.~Wos.
\newblock Hilbert's twenty-fourth problem.
\newblock {\em Journal of Automated Reasoning}, 29(1):67--89, 2002.

\bibitem[Whi78]{Whi78}
G.W Whitehead.
\newblock {\em Elements of Homotopy Theory}.
\newblock Springer-Verlag, 1978.

\end{thebibliography}

\end{document}